\documentclass[a4paper,11pt]{article}
\usepackage{geometry}
\usepackage{cite}
\usepackage{graphicx}
\usepackage{amsmath,amssymb,amsthm}
\usepackage{enumerate}
\usepackage{hyperref}
\usepackage[percent]{overpic} 
\usepackage{enumitem}

\usepackage{color}
\newtheorem{theo}{Theorem}[section]

\newtheorem{cor}[theo]{Corollary}
\newtheorem{prop}[theo]{Proposition}
\newtheorem{rmk}[theo]{Remark}
\newcommand{\eps}{\varepsilon}

\newcommand{\R}{\mathbb{R}}

\renewcommand{\a}{\alpha}
\renewcommand{\L}{\Lambda}

\newcommand{\mc}{\mathcal}

\newcommand{\ka}{\kappa(y)}
\newcommand{\f}{\partial_v f(y,0)}
\newcommand{\g}{\partial_u g(y,0)}
\newcommand{\D}{\operatorname{{\bf D}}}
\newcommand{\Hess}{\operatorname{{\bf H}}}

\newcommand{\pd}[2]{\frac{\partial #1}{\partial #2}}

\allowdisplaybreaks

\newcommand{\bo}[1]{{\bf #1}}

\newcommand{\leps}{\lambda_\varepsilon}
\newcommand{\BRR}{\mathrm{BRR}}

\numberwithin{equation}{section}
\numberwithin{figure}{section}
\numberwithin{table}{section}
\begin{document}
\title{Propagation for KPP bulk-surface systems in a general cylindrical domain}
\author{Beniamin Bogosel\footnote{CMAP, Ecole Polytechnique, UMR CNRS 7641, 91128 Palaiseau, France}, Thomas Giletti\footnote{IECL, University of Lorraine, UMR 7502, B.P. 70239, 54506 Vandoeuvre-l\`{e}s-Nancy Cedex, France}, Andrea Tellini \footnote{Universidad Polit\'{e}cnica de Madrid, ETSIDI, Departamento de Matem\'{a}tica Aplicada a la Ingenier\'{i}a Industrial, Ronda de Valencia 3, 28012 Madrid, Spain}}
\maketitle

\abstract{{In this paper, we investigate propagation phenomena for KPP bulk-surface systems in a cylindrical domain with general section and heterogeneous coefficients. As for the scalar KPP equation, we show that the asymptotic spreading speed of solutions can be computed in terms of the principal eigenvalues of a family of self-adjoint elliptic operators.
		
		Using this characterization, we analyze the dependence of the spreading speed on various parameters, including diffusion rates and the size and shape of the section of the domain. In particular, we provide new theoretical results on several asymptotic regimes like small and high diffusion rates and sections with small and large sizes. These results generalize earlier ones which were available in the radial homogeneous case.
		
		Finally, we numerically investigate the issue of shape optimization of the spreading speed. By computing  its shape derivative, we observe, in the case of homogeneous coefficients, that a disk either maximizes or minimizes the speed, depending on the parameters of the problem, both with or without constraints. We also show the results of numerical shape optimization with non homogeneous coefficients, when the disk is no longer an optimizer.}

\section{Introduction}
\label{section1}

In this work we consider the following system of coupled reaction-diffusion equations
\begin{equation}\label{eq:evol}
\left\{
\begin{array}{ll}
\partial_t v = d \Delta v + f(y,v), & \quad \mbox{ for } t >0, \ (x,y) \in \R \times  \omega, \vspace{3pt}\\
d \, \partial_n v  =   \kappa (y) ( \mu u - \nu v ), & \quad \mbox{ for } t >0, \ (x,y) \in \R \times \partial \omega  ,\vspace{3pt}\\
\partial_t u = D \Delta u + g(y,u) + \kappa (y) (\nu v - \mu u), & \quad \mbox{ for } t >0, \ (x,y) \in \R \times \partial \omega.\vspace{3pt}\\
\end{array}
\right.
\end{equation}
where $\omega$ is an open, bounded, connected set of $\R^N$, $N\geq 1$, and its boundary $\partial \omega$ is assumed to be smooth, in the sense that each connected component of $\partial\omega$ is a smooth hypersurface in $\R^N$. The operator $\Delta$ denotes the Laplacian with respect to the $(x,y)$ variables, and, for simplicity, we still use the same symbol for the Laplace-Beltrami operator on $\R\times\partial\omega$, with the convention that it reduces to $\partial_{xx}$ when $N=1$. The diffusion coefficients $d$ and~$D$ are assumed to be positive constants, while the reaction terms $f$ and $g$ are, respectively, functions of class $C^{1,r}\left(\omega \times [0,+\infty)\right)$ and $C^{1,r}\left(\partial\omega \times [0,+\infty)\right)$, with $0<r<1$, that satisfy the following KPP-type properties:
\begin{equation}\label{ass:KPP}
\left. \begin{array}{c}
\displaystyle f(\cdot , 0 ) \equiv 0 \equiv g (\cdot,0), \vspace{4pt} \\
\displaystyle v \mapsto \frac{f(\cdot, v)}{v} \ \mbox{ and } \ u \mapsto \frac{g (\cdot, u )}{u} \ \ \mbox{ are, respectively, decreasing and nonincreasing,} \vspace{4pt}\\
\displaystyle \exists M>0, \ \forall u,v \geq M , \ \ f (y,v) \leq 0 \ \  \forall y \in \omega \quad \text{and} \quad  g(y,u) \leq 0  \ \ \forall y \in \partial \omega . 
\end{array}
\right.
\end{equation}
In the second equation of \eqref{eq:evol}, $n$ denotes the exterior normal direction to $\R\times\omega$ (observe that its first component is $0$), and, for the coupling term $\kappa(y)\left(\mu u-\nu v\right)$, we assume that $\mu$ and $\nu$ are positive constants, and $\kappa\in C^{1,r}\left(\partial\omega\right)$ is a nonnegative function, not identically equal to~$0$.

Finally, problem \eqref{eq:evol} is complemented with an initial datum $(u_0,v_0)$ which, unless differently specified, will be assumed to be continuous, non-negative, not identically equal to $(0,0)$ and with compact support.

System \eqref{eq:evol} has a biological interpretation which can be better understood in the two dimensional case, i.e., for $N=1$, when $\omega$ is simply a bounded, open interval. In such a case, $u$ and $v$ represent the densities of two parts of a single species which move randomly, the latter in the horizontal strip $\R\times\omega$ (which may be referred to as the \emph{field}) and the former, with a possibly different diffusion rate, on the two lines at the boundary that play the role of two \emph{roads}. Also the reproduction and mortality rates, as well as the interactions between the individuals, described by $f$ and $g$, can be different for the two densities. These aspects, together with the possible dependence of $f$ and $g$ on the transversal variable $y$, represent a first element of heterogeneity in the system. 

A second one is given by the coupling term $\kappa(y)\left(\mu u-\nu v\right)$, which is present both in the differential equation on the boundary and as an inhomogeneous Robin boundary condition for $v$. In biological terms, it represents an exchange between the two densities: a portion $\kappa\mu u$ passes from the boundary to the interior, while a portion $\kappa \nu v$ from the interior to the boundary. In this sense, the exchange is symmetric, but it can vary from spot to spot, as the function $\kappa$ depends on $y$. We point out that the total mass of the population $\int_{\mathbb{R} \times \partial \omega} u + \int_{\mathbb{R} \times \omega} v$ is constant in time when $f = g \equiv 0$, i.e., when no reproduction, mortality nor interactions are taken into account, which is consistent with this interpretation.

Finally, a third ingredient of heterogeneity is indirectly given by the geometry of the domain $\omega$, which will be one of the points that we will address in this work. Nonetheless, we point out that \eqref{eq:evol} is invariant by translations in the $x$-variable, which is the direction in which we will study the propagation of solutions, and this will play an important role in the techniques that we use for our analysis.

We will focus on the qualitative behaviour of the solution of \eqref{eq:evol}, which is initially concentrated in some region of the domain, for large times. Mainly, we will address the questions of whether an invasion of the domain occurs, i.e., the solution converges to a positive steady state, or, instead, the species is driven to extinction. In the former case, in addition, we will study how fast the invasion takes place, through the concept of \emph{asymptotic speed of propagation} in the direction of the axis of the cylinder, which will be made precise in Theorem \ref{th:spreading}.

The specificity of \eqref{eq:evol} is that the two equations are posed in domains with different spatial dimensions (precisely a $(N+1)$-dimensional and a $N$-dimensional manifold), and they are coupled through a flux condition of Robin type on the boundary. Systems of this kind have been introduced by H. Berestycki, J.-M. Roquejoffre and L. Rossi in \cite{BRR_influence_of_a_line, BRR_plus, BRR_shape_of_expansion} as a tool to analyze, from the mathematical point of view, the effect of a transport network on the propagation of biological species. Indeed, in many situations, transport networks accelerate the propagation of species or diseases. Some specific examples are the tiger mosquito (\emph{aedes albopictus}), which is transported by cars along roads, the Asian wasp (\emph{vespa velutina}), which propagates quicker along rivers, and the black plague in the Middle Ages, that spread faster through trade routes, like the silk road. 

With this biological background in mind, the choice of the names \emph{field} and \emph{road(s)}, that we have used above to refer, respectively, to the generalized cylinder $\R\times\omega$ and its boundary $\R\times\partial\omega$, becomes clear. We point that the higher dimensional situation, which may also be referred to as a \emph{volume-surface} or \emph{bulk-surface} reaction-diffusion system, appears in biological models for cell polarization or cell division models~\cite{FLT,MS2018,RR2014}. We will diffusely use this nomenclature in the rest of this work.\\

Let us point out other recent studies on the propagation phenomena in the framework of road-field systems. In the aforementioned papers~\cite{BRR_influence_of_a_line, BRR_plus, BRR_shape_of_expansion}, where a half-plane is considered for the field, and all parameters are homogeneous, H. Berestycki, J.-M. Roquejoffre and L. Rossi have proved the existence of a  threshold for the value of the diffusion coefficient $D$ below which the spreading speed is equal to $2 \sqrt{d f'(0)}$, the spreading speed in the Euclidean space with no road. Thus, for small diffusion rates on the road, this has no effect on the propagation. Instead, if $D$ is above this threshold, the spreading speed is strictly larger than in the case without the road and therefore the road makes the propagation faster. The same authors have also shown the existence of travelling wave solutions in~\cite{BRR_road_fields_travelling_waves}.

One of the present authors has shown that such a diffusion threshold still exist in the case of a standard cylinder, i.e., $\omega$ is a ball, again under the additional assumption that all the equation parameters are homogeneous, and also that there is no reaction on the surface~\cite{T16,RTV}. We send the interested reader to~\cite{T19} for a a review and comparison of those systems.

The arguments in all the aforementioned works relied on a semi-explicit construction of solutions of the linearized system around $(0,0)$, which allow one to construct sub- and supersolutions for the nonlinear problem that give, respectively, lower and upper bounds for the spreading speed. The main goal of this paper is to present a more robust approach, consisting in the characterization of the spreading speed through the principal eigenvalue of a family of elliptic operators. This allows us to recover some results of~\cite{T16,RTV} and to tackle the heterogeneous framework described above, including more general cylindrical domains. Such a principal eigenvalue approach is well-known in the case of a scalar KPP reaction-diffusion equation; let us mention the classical work~\cite{BerestyckiNirenberg-TWcylinders}. In the context of road-field equations, it has been used in~\cite{GMZ} for a planar and spatially periodic case. In parallel to the present work, a notion of a generalized principal eigenvalue has also been introduced in~\cite{BDR19, BDRpreprint}; yet, the goal of these latter works was to study not the spreading properties but the persistence or extinction of the solution in heterogeneous domains.

Finally, we briefly mention that a variety of other road-field systems have been considered in the literature: with fractional diffusion on the road~\cite{BCRR}, ignition type nonlinearity~\cite{Dietrich1,Dietrich2,DiRo}, nonlocal exchange terms~\cite{Pauthier}, and a SIR epidemiological model with an additional compartment of infected that travel on a road with fast diffusion~\cite{BRR_SIR_preprint}. We also mention the forthcoming work~\cite{BRTpreprint} where the large time behaviour of solutions of both road-field systems and systems of parabolic equations on adjacent domains is studied.

\begin{rmk}
We highlight that our study applies to some other situations without difficulty. For instance, one may consider a surface that consists of two or more connected components, or even a mixed situation where some connected components of $\R \times \partial \omega$ involve a surface diffusion equation, while a more classical boundary condition (e.g., Dirichlet, Neumann) is imposed on others. The key assumption that has to be maintained is the invariance by translation in the first variable, which is the direction of the propagation. Nonetheless, we only restrict ourselves to the above described framework, in order to simplify the presentation.\end{rmk}

\paragraph{Plan of the paper and summary of the main results.} Section~\ref{section2} is mainly concerned with the characterization of the spreading speed of the solutions using the principal eigenvalues of a family of self-adjoint elliptic operators. In particular, we present a variational formula for the spreading speed, which relies on a Rayleigh formula for these principal eigenvalues, and which is the basis of our analysis. We refer to Theorem~\ref{th:spreading} and Proposition~\ref{prop:main} for the main results of Section~\ref{section2}. An immediate consequence of this approach will be the monotonicity of the speed with respect to the exchange parameter~$\kappa$. Moreover, we establish some fundamental properties of the principal eigenvalue, that will be the key points for several arguments in this work, both in the theoretical and the numerical parts.

Sections~\ref{section:c^*_D} and~\ref{section:c^*_R} are respectively devoted to the theoretical study of how the spreading speed depends on diffusion rates and the size of the domain. More precisely, Theorem~\ref{th:D_monotonicity} deals with the asymptotic behaviour of the  speed in the slow and high diffusion regimes, and it shows that the surface may either slow down or accelerate the propagation. However, we also show in Section~\ref{sec:D_counter} that, in the non radial case, the spreading speed may not be monotonic with respect to diffusion rates and, hence, there may no longer be a diffusion threshold for acceleration by the surface; this is a new phenomenon which did not appear in earlier studies of road-field systems. Theorems~\ref{th:R0} and~\ref{th:Rinfty} are respectively concerned  with the limits for small and large domain sections. In the former, we prove that the speed converges to the one of an averaged version of the KPP equation on the surface; in the latter, to the one of the road-field problem considered in~\cite{BRR_influence_of_a_line, BRR_plus} with a half-plane or, in higher dimension, in the analogous situation of a half-space.

Finally, in Section~\ref{sec:numeric} we propose a numerical scheme to compute the spreading speed, which we then use in Section~\ref{sec:shape}, to address, in dimension~$N +1 =3$, the issue of finding the shape of the section that optimizes the speed. We consider both the minimization and maximization problem, in either unconstrained domains or under a perimeter or area constraint. Our numerical simulations suggest that in the homogeneous case where functions $f$, $g$ and $\kappa$ are spatially constant, then, depending on the parameters, each problem is either ill-posed, or it admits the disk as an optimal shape, at least among simply connected domains. However, we will also show that other shapes appear in the case of heterogeneous coefficients.

\section{Preliminaries: characterization of the spreading speed}
\label{section2}
Since the coupling in \eqref{eq:evol} is non-standard, we begin with a general result on the existence and uniqueness of solutions for the parabolic system, as well as a result on its stationary states.

\begin{theo}
	\label{th:2.1}
	For any initial datum $(u_0, v_0)$, system \eqref{eq:evol} admits a solution, which is globally defined in time. Moreover, the problem satisfies a comparison principle; thus, in particular, the solution is unique. 
	
	Finally, there exists at most one positive steady state, which we denote by $(U^* (y), V^* (y))$ when it exists.
\end{theo}

 We do not provide the proof of this result, since it can be obtained by following the same arguments used in~\cite{BRR_influence_of_a_line,BRR_plus,BDRpreprint}. We just mention that the existence part holds true under very general assumptions, while the uniqueness of the positive steady state is a direct consequence of the KPP hypothesis~\eqref{ass:KPP}. Uniqueness actually also holds true if one assumes, in place of the second relation of~\eqref{ass:KPP} that $v \mapsto \frac{f(\cdot ,v)}{v}$ is nonincreasing and $u \mapsto \frac{g (\cdot, u) }{u}$ is decreasing.

Moreover, we point out that the steady state $(U^*,V^*)$ only depends on the cross variable~$y$, which follows from its uniqueness together with the invariance of~\eqref{eq:evol} by translations in the $x$-direction.
 
 We now introduce the following principal eigenvalue problem which will be the main tool of our subsequent analysis:
 \begin{equation}\label{eq:eigen}
 \left\{
 \begin{array}{rcll}
 -D \Delta U - \left(D\alpha^2 + \partial_u g(y,0) - \kappa (y) \mu\right) U - \kappa (y) \nu V & = & \Lambda (\alpha ) U, & \quad \mbox{ on } \partial \omega,\vspace{3pt}\\
 -d \Delta V - \left(d\alpha^2 + \partial_v f(y,0) \right) V  & = & \Lambda (\alpha) V, & \quad \mbox{ in }  \omega,\vspace{3pt}\\
 d \, \partial_n V -\kappa(y) \left( \mu U - \nu V \right) & =& 0 , & \quad \mbox{ on } \partial \omega, \vspace{3pt}\\
U >0 \mbox{ on } \partial \omega \mbox{ and } V >0 \mbox{ in } \omega .
 \end{array}
 \right.
 \end{equation}
 Observe now that all the functions and, thus, the Laplacian, only involve the $y$-variable. The following result addresses the question of existence of a solution to this eigenvalue problem.
  \begin{theo}\label{theo_eigen}
 	For every $\a\in\R$, the eigenvalue problem \eqref{eq:eigen} admits a unique eigenvalue $\Lambda (\alpha)$ associated with an eigenfunction pair $(U_\alpha,V_\alpha)$. Moreover, the eigenvalue $\L (\a)$ is simple, i.e., the eigenfunction is unique up to multiplication by a positive factor.
 \end{theo}
\noindent Indeed, since we require $U_\a$ and $V_\a$ to be positive, $\Lambda (\alpha)$ is the principal eigenvalue of the operator which to a couple $(U, V)$ associates the solution of the underlying elliptic system. Therefore, the previous result follows from a standard approach and Krein-Rutman theorem. We omit the proof and simply refer to T. Giletti, L. Monsaingeon and M. Zhou~\cite{GMZ} for the details (the problem treated there is slightly different, but the proof of Theorem~\ref{theo_eigen} can be obtained through straightforward changes).\\

Let us now briefly sketch the relation between positive solutions of~\eqref{eq:eigen} and problem~\eqref{eq:evol}. As for the classical KPP equation (see, e.g., \cite{AW78}), assumption \eqref{ass:KPP} implies that the dynamics of \eqref{eq:evol} can be characterized through its linearization at $(0,0)$. It is easy to see that positive solutions of such a linearization of the form
 \begin{equation}
 \label{eq:supersol}
 (u(t,x,y),v(t,x,y)) =  e^{-\alpha (x-ct)}\left(U (y), V (y)\right),
 \end{equation}
  with $\alpha, c \in \R$ having the same sign, and $U:\partial \omega \mapsto \R$ and $V:\omega \mapsto \R$ some positive functions, exist if and only if $(U,V)$ is a principal eigenfunction of \eqref{eq:eigen}, and $c \a=-\L (\a)$, i.e., $c=\frac{-\L (\a)}{\a}$.
  
  The comparison principle guarantees that positive solutions of the linearized problem of the form \eqref{eq:supersol} can be used to bound from above the asymptotic speed of propagation of~\eqref{eq:evol} in the $x$-direction by the quantity $c$. Thus, one want to look for the minimum value of $c$ for which such solutions exist, in order to get the best possible upper bound for the speed of propagation.
  
  Actually, the following result, which is the fundamental result for the spreading of solutions of \eqref{eq:evol}, guarantees that such a minimum provides us exactly with the asymptotic speed of propagation. 
 \begin{theo}\label{th:spreading}
 	Assume that $\Lambda (0) <0$, and define 
 	\begin{equation}
 	c^* := \min_{\alpha >0} \frac{-\Lambda (\alpha)}{\alpha} >0 .
 	\label{spreading-speed}
 	\end{equation}
 	Then the positive steady state $(U^* (y), V^* (y))$ exists, and solutions of \eqref{eq:evol} with compactly supported initial data spread in the $x$-direction with speed $c^*>0$, in the sense that:
 	$$\forall\, 0 < c < c^*, \quad \lim_{t \to +\infty} \sup_{|x| \leq c t}  \left[ \sup_{y \in \partial \omega} \left|u(t,x,y) - U^* (y)\right| + \sup_{y \in \omega} \left|v (t,x,y) - V^* (y)\right| \right] =0,$$
 	$$\forall\, c>  c^*, \quad \lim_{t \to +\infty} \sup_{|x| \geq c t}  \left[ \sup_{y \in \partial \omega} |u(t,x,y) | + \sup_{y \in \omega} |v (t,x,y) | \right] =0.$$
 	Conversely, if $\Lambda (0) \geq 0$ then the solution $(u(t,x,y),v(t,x,y))$ converges, uniformly in space, to $(0,0)$ as $t\to +\infty$. In such a case, we define, by convention, $c^*=0$.
 \end{theo}
 We point out that the existence and positivity of the quantity $c^*$ defined in \eqref{spreading-speed} will follow from some concavity properties related to $\L (\a)$ that we prove in Proposition \ref{prop:concavity}.
 
 Moreover, we emphasize that the sign of $\L (0)$ allows us to establish whether invasion or extinction occurs for \eqref{eq:evol}. This is also in line with classical results on propagation for the KPP equation and, in the context of road-field systems, it had already been obtained in~\cite{GMZ, BDR19}.
 
Although the proof of Theorem \ref{th:spreading} is not trivial, it largely follows the lines of~\cite{GMZ}. The main idea is that the linearized system still admits solutions of the type~\eqref{eq:supersol} when $c< c^*$, but with imaginary $\alpha$'s. These can be used to construct compactly supported subsolutions of~\eqref{eq:evol}. We also point out that the spreading speed of exponentially decaying initial data may be computed, using the same family of exponential ans\"{a}tze and a similar sub- and supersolution approach. We refer to~\cite[Proposition~1.4]{GMZ} for a related result in spatially periodic case.

 Theorem \ref{th:spreading} leads us to further investigate the principal eigenvalue $\Lambda (\alpha)$. First of all, we observe that, if $\left(U_1,V_1\right)$ is a positive eigenfunction of \eqref{eq:eigen}, then $\left(U_2,V_2\right)=\left(\sqrt{\mu}U_1,\sqrt{\nu}V_1\right)$ is a solution of the following problem:
  \begin{equation}
 \left\{
 \begin{array}{rcll}
 -D \Delta U - \left(D\alpha^2 +  \partial_u g (y,0) - \kappa (y) \mu\right) U - \kappa (y) \sqrt{\mu \nu} V & = & \Lambda (\alpha) U, & \quad \mbox{ on } \partial \omega,\vspace{3pt}\\
 -d \Delta V - \left(d\alpha^2 + \partial_v f (y,0) \right) V  & = & \Lambda (\alpha) V, & \quad \mbox{ in }  \omega,\vspace{3pt}\\
 d \, \partial_n V-\kappa(y) \left( \sqrt{\mu \nu } U - \nu V \right) & =& 0  , & \quad \mbox{ on } \partial \omega, \vspace{3pt}\\
U >0 \mbox{ on } \partial \omega \mbox{ and } V >0 \mbox{ in } \omega .
 \end{array}
 \right.
 \label{coupled-Eig-Pb}
 \end{equation}
 Since \eqref{coupled-Eig-Pb} has a similar structure to \eqref{eq:eigen}, an analogue of Theorem \ref{theo_eigen} holds also for this eigenvalue problem. Hence, the principal eigenvalue for both problems is the same.
 
 The advantage of considering \eqref{coupled-Eig-Pb} becomes clear if we write the equivalent weak formulations associated with both the surface and the bulk equations and add them, obtaining
 \begin{equation}
 	\label{eq:weak_eig}
 	A(\a)\left(W,\Psi\right)=\L (\a)B\left(W,\Psi\right) \qquad \text{for all } \Psi\in \mathcal{T},
 	\end{equation}
 where the test space is $\mathcal{T}=H^1(\partial\omega)\times H^1(\omega)$, and we have set $W=(U,V)$, $\Psi=(\varphi,\psi)$,
 \begin{equation}
 	\label{eq:bilA}
 	\begin{aligned}
 	A(\a)\left(W,\Psi\right)&:=D\int_{\partial \omega} \nabla U \cdot \nabla \varphi - \int_{\partial \omega} \left(D\alpha^2 +  \partial_u g (y,0) - \kappa (y) \mu\right)  U \varphi - \int_{\partial \omega} \kappa (y) \sqrt{\mu \nu} V \varphi \\
 	& \quad + d\int_\omega \nabla V \cdot  \nabla \psi -  \int_{\omega}(d\alpha^2+\f)  V \psi - \int_{\partial \omega} \ka \left( \sqrt{\mu \nu } U - \nu V \right)\psi,  
 	\end{aligned}
 	\end{equation}
 and
 \begin{equation}
 	\label{eq:bilB}
 	B\left(W,\Psi\right):=\int_{\partial \omega} U \varphi +\int_\omega V \psi   
\end{equation}
(observe that, in the integrals on $\partial\omega$ appearing in \eqref{eq:bilA}, $V$ and $\psi$ are meant in the sense of traces). Indeed, the bilinear form $A(\a)$ is self-adjoint, which is not the case for the weak formulation obtained from problem \eqref{eq:eigen}. Thus, as in the classical theory, a Rayleigh variational formula for the principal eigenvalue $\L (\a)$ can be obtained.  Such a formula is presented in the next proposition. We mention that a similar formula has been obtained in \cite{BDR19} in a different context. 
 \begin{prop}\label{prop:main}
 	The principal eigenvalue $\Lambda (\alpha)$ of problem \eqref{coupled-Eig-Pb} satisfies
 	\begin{equation}
 	\label{eq:Rayleigh} \Lambda (\alpha)=\min_{\substack{(U,V)\in \mathcal{T} \\(U,V)\neq(0,0)}}  \, \frac{\mathcal{F} (\alpha) [ U,V]}{\int_{\partial \omega} U^2 + \int_{\omega} V^2}=\min_{(U,V)\in\mc{A}}  \, \mathcal{F} (\alpha) [U,V], 
 	\end{equation}
 	where $\mathcal{F} (\alpha) [U,V]:=A(\a)((U,V),(U,V))$, i.e., from \eqref{eq:bilA},
 	\begin{equation}
 		\label{eq:F}
 		\begin{aligned}
 			\mathcal{F} (\alpha) [U,V] & = D \int_{\partial \omega} |\nabla U|^2 -  \int_{\partial \omega} \left(D\alpha^2+\partial_u g (y,0)\right)U^2  \\ 	 
 			& \quad +  d \int_{\omega} |\nabla V|^2 -  \int_{ \omega} \left(d\alpha^2 +  \partial_v f (y,0)\right)V^2 +  \int_{\partial \omega} \kappa (y) \left(\sqrt{\mu }U-\sqrt{\nu}V\right)^2  , 
 		\end{aligned}
 	\end{equation}
 	and 
 	$$\mathcal{A}:=\left\{ (U,V) \in H^1 (\partial \omega) \times H^1 (\omega) :  \int_{\partial \omega} U^2 + \int_{\omega} V^2 = 1  \right\}.$$
 	The minimum in \eqref{eq:Rayleigh} is uniquely achieved when $(U,V)$ is a principal eigenfunction of \eqref{coupled-Eig-Pb}.
 \end{prop}
 This formula will be the main tool for our analysis on the spreading speed~$c^*$ of solutions of \eqref{eq:evol}. More precisely, we will use it to study the influence of diffusion and of the size and shape of the domain on the propagation. However, in order to already highlight its usefulness, we start with the following immediate corollary:
\begin{cor}  \label{cor:c^*_kappa}
Let $\kappa_1 \leq  \kappa_2$ be two nonnegative and nontrivial functions in $C^{1,r} (\partial \omega)$. 

Define $c^*_1$ and $c^*_2$ the spreading speeds of solutions of \eqref{eq:evol}, where $\kappa$ is replaced respectively by $\kappa_1$ and $\kappa_2$. Then $c^*_1  \geq c^*_2.$
\end{cor}
Indeed, one may easily see from \eqref{eq:Rayleigh} and \eqref{eq:F} that $\Lambda(\a)$ depends monotonically on the function $\kappa$, and, together with \eqref{spreading-speed}, one reaches the wanted conclusion. In other words Corollary~\ref{cor:c^*_kappa} insures that, if the porosity between the surface and the bulk is increased, then the invasion is slowed down, or possibly even blocked. Notice that this was not obvious a priori, because one cannot apply a comparison principle to compare solutions of \eqref{eq:evol}, respectively with $\kappa = \kappa_1$ and $\kappa = \kappa_2$.\\

Before we further study the influence of various parameters on the invasion speed, we begin by establishing some properties of the principal eigenvalue $\L (\a)$ in the next subsection. In particular, these properties will guarantee that the spreading speed $c^*$  given by~\eqref{spreading-speed} is well-defined.

 \subsection{Some properties of $\L(\a)$}
In this section we use the Rayleigh formula given in Proposition \ref{prop:main} to obtain several properties on the principal eigenvalue $\L(\a)$. We begin with some sufficient conditions that, in the light of Theorem \ref{th:spreading}, guarantee that the solution of \eqref{eq:evol} invades the domain.

\begin{prop}
	\label{prop:sufficient}
	Assume that
	\begin{equation}
	\label{eq:suff1}
	\nu\int_{\partial\omega}\g+\mu\int_{\omega}\f>0.
	\end{equation}
	Then $\L(0)<0$. In particular, \eqref{eq:suff1} holds true if
	\begin{equation*}
	\label{eq:suff2}
	\g \geq 0, \ \forall y\in\partial\omega, \quad \text{and} \quad  \f>0, \ \forall y\in\omega.
	\end{equation*}
	\end{prop}

\begin{proof}
	From \eqref{eq:F} and \eqref{eq:suff1}, we have
	\begin{equation*}
	\mathcal{F} (0) [\sqrt{\nu},\sqrt{\mu}]=-\nu\int_{\partial\omega}\g-\mu\int_{\omega}\f<0.
	\end{equation*}
	Thus, Proposition \ref{prop:main} gives the desired result.
\end{proof}

Our second result compares $\L(\a)$ with two other eigenvalues which are related to the problems on the surface and in the bulk, considered separately. As a first step, we introduce such problems. The problem in the bulk, with no diffusion equation on the boundary surface, can be modeled with Neumann boundary conditions, where the boundary $\R\times\partial\omega$ acts as a barrier. It reads:
\begin{equation}\label{eq:evol_field}
\left\{
\begin{array}{ll}
\partial_t v = d \Delta v + f(y,v), & \quad \mbox{ for } t >0 , \ (x,y)\in \R \times  \omega, \vspace{3pt}\\
d \, \partial_n v  =   0, & \quad \mbox{ for } t >0, \ (x,y) \in  \R \times \partial \omega.\vspace{3pt}\\
\end{array}
\right.
\end{equation}
Similar to what we have done for the bulk-surface system, we can introduce the quantity
\begin{equation}
\label{eq:Ff}
\mathcal{F}_{f} (\alpha) [V] =  d \int_{\omega} |\nabla V|^2 -  \int_{ \omega} \left(d\alpha^2 +  \partial_v f (y,0)\right)V^2  ,
\end{equation}
and the spreading speed in the $x$-direction for problem \eqref{eq:evol_field} is given by the formula
$$c^*_f := \min_{\alpha >0} \frac{-\Lambda_f (\alpha)}{\alpha},$$
where 
\begin{equation}
	\label{eq:eig_f}
\Lambda_f (\alpha) := \min_{\substack{V \in H^1 (\omega) \\V\not\equiv 0}} \frac{\mathcal{F}_f (\alpha) [V] }{ \int_{\omega} V^2} = \Lambda_f (0) - d \alpha^2.
\end{equation}
More precisely, $c^*_f$ is well-defined provided that 
\begin{equation}\label{instability_f}
\Lambda_f (0) < 0.
\end{equation}
These results are well-known and we refer to~\cite{BerestyckiNirenberg-TWcylinders} for a proof of them. In the particular homogeneous case where $f (y,v) \equiv f(v) $ does not depend on $y$, then $\Lambda_f (\alpha) = - f '(0) - d\alpha^2$ and one recovers the usual KPP speed $c^*_f = 2 \sqrt{d f'(0)}$.

For later purposes, we also consider the more general problem with Robin boundary conditions:
\begin{equation*}\label{eq:evol_field_Robin}
\left\{
\begin{array}{ll}
\partial_t v = d \Delta v + f(y,v), & \quad \mbox{ for } t >0 , \ (x,y)\in \R \times  \omega, \vspace{3pt}\\
d \, \partial_n v  =  - \kappa(y)\nu v, & \quad \mbox{ for } t >0, \ (x,y) \in  \R \times \partial \omega.\vspace{3pt}\\
\end{array}
\right.
\end{equation*}
whose associated functional and eigenvalue are given, respectively, by
\begin{equation}
\label{eq:Ffk}
\mathcal{F}_{f,\kappa} (\alpha) [V] =  d \int_{\omega} |\nabla V|^2 -  \int_{ \omega} \left(d\alpha^2 +  \partial_v f (y,0)\right)V^2 + \int_{\partial \omega} \kappa(y)\nu V^2  ,
\end{equation}
and
\begin{equation}
\label{eq:eig_fk}
\Lambda_{f,\kappa} (\alpha) := \min_{\substack{V \in H^1 (\omega) \\V\not\equiv 0}} \frac{\mathcal{F}_{f,\kappa} (\alpha) [V] }{ \int_{\omega} V^2} .
\end{equation}
Observe that $\L_{f,0}(\a)$ equals the eigenvalue $\L_{f}(\a)$ defined in \eqref{eq:eig_f}.

Analogously, if we consider propagation on the surface only, we obtain
\begin{equation*}\label{eq:evol_road}
\partial_t u = D \Delta u + g(y,u),  \qquad \mbox{ for } t >0, \ (x,y)\in \R \times \partial \omega, \end{equation*}
and this problem has an asymptotic speed of propagation which can be characterized by means of the principal eigenvalue
\begin{equation}
	\label{eq:eig_g}
	\Lambda_g (\alpha) := \min_{\substack{U \in H^1 (\partial\omega) \\U\not\equiv 0}} \frac{\mathcal{F}_g(\a)[U]}{\int_{\partial\omega} U^2}=\Lambda_g (0) - D \alpha^2,
	\end{equation}
where
$$
\mathcal{F}_g(\a)[U]:= D \int_{\partial\omega} | \nabla U |^2 - \int_{\partial\omega} \left(D \alpha^2 + \partial_u g (y,0) \right) U^2.
$$
Let us insist on the fact that, although the forms of $\mathcal{F}_f$ and $\mathcal{F}_g$ are similar, these two functionals are completely distinct because integrals are taken on different domains.

We are now in position to give a bound for the principal eigenvalue $\L(\a)$.

\begin{prop}\label{prop:ineq1}
	For any $\alpha \geq 0$, the following inequality holds:
	$$\Lambda (\alpha) \geq \min \{ \Lambda_{g} (\alpha), \Lambda_f (\alpha) \}.$$
\end{prop}
\begin{proof}
	Take any $\alpha \geq 0$ and any $(U,V)\in\mc{A}$. Then, from \eqref{eq:F}, \eqref{eq:eig_f} and \eqref{eq:eig_g}, we have
	$$\mathcal{F}(\a)[U,V]\geq\mathcal{F}_g(\a)[U]+\mathcal{F}_f(\a)[V]\geq\L_g(\a)\int_{\partial\omega} U^2+\L_f(\a)\int_{\omega} V^2\geq \min \{ \Lambda_{g} (\alpha), \Lambda_f (\alpha) \}. $$
	Then, from Proposition \ref{prop:main},
	$$\L (\a)=\min_{(U,V)\in\mc{A}}  \, \mathcal{F} (\alpha) [U,V]\geq \min \{ \Lambda_{g} (\alpha), \Lambda_f (\alpha) \}. $$
\end{proof}

The next result is related to some properties of the function $\L (\a)$, its derivatives, and the function $\frac{-\L (\a)}{\a}$. It is the key to prove the existence and uniqueness of the minimum point defining $c^*$ in Theorem \ref{th:spreading}. Moreover, it will be useful for our numerical computations presented in later sections.

\begin{prop}
	\label{prop:concavity}
	\begin{enumerate}[label=(\roman*)]
		\item \label{prop:concavity.i} The function $\a\mapsto\L (\a)$ is continuous, even and concave on $\R$.
		\item \label{prop:concavity.ii} The function $\a\mapsto\L (\a)$ is differentiable and satisfies
		\begin{equation}
		\label{eq:derivL}
		\L '(\a)=-2\a\left(D\int_{\partial \omega} U_\a^2 +d \int_{\omega} V_\a^2\right),
		\end{equation}
		where $\left(U_\a,V_\a\right)$ is the principal eigenfunction of \eqref{coupled-Eig-Pb} normalized so that $\left(U_\a,V_\a\right)\in\mathcal{A}$.
		\item \label{prop:concavity.iii} If $\L (0)<0$, the function $\a\mapsto\frac{-\L (\a)}{\a}$, $\a>0$, is positive and has a unique minimum point.
		\end{enumerate}
	\end{prop}

\begin{proof}
	\begin{enumerate}[label=(\roman*)]
		\item It is clear from \eqref{eq:F} that, for all $(U,V)\in\mathcal{T}$, the function $\a\mapsto\mathcal{F}(\a)[U,V]$ is even and concave. Then, $\L(\a)$, as a minimum of even and concave functions, is also even and concave. Finally, the continuity follows from the concavity.
		\item The differentiability of $\L(\a)$, as well as the one of $(U_\a,V_\a)$ with respect to $\a$, follows from the theory of T. Kato \cite[Chapter VII]{kato-perturbation} since the operator defining this eigenvalue depends analytically on $\alpha$, and since its principal eigenvalue is isolated (as a consequence of the fact that it is the only eigenvalue associated with a positive eigenfunction). Then, if we define $\mathcal L:\R\times\R\times\mathcal{T}\times\mathcal{T}\to\R$ as follows,
		\begin{equation*}
			 \mathcal L (\lambda,\alpha,W,\Psi) := \lambda + A(\alpha)(W,\Psi)-\lambda B(W,\Psi),
			\end{equation*}
		where $A(\a)$ and $B$ are the bilinear forms introduced in \eqref{eq:bilA} and \eqref{eq:bilB}, Proposition \ref{prop:main} gives that $\L(\a)=\mathcal L (\L(\a),\alpha,W_\a,W_\a)$, where we have set $W_\a=(U_\a,V_\a)$. Hence,
		\[\L'(\a)=\frac{\partial\mathcal{L}}{\partial \lambda}\Lambda'(\a)+\frac{\partial\mathcal{L}}{\partial \alpha}+\frac{\partial\mathcal{L}}{\partial W}\frac{\mathrm{d} W_\a}{\mathrm{d}\a}+\frac{\partial\mathcal{L}}{\partial \Psi}\frac{\mathrm{d} W_\a}{\mathrm{d}\a},\]
		with all the partial derivatives evaluated at $(\L(\a),\alpha,W_\a,W_\a)$. Since $W_\a\in\mathcal{A}$, we have
		\begin{equation}
			\label{eq:2.18bis}
		\frac{\partial\mathcal{L}}{\partial \lambda}(\L(\a),\alpha,W_\a,W_\a)=1-B(W_\a,W_\a)=0,
		\end{equation}
		and
		\[
		\frac{\partial\mathcal{L}}{\partial W}(\L(\a),\alpha,W_\a,W_\a)=\frac{\partial\mathcal{L}}{\partial \Psi}(\L(\a),\alpha,W_\a,W_\a)=A(\a)(W_\a,W_\a)-\L(\a)B(W_\a,W_\a)=0,
		\]
		thanks again to Proposition \ref{prop:main}. Thus, we conclude that
		\[\L'(\a)=\frac{\partial\mathcal{L}}{\partial \alpha}(\L(\a),\alpha,W_\a,W_\a)=-2\a\left(D\int_{\partial \omega} U_\a^2 +d \int_{\omega} V_\a^2\right).
		\]
		\item First of all observe that the function $\frac{-\L(\a)}{\a}$ is continuous for $\a>0$ thanks to part \ref{prop:concavity.i}, and also positive if $-\L (0) >0$. Moreover, we have that
		\[
		\lim_{\a\searrow0}\frac{-\L(\a)}{\a}=+\infty=\lim_{\a\to+\infty}\frac{-\L(\a)}{\a}
		\]
		where the first equality comes from the assumption $\L(0)<0$ and the second one can be obtained by taking $(U,V)=(1,0)$ as a test function in~\eqref{eq:Rayleigh}. Then, $\frac{-\L(\a)}{\a}$ has at least a minimum point in $\{\a>0\}$. To check that it is unique, we show that, at any critical point, the second derivative of $\frac{-\L(\a)}{\a}$ is positive.
		
		On the one hand, at any critical point, a direct differentiation gives
		\begin{equation*}
			\left(-\frac{\L(\alpha)}{\alpha} \right)'' =  - \frac{\L '' (\alpha)}{\alpha} + 2 \frac{\L ' (\alpha)\a-\L(\a)}{\alpha^3}=- \frac{\L '' (\alpha)}{\alpha},
		\end{equation*}
	since $\left(-\frac{\L(\alpha)}{\alpha} \right)'=\frac{-\L ' (\alpha)\a+\L(\a)}{\a^2}$ vanishes.
	
	On the other hand, by differentiating \eqref{eq:derivL}, we have
	\begin{equation}
		\label{eq:secderiv}
	\L''(\a)=-2\left(D\int_{\partial \omega} U_\a^2 +d \int_{\omega} V_\a^2\right)-4\a\left(D\int_{\partial \omega} U_\a\frac{\mathrm{d}U_\a}{\mathrm{d}\a} +d \int_{\omega} V_\a\frac{\mathrm{d}V_\a}{\mathrm{d}\a}\right).
	\end{equation}
	The first summand in the previous expression, is negative. To study the sign of the second one, taking $W=W_\a$ and any $\Psi\in\mathcal{T}$, we differentiate \eqref{eq:weak_eig} with respect to $\a$:
	\[
	\frac{\partial A(\a)}{\partial \a}\left(W_\a,\Psi\right)+A(\a)\left(\frac{\mathrm{d}W_\a}{\mathrm{d}\a},\Psi\right)=\L'(\a)B\left(W_\a,\Psi\right)+\L(\a)B\left(\frac{\mathrm{d}W_\a}{\mathrm{d}\a},\Psi\right).
	\]
	Now, we take $\Psi=\frac{\mathrm{d}W_\a}{\mathrm{d}\a}$ in the above expression and, observing that $B\left(W_\a,\frac{\mathrm{d}W_\a}{\mathrm{d}\a}\right)=0$, since $B\left(W_\a,W_\a\right)=1$, we get
	\begin{align*}
		0&\leq A(\a)\left(\frac{\mathrm{d}W_\a}{\mathrm{d}\a},\frac{\mathrm{d}W_\a}{\mathrm{d}\a}\right)-\L(\a)B\left(\frac{\mathrm{d}W_\a}{\mathrm{d}\a},\frac{\mathrm{d}W_\a}{\mathrm{d}\a}\right) \\
		&=-\frac{\partial A(\a)}{\partial \a}\left(W_\a,\frac{\mathrm{d}W_\a}{\mathrm{d}\a}\right) \\
		&=2\a\left(D\int_{\partial \omega} U_\a\frac{\mathrm{d}U_\a}{\mathrm{d}\a} +d \int_{\omega} V_\a\frac{\mathrm{d}V_\a}{\mathrm{d}\a}\right),
		\end{align*}
		where the first inequality comes from Proposition \ref{prop:main} and the last identity from \eqref{eq:bilA}. As a consequence, \eqref{eq:secderiv} gives, at any critical point,
		\[
		\L''(\a)\leq-2\min\{D,d\}<0.
		\]
It follows that any critical point of $\frac{-\Lambda (\alpha)}{\alpha}$ must be a strict local minimum; hence the minimum point is unique.
		\end{enumerate} \end{proof}

\section{Dependence of the speed of propagation on $D$}
\label{section:c^*_D}
We are now in a position to analyze the influence of various parameters on the spreading speed~$c^*$. In later sections we will turn to its dependence on the size and shape of the domain, while here we start  by focusing on the dependence on the diffusion coefficient $D$ on the surface. In particular, we will see that some previous results, which had been obtained in~\cite{RTV} in the homogeneous case and with $\omega$ being a ball in $\R^N$, can be extended to our heterogeneous framework. Nonetheless, we will also point out that some results remain true only in the radially symmetric case.

In this subsection, to highlight the fact that the parameter $D$ varies, we will write $c^* (D)$, $\L_D (\a)$ and $\mathcal{F}_D (\alpha)$, instead of, respectively, $c^*$,  $\L (\alpha)$ and $\mathcal{F} (\alpha)$. 
Moreover, for the spreading speed to always be well-defined, we assume that
\begin{equation}\label{instability}
\sup_{D >0} \, \Lambda_D (0) <0 .
\end{equation}

Our first result is concerned with the limits of $c^*(D)$ as~$D$ is either small or large. Moreover, for some values of the parameters of our problem, we obtain a comparison of~$c^*(D)$ with $c^*_f$, the spreading speed of problem~\eqref{eq:evol_field} without the surface.

\begin{theo}\label{th:D_asymp}
Assume that \eqref{instability} and \eqref{instability_f} hold.
\begin{enumerate}[label=(\roman*)]
\item The spreading speed $c^* (D)$ converges to some $c_0 $ as $D~\to~0$. If, moreover, 
\begin{equation}\label{hyp:D_asymp}\max_{y \in \partial \omega} \partial_u g(y,0) <- 2 \Lambda_f (0),
\end{equation}
then 
$$c_0 \in (0, c^*_f) .$$
\item We have that $c^* (D) \to +\infty$ as $D \to +\infty$, and more precisely 
$$\exists A, B> 0, \quad A \sqrt{D} \leq c^* (D) \leq  B \sqrt{D}, \quad \text{for all $D$ large enough.}$$
\end{enumerate}
\end{theo}
\begin{rmk}
As far as statement $(i)$ is concerned, it is enough to assume that $\Lambda_D (0) <0$ for some (possibly small) $D>0$. Indeed, it easily follows from Proposition~\ref{prop:main} that $\Lambda_D (0)$ is nonincreasing with respect to $D$, so that this would be enough to guarantee the existence of a spreading speed as $D \to 0$.

Notice also that, although inequality~\eqref{hyp:D_asymp} may seem a bit abstract, in the case of an homogeneous reaction term $f(y,v) \equiv f(v)$ it simply rewrites as
$$\max_{y \in \partial \omega} \partial_u g (y,0) < 2 f' (0).$$
This is consistent with the fact that, in the homogeneous case, the diffusion threshold for acceleration by the surface becomes $D = 0$ when $g '(0) = 2 f' (0)$; see~\cite{BRR_plus}.

In the more general heterogeneous case, an explicit sufficient condition for~\eqref{hyp:D_asymp} to hold, that can be obtained by taking $V=1$ as a test function in~\eqref{eq:eig_f}, is
$$\max_{y \in \partial \omega} \partial_u g (y,0) < \frac{2}{|\omega|}  \int_{\omega} \partial_v f (\cdot,0).$$
\end{rmk}
The proof of the above theorem is postponed to the next subsection. This result states that, when the reaction on the surface is weaker than the reaction in the bulk in the sense specified by~\eqref{hyp:D_asymp}, then the surface slows down the propagation if the diffusion parameter~$D$ is too small. On the other hand, regardless of the reaction terms (provided that there is a positive spreading speed, which is guaranteed by~\eqref{instability}), the surface accelerates the propagation when $D$ is large.

To illustrate the previous result, we point out that both statements $(i)$ and~$(ii)$ hold true in the case when the domain is radially symmetric, the reaction term is logistic in the bulk, i.e., $f(y,v) = r v (1-v)$ for some $r >0$ (this also guarantees that $\Lambda_f (0) <0$), and there is no reaction on the surface, i.e., $\partial_u g (\cdot ,0) \equiv 0$. They also hold true if, instead of assuming that the domain is radially symmetric, one assumes that it is large enough. Indeed in both cases one can check all the necessary assumptions thanks to the variational characterization in Proposition~\ref{prop:main}. In particular, we partly recover the main results of~\cite{T16,RTV} but in a more general setting where the reaction and exchange terms may not be homogeneous and the domain may not be spherical.

However, in~\cite{T16,RTV}, one of the present authors and his collaborators went further and showed that there exists a critical diffusion threshold $D_0$ below which spreading is slowed down by the surface and above which it is accelerated. The key point was the strict monotonicity of $c^* (D)$ with respect to $D$, which was shown thanks to a semi-explicit construction of sub and supersolutions which allowed to characterize the spreading speed as a solution of a finite dimensional system. As we will see below in Proposition~\ref{prop:nonmonotoneD}, such a monotonicity may not be true in general. Still, we manage to recover it in a more general radially symmetric framework, as a consequence of the variational formula for $\Lambda (\alpha)$ given in Proposition~\ref{prop:main}. This is the content of the next result.
\begin{theo}\label{th:D_monotonicity}
Assume that \eqref{instability} and~\eqref{instability_f} hold, that the domain is either annular or a ball, i.e.,
\begin{equation*}
\omega = B_R \setminus B_r, \qquad R > r \geq 0,
\end{equation*}
where $B_r\subset\R^N$ denotes the open ball centered at the origin with radius $r$, and that all heterogeneities are radial functions, in the sense that
\begin{equation}\label{ass:radial}
 \kappa (y)= \kappa (|y|), \quad \partial_u f (y,0) = \tilde{f} (|y|), \quad  \partial_u g (y,0) = \tilde{g} (|y|),
\end{equation}
where $|y|$ denotes the Euclidean norm of $y$. 

Then $D \mapsto c^* (D)$ is increasing, and, in particular, there exists at most one value $D_0$ such that $c^* (D_0)= c^*_f$.
\end{theo}
Together with Theorem~\ref{th:D_asymp}, this result provides sufficient conditions for the existence of a diffusion threshold on the surface below which the surface slows down the propagation and above which it accelerates the propagation. These sufficient conditions include the assumptions of~\cite{T16,RTV}.
\begin{proof}[Proof of Theorem~\ref{th:D_monotonicity}] The key point is that, since it is unique (up to multiplication by a positive factor), under the assumptions of Theorem~\ref{th:D_monotonicity} the eigenfunction pair $(U_{\alpha,D}, V_{\alpha,D})$ associated with $\Lambda_D (\alpha)$ has to be radially symmetric. As a consequence, the minimum in~\eqref{eq:Rayleigh} can be taken over radially symmetric functions, and therefore in the definition of $\mathcal{F}_D (\alpha) [U, V]$ given in \eqref{eq:F}, we may remove the term $D \int_{\partial \omega} | \nabla U|^2$ without loss of generality. Since the remaining term that depends on $D$, that is $-D \alpha^2 \int_{\partial \omega} U^2$, is decreasing with respect to~$D$, we have that $c^*(D)$ is nondecreasing in $D$.

To show the strict monotonicity, we fix $0 < D_1 < D_2$, and $\alpha_2 >0 $ such that
$$\frac{-\Lambda_{D_2} (\alpha_2)}{\alpha_2} = c^* (D_2).$$
Now consider $(U_1,V_1)\in\mc{A}$ the eigenfunction pair associated with $\Lambda_{D_1} (\alpha_2)$. Using it as a test function in the variational formula for $\Lambda_{D_2} (\alpha_2)$, we deduce that
$$\Lambda_{D_2} (\alpha_2) \leq \mc{F}_{D_2}(\a_2)[U_1,V_1]= \Lambda_{D_1} (\alpha_2) + \left(D_1 - D_2\right) \alpha_2^2 \int_{\partial \omega} U_1^2.$$ Since $U_1 >0$, we get that 
$$c^*(D_2)=-\frac{\Lambda_{D_2}(\a_2)}{\a_2}>-\frac{\Lambda_{D_1}(\a_2)}{\a_2}\geq c^*(D_1).$$
\end{proof}
\begin{rmk}\label{rmk:several}
Let us highlight several points.
\begin{itemize}
\item The difficulty in checking the monotonicity with respect to the  parameter $D$ in the non radial case is that diffusion along the $x$-variable and diffusion across the $y$-variable might have opposite effects on the speed. Indeed, if one replaces the operator $D \Delta u$ with $D_x \partial^2_{xx} u + D_y \Delta_y u$ in the surface equation in~\eqref{eq:evol}, one could immediately conclude, after obtaining the analogous characterization of the principal eigenvalue for such an operator, as in Proposition~\ref{prop:main} and reasoning as in the proof of Theorem~\ref{th:D_monotonicity}, that the speed $c^* (D_x,D_y)$ is nondecreasing with respect to $D_x$ and nonincreasing with respect to~$D_y$.
\item In a similar fashion, it is easy to check that Theorem~\ref{th:D_monotonicity} remains true in dimension~$N=1$ (which is consistent with the original ecological motivation of road-field systems) without assuming~\eqref{ass:radial} on the heterogeneities. Indeed, in this case, the surface reduces to a union of parallel straight lines, thus the diffusion in the surface cross section component vanishes.

\item Our argument also shows that, under the assumptions of Theorem~\ref{th:D_monotonicity}, the eigenvalue $\Lambda_D (0)$ actually does not depend on $D$. In particular, in this theorem, assumption~\eqref{instability} can be simply replaced by $\Lambda_0 (0) < 0$. 
\item In the non-radial case, we have that $\Lambda_D (0)$ increases with respect to~$D$. Therefore, one may find a situation where the speed is positive for small~$D$ but extinction occurs for large~$D$ (see Subsection~\ref{sec:D_counter}). However, we do not theoretically know whether the monotonicity with respect to $D$ can fail under assumption~\eqref{instability}. Numerical simulations suggest that it does fail in some situations (see Section~\ref{sec:initial-experiments}).
\end{itemize}
\end{rmk}

\subsection{Proof of Theorem~\ref{th:D_asymp}} 

\emph{Proof of (i).} We start with the limit as $D \to 0$. First fix $\alpha \geq 0$ and introduce the formal limit of $\L_D (\alpha)$, 
$$\Lambda_0 (\alpha) :=  \inf_{(U,V)\in\mc A} \mc F_0(\a)[U,V], $$
where $\mc F_0(\a)[U,V]$ is the quantity defined in \eqref{eq:F} with $D=0$.
Note that this might no longer be a minimum because the regularizing term with $\nabla U$ has vanished. However, $\mc F_0(\a)[U,V]$ with~$(U,V)\in\mc A$ can still be bounded from below, so that the infimum is well-posed.

Now we take a minimizing sequence $(U_n, V_n)\in\mc A$. Then,
$$\Lambda_D (\alpha) \leq \mc F_D(\a)\left[U_n,V_n\right]\leq \Lambda_0 (\alpha) + \varepsilon + D \int_{\partial \omega} |\nabla U_n|^2 - D \alpha^2 \int_{\partial \omega} U_n^2 ,$$
where $\varepsilon >0$ is arbitrarily small and $n$ is large (depending on $\varepsilon$). Thus
$$\limsup_{D \to 0} \Lambda_D (\alpha) \leq \Lambda_0 (\alpha).$$
Conversely, let $(U_D,V_D)\in\mc A$ be the minimizer  associated with $\Lambda_D (\alpha)$; thus
$$\Lambda_0 (\alpha) \leq \Lambda_D (\alpha) - D \int_{\partial \omega} |\nabla U_D|^2 + D\alpha^2 \int_{\partial \omega} U_D^2\leq \Lambda_D (\alpha)+  D \alpha^2,$$
and we can pass to the liminf and find that
$$\Lambda_0 (\alpha) \leq \liminf_{D \to 0} \Lambda_D (\alpha).$$
We now have constructed the pointwise limit of the eigenvalues $\L_D(\a)$ as $D \to 0$. Recall that, by Proposition~\ref{prop:concavity}, the function $\alpha \mapsto \Lambda_D (\alpha)$ is even and concave, hence it is nonincreasing on $\R_+$. We infer that the pointwise limit $\alpha \mapsto \L_0 (\alpha)$ is also even and concave; in particular, it is continuous. By applying Dini's second theorem we find that the convergence of $\Lambda_D$ to~$\Lambda_0$ is locally uniform in $\{\a\geq 0\}$.

Now, thanks to~\eqref{instability} which guarantees that $\Lambda_0 (0) < 0$, we can define
$$c_0 := \min_{\alpha >0} \frac{-\Lambda_0 (\alpha)}{\alpha} >0.$$
Here we also used the fact that $\lim_{\alpha \to +\infty} \frac{-\L_0 (\alpha)}{\alpha} \to +\infty$, which easily follows from the definition of~$\L_0(\a)$, for example by testing $\mc F_0(\a)$ with $(U,V)=\left(0,\left|\omega\right|^{-1/2}\right)$. Furthermore, the minimum in the definition of $c_0$ is reached for some $\alpha_0 >0$ (nonetheless, unlike $\Lambda_D$, we cannot apply Kato's theory and differentiate $\Lambda_0$ with respect to $\alpha$; in particular, we do not prove that the minimum point is unique).

We now prove that
\begin{equation}
\label{eq:c^*_D_bounded}
c^*(D) \quad \text{ remains bounded as $D \to 0$}.
\end{equation}
Indeed, for all $\a\geq 0$, $(U,V)\in\mc A$ and $D<d$ it holds
$$\mc F_D(\a)[U,V]=\mc F_D(0)[U,V]-\a^2\left(D\int_{\partial\omega}U^2+d\int_{\omega}V^2\right). $$
Thus,
\begin{equation}
\label{eq:bound}
\Lambda_D(0)-d\a^2= \Lambda_D(0)-\a^2\max\{D,d\}\leq \Lambda_D(\a)\leq\Lambda_D(0)-\a^2\min\{D,d\}\leq\Lambda_D(0)-D\a^2,
\end{equation}
and
$$c^*(D)\leq\min_{\a>0}\frac{d\a^2-\Lambda_D(0)}{\a}=2\sqrt{-d\Lambda_D(0)}. $$
Since $\Lambda_D(0)\to\Lambda_0(0)<0$ as $D\to 0$, this concludes the proof of \eqref{eq:c^*_D_bounded}.

Consider now the points $\a_D$, for small $D > 0$, where the minimum defining $c^*(D)$ is attained. We will now prove that
\begin{equation}
\label{eq:compactinterval}
0<\liminf_{D\to 0}\a_D\leq\limsup_{D\to 0}\a_D<+\infty.
\end{equation}

For the first inequality, assume that there exists a sequence of values of $D$ (not relabeled), for which $\a_D\to 0$ as $D\to 0$. Then, for such a sequence, we obtain from \eqref{eq:bound}
$$\lim_{D\to 0} c^*(D)=\lim_{D\to 0} -\frac{\Lambda_D(\a_D)}{\a_D}\geq\lim_{D\to 0}\frac{D\a_D^2-\L_D(0)}{\a_D}\to+\infty,$$
against \eqref{eq:c^*_D_bounded}. To prove the second inequality in \eqref{eq:compactinterval}, assume that there exists a sequence of values of $D$ (again not relabeled), for which $\a_D\to +\infty$ as $D\to 0$. Consider $(U,V)=(0,K)$, where $K$ is an appropriate constant so that $(0,K)\in\mc A$. Then,
$$c^*(D)= -\frac{\Lambda_D(\a_D)}{\a_D}\geq -\frac{\mc F_D(\a_D)[U,V]}{\a_D}=\frac{d\a_D^2+K^2\int_{\omega} \f - K^2 \int_{\partial \omega} \kappa(y)\nu }{\a_D}\to+\infty $$
as $D\to 0$, again against \eqref{eq:c^*_D_bounded}.

After these preliminaries, by the locally uniform convergence of $\Lambda_D$ to $\Lambda_0$, it is then straightforward that
$$\lim_{D \to 0} c^* (D) = c_0.$$

It remains to show that $c_0 < c^*_f$. Recall from \eqref{eq:eig_f} that $\Lambda_f (\alpha) = \Lambda_f (0) - d \alpha^2$
does not depend on $D$, and observe that $-\frac{\Lambda_f(\a)}{\a}$, $\a>0$, achieves its minimum at
$$\a^*:=\sqrt{\frac{-\L_f(0)}{d}}, $$
and
\begin{equation}
\label{eq:Lfa0}
\L_f\left(\a^*\right)=\Lambda_f (0) - d \alpha^{*2}=2\Lambda_f(0).
\end{equation}
Since
$$c^*_f = -\frac{\Lambda_f(\a^*)}{\a^*} = 2 \sqrt{-d \Lambda_f (0) },$$
for our purpose it is enough to show that
\begin{equation}\label{eq:wanted1}
\Lambda_0 \left( \a^* \right) >  \Lambda_f \left( \a^* \right).
\end{equation}
To this end, consider a minimizing sequence $(U_n,V_n)\in\mc A$ of the left-hand term. Then, 
\begin{equation*}
\begin{split}
 \mc F_0(\a^*)\left[U_n,V_n\right] & \geq   -\max_{y \in \partial \omega} \partial_u g (y,0)  \int_{\partial \omega}  U_n^2  + d \int_{\omega} |\nabla V_n|^2   -  \int_{ \omega} \left(d\a^{*2} +  \partial_v f (y,0)\right)V_n^2  \\
& \geq   -\max_{y \in \partial \omega} \partial_u g (y,0)  \int_{\partial \omega}  U_n^2  + \Lambda_f \left(\a^* \right) \int_\omega V_n^2.
\end{split}
\end{equation*}
We now distinguish two cases:

(a) if $\liminf_{n\to+\infty}\int_{\partial \omega} U_n^2 > 0$, then, by taking the $\liminf$ in the previous relation, we obtain
$$ \L_0(\a^*)\geq-\max_{y \in \partial \omega} \partial_u g (y,0)  \liminf_{n\to+\infty}\int_{\partial \omega}  U_n^2  + \Lambda_f \left(\a^* \right) \liminf_{n\to+\infty}\int_\omega V_n^2>\L_f(\a^*),$$
where the last inequality comes from \eqref{hyp:D_asymp} and \eqref{eq:Lfa0}, together with the normalization condition;

(b) if $\liminf_{n\to+\infty}\int_{\partial \omega} U_n^2 = 0$ and, thus, $\limsup_{n\to+\infty}\int_{ \omega} V_n^2 = 1$, we have in this case and up to extraction of a subsequence that
$$\L_0(\a^*)=\lim_{n\to+\infty}\mc F_0(\a^*)\left[U_n,V_n\right]=\lim_{n\to+\infty}\mc F_{f,\kappa}(\a^*)\left[V_n\right]\geq\L_{f,\kappa}(\a^*), $$
where $\mc F_{f,\kappa}(\a)$ and $\L_{f,\kappa}(\a)$ are defined in \eqref{eq:Ffk} and \eqref{eq:eig_fk}. Observe that, for the second equality, we have used that the sequence $V_n$ is bounded in $H^1(\omega)$, thus, the sequence of the traces is bounded in $L^2(\partial\omega)$ (see, e.g., \cite[Theorem 5.5.1]{Evans-pde}), while the last inequality follows from \eqref{eq:eig_fk}. We now conclude by comparing \eqref{eq:Ffk} with \eqref{eq:Ff} and recalling that the minimizer $\overline V$ in \eqref{eq:eig_fk} is positive in $\overline\omega$, thanks to the Hopf lemma, thus
$$\L_{f,\kappa}(\a^*)=\mc F_{f,\kappa}(\a^*)\left[\overline V\right]>\mc F_{f}(\a^*)\left[\overline V\right]\geq \L_{f}(\a^*). $$
We again obtain that \eqref{eq:wanted1} holds, and statement~$(i)$ of Theorem~\ref{th:D_asymp} is proved.

\vspace{0.3cm}

\noindent\emph{Proof of (ii).} Let us now investigate the limit as $D \to +\infty$. Recall again that
$$\Lambda_f (\alpha) = \Lambda_f (0) - d \alpha^2$$
does not depend on $D$, and from \eqref{eq:eig_g} that
$$\Lambda_{g,D} (\alpha) = \Lambda_{g,D} (0) - D \alpha^2.$$
From \eqref{eq:eig_g}, by taking constant $U$'s as test functions, one easily infers that $\Lambda_{g,D} (0)$ is bounded with respect to $D$. Using Proposition~\ref{prop:ineq1}, we infer that
$$c^* (D) \leq \min_{\alpha >0} \frac{ K +  D \alpha^2 }{\alpha},$$
for some $K >0$ (independent of $D$) and any $D >d$. It immediately follows that there exists a positive constant~$B$ such that $c^* (D) \leq B \sqrt{D}$ for any large $D$.

It remains to prove a lower estimate on $c^* (D)$, which requires an upper bound on $\Lambda_D (\alpha)$. First denote
$$\eta = \sup_{D >0} \Lambda_D (0) < 0,$$
where the negativity comes from hypothesis~\eqref{instability}. Since $\alpha \mapsto \Lambda_D (\alpha)$ is an even and concave function, we have on one hand that 
$$\Lambda_D (\alpha) \leq \eta <0,$$
for any $D >0$ and $\alpha \geq 0$.

On the other hand, using Proposition~\ref{prop:main} and $(U,V)=(1,0)$ as a test function, we also find that
$$\Lambda_D (\alpha) \leq \frac{1}{|\partial \omega|} \int_{\partial \omega}  (\kappa (y) \mu - \partial_u g (y,0) - D \alpha^2) \leq K - D \alpha^2,$$
for some constant $K \in \mathbb{R}$. Finally
$$\Lambda_D (\alpha ) \leq  \min \{ \eta, K - D \alpha^2 \},$$
and direct computations show that there exists $A >0$ such that $c^* (D) \geq A \sqrt{D}$ for any large~$D$. This concludes the proof of Theorem~\ref{th:D_asymp}.

\subsection{About counterexamples to Theorem~\ref{th:D_monotonicity}}\label{sec:D_counter}

In this section, we discuss the fact that the spreading speed may not be increasing with respect to $D$ in general. We first give a theoretical analysis of a situation where assumption~\eqref{instability} is not satisfied. 

\begin{prop}\label{prop:nonmonotoneD}
	Assume that $\omega$ is a ball, that
	$$\partial_v f(y,0) \equiv -1, \quad \kappa (y) \equiv 1.$$
	and that $g$ satisfies
	$$ \langle \partial_u g(0) \rangle := \frac{1}{| \partial \omega |} \int_{\partial \omega} \partial_u g (y,0) dy <0<\mu < \max_{y \in \partial \omega} \partial_u g (y,0).$$
	
Then, $\Lambda_D (0) < 0$, hence $c^* (D) >0$, for $D$ small enough, while $\Lambda_D (0) >0$ for $D$ large enough.
\end{prop}

First of all, notice that assumption~\eqref{instability} is not satisfied, since $g$ is not constant on $\partial\omega$, the boundary of the ball. Then, we recall that the fact that $\Lambda_D (0) < 0$ implies $c^* (D) >0$ comes from Proposition~\ref{prop:concavity}, while, when $\Lambda_D (0)>0$, extinction occurs and there is no spreading. Moreover, one may check that $\Lambda_D (\alpha)$ depends continuously on $D$, and then infer that, in the set of parameters of Proposition~\ref{prop:nonmonotoneD}, the speed $c^* (D)$ has to be decreasing with respect to~$D$ at least on some interval. In other words, this proposition states that there is a situation where increasing the parameter $D$ may have a negative impact on the propagation and even stop it.

In Section~\ref{sec:initial-experiments}, we will also provide a numerical example where $D \mapsto c^* (D)$ is not increasing, even under assumption~\eqref{instability} and with $\partial_v f$ everywhere positive.

\begin{proof}
Proceeding as in the proof of Theorem~\ref{th:D_asymp}, we find that
$$\lim_{D \to 0} \Lambda_D (0) =   \inf_{(U,V) \in \mc A} \mc F_0(0)\left[U,V\right].$$
Fix $y_0$ where $\partial_u g (y_0,0) = \max \partial_u g (\cdot, 0)>0$. Then, we take as an ansatz 
$$U_n = a_n \, \chi_{B_{\frac{1}{n}} \left(y_0\right)}, \quad V_n = 0,$$
where $B_{\frac{1}{n}} \left(y_0\right)$ denotes the ball of radius $\frac{1}{n}$ (with respect to the geodesic distance on the hypersurface $\partial \omega$) centered at $y_0$, and $a_n:=\left|B_{\frac{1}{n}} \left(y_0\right)\right|^{-1/2}$, so that the normalization $\int_{\partial \omega} U^2_n  = 1$ is respected. Notice that $U_n$ belongs to $L^2 (\partial \omega)$ but not to $H^1 (\partial \omega)$, thus we also introduce a smooth function $U_n^k$ such that
$$U_n^k \to U_n \mbox{ in } L^2 (\partial \omega),$$
and $\|U_n^k \|_{L^2 (\partial \omega)} = 1$. Plugging $\left(U_n^k , 0\right)$ in the above variational formula, passing to the limit as $k \to +\infty$ and then as $n \to +\infty$, we find thanks to the spatial regularity of~$g$ that
$$\lim_{D \to 0} \Lambda_D (0)  \leq\mu -\partial_u g (y_0,0) < 0.$$

On the other hand, it is clear that $D \mapsto \Lambda_D (0)$ is nondecreasing, and as a matter of fact it is even increasing because of the lack of radial symmetry of the eigenfunction. Regardless, let us check that $\Lambda_D (0)$ becomes positive for large $D$. From Proposition~\ref{prop:ineq1} we have
$$\Lambda_D (0) \geq \min \{ \Lambda_{g,D} (0), \Lambda_f(0) \} = \min \{ \Lambda_{g,D} (0), 1 \},$$
where $\Lambda_{g,D} (0)$ is as in \eqref{eq:eig_g} and, with the assumptions of this proposition, we have
$$\Lambda_f(0)=\min_{\substack{V \in H^1 (\omega) \\V\not\equiv 0}} \frac{d \int_{\omega} |\nabla V|^2 +  \int_{ \omega} V^2}{\int_{ \omega} V^2}\geq 1  ,$$
and the equality is reached for constant functions, thus $\Lambda_f(0)=1$.

 Recalling the definition of $\mc F_{g,D}$, one can show (we omit the details because the proof is similar to that of Theorem~\ref{th:R0} in the next section) that
$\Lambda_{g,D} (0)$ converges to $- \langle \partial_u g (0) \rangle$ as $D \to +\infty$. In particular, it is positive for any $D$ large enough, and, as announced, so is~$\Lambda_D (0)$. \end{proof}

\section{Dependence of the speed on the size of the domain}
\label{section:c^*_R}

In this section, we investigate how the speed $c^*$ depends on the domain $\omega$. More precisely, we focus on the influence of the size of the domain when its shape is fixed, and provide some theoretical results. The issue of shape optimization will be studied numerically in Section~\ref{sec:shape}.

Here, we will consider~\eqref{eq:evol} in a spatial domain~$\omega_R$ which is a rescaling of a fixed domain~$\omega$:
$$\omega_R = \{ Ry \ | \ y \in \omega \}.$$ 
As before, $\omega$ is a bounded, open and connected set with a smooth boundary. Since the system \eqref{eq:evol} involves spatially dependent functions, it is necessary to specify how these depend on the parameter~$R$, which will be done below. In any case, the resulting modified system will satisfy the same assumptions as before, for any $R >0$.

To reflect our choice of the varying parameter, we will temporarily denote by $c^* (R)$ the spreading speed (when positive) and by $\Lambda_R (\cdot)$ the corresponding eigenvalues. 

A complete understanding of how the speed depends on the parameter $R$ is a rather complicated issue. Thus, from the theoretical point of view, we will mostly focus on the asymptotic behaviour of the speed $c^* (R)$ as $R$ goes to either 0 or~$+\infty$.

\subsection{Behaviour of $c^*$ as the size of the domain shrinks}

In accordance with our choice of a rescaled domain, in \eqref{eq:evol} we replace the spatially dependent functions $f$, $\kappa$ and $g$ respectively by
\begin{equation}\label{eq:rescaled_f}
f_R (y,v) = f \left(\frac{y}{R},v\right) , \quad \kappa_R (y) = \kappa \left(\frac{y}{R}\right), \quad \mbox{ and } \ g_R (y,u) =  g\left(\frac{y}{R},u\right), \qquad y\in\omega_R,
\end{equation}
where $f$, $\kappa$ and $g$ satisfy the assumptions given in Section \ref{section1}.

Before stating the main result of this subsection, let us introduce some new notions related to an averaged version of the surface equation on the original domain $\partial\omega$. First, when it is possible we define
$$c^*_{\langle g \rangle} := 2 \sqrt{D \langle \partial_u g (0) \rangle},$$
where
$$\langle \partial_u g(0) \rangle := \frac{1}{| \partial\omega |} \int_{\partial \omega} \partial_u g (y,0) dy .$$
As far as we know, there are very few results on spreading speeds for reaction-diffusion equations on manifolds. Still, by a straightforward analogy one may expect $c^*_{\langle g \rangle}$ to be the spreading speed of the equation
$$\partial_t u = D \Delta u + \langle g (u) \rangle \quad \mbox{ for }  (x,y) \in \R \times \partial \omega , \ t> 0,$$
where similarly as above $ \langle g (u) \rangle$ denotes, for each value of $u$, the average of $g(\cdot,u)$ over~$\partial \omega$.

We are now in a position to describe the situation when $R \to 0$:
\begin{theo}\label{th:R0}
We have that
$$-\Lambda_R (\alpha) \to D \alpha^2 + \langle \partial_u  g (0) \rangle \ \mbox{ as } R \to 0,$$
where the convergence is locally uniform in~$\{\alpha\geq 0\}$.

In particular, if $\Lambda_R (0) <0$ for small $R>0$, then $\langle \partial_u g (0) \rangle \geq 0$ and 
$$c^* (R) \to c^*_{\langle g \rangle} \ \mbox{ as } R \to 0.$$
On the other hand, if $\langle \partial_u g (0) \rangle $ is negative, then for  small enough $R$ we have that $\Lambda_R (0) >0$, and extinction takes place.
\end{theo}
The above theorem describes some sort of homogenization which occurs as the size of the domain goes to 0. It is interesting to see that, since the volume of $\omega_R$ decays to 0 faster than its boundary surface as $R \to 0$, the limit does not involve the bulk. In the situation when $g \equiv 0$, we recover that the speed goes to 0 as the domain shrinks (for fixed diffusions). This again represents a generalization of earlier results from \cite{RTV} to our more general setting.

\begin{proof}[Proof of Theorem~\ref{th:R0}]
Recall that
$$\Lambda_R (\alpha) = \min_{\mathcal{A}_R} \mathcal{F}_R (\alpha) [U,V],$$
where $\mathcal{F}_R$ and $\mathcal{A}_R$ are defined as in Proposition~\ref{prop:main} with $\omega_R$ and $\partial\omega_R$ instead of $\omega$ and $\partial\omega$.

By letting 
$$(U,V)= R^{-\frac{N-1}{2}} \left(\tilde{U} \left(\frac{y}{R}\right),  R^{-\frac{1}{2} } \tilde{V} \left(\frac{y}{R}\right) \right),$$
and immediately dropping the tildes on the functions  for convenience, we find that
\begin{equation}\label{eq:R_eigen1}
\Lambda_R (\alpha) = \min_{\mathcal{A}} \widetilde{\mathcal{F}}_R (\alpha) [U,V],
\end{equation}
where
\begin{align*}
\widetilde{\mathcal{F}}_R (\alpha) [U,V]&  := DR^{-2} \int_{\partial \omega} |\nabla U|^2  - \int_{\partial \omega} \left(D\alpha^2+\partial_u g \left(y,0\right) \right)U^2  \\
&\quad +  d R^{-2} \int_{\omega} |\nabla V|^2   -   \int_{ \omega} \left(d\alpha^2 +  \partial_v f (y,0)\right)V^2  + \int_{\partial \omega} \kappa (y) \left( \sqrt{\mu} U- \sqrt{ \frac{\nu}{R}} V \right)^2.
\end{align*}
Now choose $(U,V)= \left(|\partial \omega| + R \frac{\mu}{\nu} |\omega | \right)^{-1/2} \left(1, \sqrt{\frac{R \mu}{\nu}} \right)$ as a test function (observe that the multiplying constant is chosen to satisfy the normalization). 
Passing to the limsup as $R \to 0$, it is straightforward to obtain that, for any $\alpha \geq 0$,
\begin{equation}\label{eq:limsup_R0}
\limsup_{R \to 0} \Lambda_R (\alpha) \leq -D\alpha^2  - \langle \partial_u g (0) \rangle . 
\end{equation}
Let us now prove a reverse inequality. Denote by $(U_R,V_R)$ the (unique) minimizer of $\widetilde {\mathcal{F}}_R (\alpha)$ with the same normalization. 

Notice that
$$\frac{1}{R^2}\!\left(D \!\! \int_{\partial \omega} \!\!\!\! \left| \nabla U_R\right|^2 \!+\! d  \!\int_{\omega} \! | \nabla V_R |^2\right)\! \leq \!\Lambda_R (\alpha) + \!\int_{\partial \omega} \!\!\! \left(D\alpha^2+\partial_u g (y,0) \right)U_R^2  + \!\int_{ \omega} \!\! \left(d\alpha^2 +  \partial_v f (y,0)\right)V_R^2 .$$
The first term on the right-hand side can be bounded from above for small~$R$ thanks to~\eqref{eq:limsup_R0}. The second and third terms can be bounded uniformly with respect to $R$ thanks to the normalization $\int_{\partial \omega} U_R^2 + \int_\omega V_R^2 = 1$. It follows that
$$ \| \nabla U_R \|^2_{L^2 (\partial \omega)} +  \| \nabla V_R \|^2_{L^2 (\omega)} = O (R^2) \ \mbox{ as } R \to 0.$$
By Poincar\'{e}-Wirtinger inequality, we infer that
\begin{equation}\label{poinca}
 \| U_R\,  -  \langle  U_R \rangle \|_{H^1 (\partial \omega)} +   \| V_R\, -  \langle V_R  \rangle \|_{H^1 (\omega)}  = O (R)  \ \mbox{ as } R \to 0,
\end{equation}
where $\langle U_R \rangle$ and $\langle V_R \rangle $ denote the averages of $U_R$ and $V_R$ respectively on $\partial \omega$ and $\omega$.

Then, using again the normalization which insures that $U_R$ and $V_R$ are bounded in the $L^2$-norm, we get
$$\lim_{R \to 0} \int_{\partial \omega} \left(\kappa (y) \mu - D \alpha^2 - \partial_u g (y,0) \right) \left(U_R^2  \, - \langle U_R\rangle^2\right) = 0,$$ 
$$\lim_{R \to 0} \int_{\omega}  (d \alpha^2 + \partial_v f (y,0) ) (V_R^2 \,  - \langle V_R \rangle^2) = 0. $$ 
By the standard trace and Sobolev embeddings from $H^1 (\omega)$ to $H^{1/2} (\partial \omega)$ and from $H^{1/2} (\partial \omega)$ to $L^2 (\partial \omega)$, we also get from \eqref{poinca} that
\begin{eqnarray*}
	&& \left| \int_{\partial \omega} \kappa (y) \sqrt{\frac{\mu \nu}{R}} \left( U_R V_R\,  -  \langle U_R \rangle  \langle V_R \rangle \right) \right| \\
	&\leq  & \left| \int_{\partial \omega} \kappa (y) \sqrt{\frac{\mu \nu}{R}} (U_R \, - \langle U_R\rangle ) V_R \right| + \left| \int_{\partial \omega} \kappa (y) \sqrt{\frac{\mu \nu}{R}} \langle U_R\rangle  (V_R \, - \langle V_R\rangle ) \right| \\
	& \leq & \frac{C}{\sqrt{R}}  \| U_R \, - \langle U_R\rangle  \|_{L^2 (\partial \omega)}  \|V_R \|_{L^2 (\partial \omega)} +  \frac{C}{\sqrt{R}} \|U_R \|_{L^2 (\partial \omega)} \|V_R\,  - \langle V_R\rangle  \|_{L^2 (\partial \omega)} \\
	& \leq &  C  \sqrt{R} \\
	& \to & 0 \ \mbox{ as } \ R \to 0,
\end{eqnarray*}
where $C$ denotes various positive constants along this computation. 

We next compute
\begin{equation}
	\label{eq:4.5}
 \left| \int_{\partial \omega} \frac{\kappa (y)\nu}{R} \left( V_R^2 \, - \langle V_R\rangle ^2\right) \right| \\
	\leq \frac{C}{R} \| V_R\,  - \langle V_R\rangle  \|_{L^2 (\partial \omega)}  \|V_R \, + \langle V_R\rangle \|_{L^2 (\partial \omega)} \\
	=  O (1),  \mbox{ as }  R \to 0.
\end{equation}
We point out that, although the integrals in the above two expressions are taken on $\partial \omega$, the quantity $\langle V_R \rangle$ still denotes the average of $V_R$ on the whole domain $\omega$.

Going back to the definition of $\widetilde{\mathcal{F}}_{R} (\alpha)$, we get that
\begin{eqnarray*}
	 \widetilde{\mathcal{F}}_R (\alpha) [U_R,V_R]
	& \geq &   - \int_{\partial \omega} \left(D\alpha^2+\partial_u g (y,0) \right)U_R ^2 -   \int_{ \omega} \left(d\alpha^2 +  \partial_v f (y,0)\right)V_R^2   \\
	& &   + \int_{\partial \omega} \kappa (y) \left( \mu U_R^2 - 2 \sqrt{\frac{ \mu \nu}{R}} U_R V_R + \frac{\nu}{R} V_R^2 \right) \\
	& \geq &  \int_{\partial \omega} \left(\kappa (y) \mu- D\a^2 -  \partial_u g (y,0)\right) \langle U_R\rangle^2 -   \int_{ \omega} \left(d\alpha^2 +  \partial_v f (y,0)\right) \langle V_R\rangle^2     \\
	& &    + \int_{\partial \omega}\kappa (y) \frac{ \nu}{R}  V_R^2  -  2 \int_{\partial \omega} \kappa (y)  \sqrt{\frac{\mu  \nu}{R}}\langle U_R\rangle \langle V_R\rangle   -   \delta_1 (R),
\end{eqnarray*}
where $\delta_1 (R) \to 0$ as $R \to 0$.

Now, assume by contradiction  that there exists a sequence $R_n$ converging to $0$ as $n\to+\infty$ and such that $\liminf_{n \to +\infty}  \| V_{R_n} \|_{L^2 (\omega)} > 0.$ 
Thanks to~\eqref{poinca}, this is equivalent to $\langle V_{R_n} \rangle \geq \varepsilon$ for some $\varepsilon >0$. Replacing $R$ by $R_n$ and continuing the previous computation, notice that all terms are bounded as $n \to +\infty$ except $- 2 \int_{\partial \omega} \kappa R_n^{-1/2} \sqrt{\mu \nu} \langle U_{R_n} \rangle \langle V_{R_n} \rangle$, which is of order $R_n^{-1/2}$, and $\int_{\partial \omega} \kappa R_n^{-1} \nu V_{R_n}^2$. The latter is equivalent to $\int_{\partial \omega}\kappa R_n^{-1} \nu \langle V_{R_n} \rangle^2$, which is larger than $\varepsilon^2 R_n^{-1} \int_{\partial \omega} \kappa \nu$. We conclude that $\widetilde{\mathcal{F}}_{R_n} (\alpha) [U_{R_n},V_{R_n}] \to +\infty$ as $n \to +\infty$, which  contradicts~\eqref{eq:limsup_R0}.

Therefore, we can assume that $ \|V_{R} \|_{L^2 (\omega)} \to 0$, which implies that $ \langle V_R \rangle \to 0$, as $R \to 0$. Thus, not only the left-hand side in~\eqref{eq:4.5} is bounded as $R\to 0$, but it actually converges to~$0$. We then have
\begin{align*}
  \widetilde{\mathcal{F}}_R (\alpha) [U_R ,V_R] & \geq   - \int_{\partial \omega} \left(D\alpha^2+ \partial_u g (y,0) \right) \langle U_R\rangle^2    -  \int_{ \omega} \left(d\alpha^2 +  \partial_v f (y,0)\right) \langle V_R\rangle^2  \\
& \quad  + \int_{\partial \omega}\kappa (y) \left( \sqrt{\mu} \langle U_R \rangle - \sqrt{\frac{\nu}{R}} \langle V_R \rangle \right)^2    -   \delta_2 (R)
\\
& \geq  - \int_{\partial \omega} \left(D\alpha^2+  \partial_u g (y,0)    \right) \langle U_R\rangle^2  -    \int_{ \omega} \left(d\alpha^2 +  \partial_v f (y,0)\right) \langle V_R\rangle^2    -   \delta_2 (R), 
\end{align*}
where $\delta_2 (R) \to 0 $ as $R \to 0$. Using again the fact that $\langle V_R\rangle \to 0$ as $R \to 0$, we get
\begin{equation*}
\liminf_{R \to 0 } \Lambda_R (\alpha)= \liminf_{R \to 0 } \widetilde{\mathcal{F}}_R (\alpha) [U_R ,V_R] = \liminf_{R \to 0 }  \left(-D\alpha^2 - \langle \partial_u g (0)\rangle\right) \langle U_R\rangle^2  \left|\partial \omega \right| .
\end{equation*}
Recalling that, by our normalization, $\int_{\partial \omega} U_R^2 = 1 -  \| V_R\|_{L^2 (\omega)} \to 1$ as $R \to 0$, and that $U_R \, -  \langle U_R \rangle \to 0$ as $R \to 0$ in $H^1 (\partial \omega)$, we get the desired inequality
$$\liminf_{ R\to 0} \Lambda_R (\alpha) \geq - D\alpha^2 - \langle \partial_u g (0) \rangle.$$
Putting this together with~\eqref{eq:limsup_R0}, we have proved that
\begin{equation}
	\label{eq:4.6}
\lim_{R\to 0}\Lambda_R (\alpha) \to - D \alpha^2 -  \langle \partial_u g (0 ) \rangle \qquad \text{for any $\alpha \geq 0$.}
\end{equation}
Finally, we point out that $\Lambda_R (\alpha)$ is nonincreasing with respect to $\alpha$ (recall that it is an even and concave function), and therefore by Dini's second theorem the convergence is also locally uniform in $\{\a\geq 0\}$.\\

Let us now turn to the consequences of the above limit on the speed $c^* (R)$ as $R \to 0$. Assume first that $\Lambda_R (0) <0$ for small $R >0$. Passing to the limit as $R \to 0$, this immediately implies that $\langle \partial_u g (0)\rangle \geq 0$, and in particular
$c^*_{\langle g \rangle} = 2 \sqrt{D \langle \partial_u g (0)\rangle}$
is well-defined. Conversely, if $\langle \partial_u g(0)\rangle<0$, then $\Lambda_R (0)>0$ for small $R$, which implies extinction by Theorem~\ref{th:spreading}.

We now go back to the case when $\Lambda_R (0) < 0$ for small $R >0$, and prove that $c^* (R) \to c^*_{\langle g \rangle}$ as $R \to 0$. In the case when $\langle \partial_u g (0) \rangle = 0$, then $-\Lambda_R (\alpha) \to  D\alpha^2$ and it is clear that $c^* (R) \to 0$ as $R \to 0$. Consider then the case when $\langle \partial_u g(0) \rangle>0$. Then, it is obvious that
$$c^*_{\langle g \rangle} = \min_{\alpha >0 } \frac{D\alpha^2 + \langle \partial_u g (0)\rangle }{\alpha},$$ 
where the minimum is reached at $\alpha^* = \sqrt{\frac{\langle \partial_u g(0)\rangle}{D}}$. Since
$$c^* (R) = \min_{\alpha >0} \frac{-\Lambda_R (\alpha)}{\alpha} \leq \frac{- \Lambda_R \left(\a^* \right)}{ \a^*},$$
passing to the limsup as $R \to 0$ we get from \eqref{eq:4.6} that
$$\limsup_{R \to 0 } c^* (R)\ \leq \limsup_{R \to 0 } \frac{- \Lambda_R \left(\a^* \right)}{ \a^*}= \frac{D\alpha^{*2} + \langle \partial_u g (0)\rangle }{\alpha^*}= c^*_{\langle g \rangle}.$$ 

The opposite inequality is proved similarly but relies on the locally uniform convergence of $\Lambda_R (\alpha)$ in $\{\alpha\geq 0\}$ as $R \to 0$. Indeed, take $\alpha_R >0$ where the minimum in the definition of $c^* (R)$ is reached. We claim that
$$0 < \liminf_{R \to 0} \alpha_ R \leq \limsup_{R \to 0} \alpha_R < + \infty.$$
This follows from the fact that we have already bounded $c^* (R)$ from above. Indeed, if (up to extraction of a subsequence) $\alpha_R \to 0$ as $R \to 0$, and thanks to the locally uniform convergence of $\Lambda_R (\cdot)$ as $R \to 0$, we would have that
$$\lim_{R \to 0} c^* (R) = \frac{-\Lambda_{R} (\alpha_R)}{\alpha_R} = +\infty.$$
Using the convexity of $-\Lambda_R (\alpha)$ and its locally uniform convergence to $D \alpha^2 + \langle \partial_u g(0)\rangle$, one can also find a contradiction if (up to extraction of a subsequence) $\alpha_R \to +\infty$ as $R \to 0$.

Therefore there exists some large $M>0$ such that $\alpha_R \in [1/M , M ]$ for any small $R$. Up to extraction of a subsequence, it converges to some $\alpha_0$ as $R \to 0$, and it follows that
\begin{equation*}
	\liminf_{R \to 0} c^* (R) = \liminf_{R \to 0} \frac{-\Lambda_R (\alpha_R)}{\alpha_R}= \frac{D\alpha_0^2 + \langle \partial_u g (0)\rangle }{\alpha_0} 
	 \geq \min_{\a>0}\frac{D\alpha^2 + \langle \partial_u g (0)\rangle }{\alpha} = c^*_{\langle g \rangle}.
\end{equation*}
This concludes the proof of Theorem~\ref{th:R0}.
\end{proof}

\subsection{Behaviour of $c^*$ as the size of the domain grows to $+\infty$}

As in the previous subsection, we consider the case of a rescaled domain $\omega_R$ and now let the parameter $R \to +\infty$. In this case the most favorable zone will be selected by the solution in order to maximize the spreading speed. This opens several possibilities in the heterogeneous case. Although a more general situation (involving a spatial dependence as in \eqref{eq:rescaled_f}) could be handled similarly as we show below, for simplicity we restrict ourselves to the case of homogeneous coefficients. That is, we consider \eqref{eq:evol} with
\begin{equation}\label{eq:homogeneous}
f(y,v) \equiv f(v) , \quad \kappa (y) \equiv 1 , \quad \mbox{ and } \ g(y,u) \equiv g(u).
\end{equation}
For $\Lambda_R(0)$ to be negative, at least when $R$ is large (see Proposition  \ref{prop:sufficient}), in this subsection we will assume
\begin{equation}\label{f'0}
f'(0) >0 .
\end{equation}
Thus, Theorem \ref{th:spreading} assures that spreading occurs. In order to state our result, we introduce the speed~$c^*_{\BRR}$ from the earlier papers of Berestycki, Roquejoffre and Rossi~\cite{BRR_plus,BRR_influence_of_a_line}.
More precisely, we recall the original road-field system on a half-space:
\begin{equation}\label{eq:evol_BRR}
\left\{
\begin{array}{ll}
\partial_t v = d \Delta v + f(v), & \quad \mbox{ for } (x',y') \in \R^{N} \times  \R_+^*, \ t >0 ,\vspace{3pt}\\
-d\, \partial_{y'} v_{|{y'=0}} =   \mu u - \nu v_{|{y'=0}} , & \quad \mbox{ for } x' \in \R^{N} , \ t >0 ,\vspace{3pt}\\
\partial_t u = D \Delta u + g(u) + \nu v_{|{y'=0}} - \mu u, & \quad \mbox{ for } x' \in \R^{N} , \ t>0.\vspace{3pt}\\
\end{array}
\right.
\end{equation} 
Actually, references~\cite{BRR_plus,BRR_influence_of_a_line} only dealt with the case $N=1$. However, as was shown in~\cite{RTV}, their results extend to an arbitrary dimension. Notice that in~\eqref{eq:evol_BRR}, the field/bulk is the half-space $\{ (x', y') \in \R^{N} \times \R \, | \ y ' > 0 \}$, while in~\eqref{eq:evol} the field/bulk is $\{ (x,y ) \in \R \times \R^N \, | \ x \in \R \mbox{ and } y \in  \omega \}$. For this reason, we use a different notation for the spatial variables in order to avoid any confusion.

The authors of~\cite{BRR_plus,BRR_influence_of_a_line} characterized the spreading speed of solutions $c^*_{\BRR}$, defined in a similar sense to the one of our Theorem~\ref{th:spreading}. Moreover, this speed was shown to satisfy 
$$c^*_{\BRR } 
\left\{
\begin{array}{ll}
= 2 \sqrt{d f'(0)} & \mbox{ if } \ D \leq D^*, \vspace{3pt}\\
> 2 \sqrt{d f'(0)} & \mbox{ if } \ D > D^*,
\end{array}
\right.
$$
where
\begin{equation}\label{BRR_threshold} 
D^* = d \left( 2  - \frac{g'(0)}{f'(0)} \right).
\end{equation}
In some sense, this means that the solution selects the most favorable zone, which may be either far from the road or in its vicinity. In the former case, the spreading speed is equal to $2 \sqrt{d f'(0)}$ which is precisely the speed of the field equation with no road. In the latter case, which occurs when $D$ is large enough, no explicit formula is available for the speed, but, through a geometrical characterization, it is shown to be strictly larger than the usual KPP speed and to depend on all the parameters from the equation on the road.
 
When the domain $\omega$ is bounded but very large, a similar situation occurs for problem~\eqref{eq:evol}. Indeed we show the following:
\begin{theo}\label{th:Rinfty}
Under the above assumptions, i.e., \eqref{eq:homogeneous} and~\eqref{f'0}, we have that
$$c^* (R) \to c^*_{\BRR}, \qquad \text{as $R \to +\infty$.}$$
\end{theo}
Let us immediately prove this theorem. The difficulty here is that, since the limiting problem is posed in an unbounded domain, it admits no equivalent to Proposition~\ref{prop:main}. Nonetheless, a notion of generalized principal eigenvalue can be introduced, which allows us to characterize $c^*_{\BRR}$.
\begin{proof}
	\emph{Step 1. Characterization of $c^*_{\BRR}$ in terms of a family of
		generalized principal eigenvalues.}
By adapting the arguments of~\cite{GMZ}, it can be shown that the spreading speed $c^*_{\BRR}$ of solutions of~\eqref{eq:evol_BRR} can characterized as
\begin{equation}
	\label{eq:c^*_BRR_min}
c^*_{\BRR}  = \min_{\alpha >0} \frac{-\Lambda_{\BRR} (\alpha)}{\alpha},
\end{equation}
where $\Lambda_{\BRR} (\alpha)$ is a concave and even function of $\alpha$ such that there exists a positive solution $(U,V) \in \mathbb{R}\times (C^2 (\mathbb{R}_+^*) \cap C^1 (\mathbb{R}_+) )$ of
\begin{equation*}\label{eq:eigen_BRR}
\left\{
\begin{array}{rcl}
-\left(D\alpha^2 +  g'(0)-  \mu\right)  U - \nu V (0) & = & \lambda U,  \vspace{3pt}\\
-d V''  (y') -  (d\alpha^2 +  f'(0) ) V (y')  & = & \lambda V (y'), \quad \mbox{ for } y' \in \mathbb{R}_+^*,\vspace{3pt}\\
-d V' (0)-\left(\mu U - \nu V (0)\right) & =&  0  , \vspace{3pt}\\
\end{array}
\right.
\end{equation*}
if $\lambda = \Lambda_{\BRR} (\alpha)$, but no positive solution exists if $ \lambda > \Lambda_{\BRR} (\alpha)$. Observe that, unlike in the case of a bounded domain, this is a generalized principal eigenvalue and therefore we do not have uniqueness a priori of the eigenvalue associated with a positive eigenfunction.

More precisely, $\Lambda_{\BRR} (\alpha)$ is constructed as the limit as $L \to +\infty$ of the principal eigenvalue of the following truncated problem:
\begin{equation}\label{eq:eigen_BRR_trunc1}
\left\{
\begin{array}{rcl}
-\left(D\alpha^2 +  g'(0)-  \mu \right) U - \nu V (0) & = & \Lambda_{\BRR,L} (\alpha) U,  \vspace{3pt}\\
-d  V'' (y') -\left(d\alpha^2 + f'(0) \right) V  (y') & = &  \Lambda_{\BRR,L} (\alpha) V (y'), \quad  \mbox{ for } y' \in (0,L),\vspace{3pt}\\
 -d V' (0 )-\left(\mu U - \nu V (0)\right)  & =& 0  , \vspace{3pt}\\
V (L) & = & 0.
\end{array}
\right.
\end{equation}
Here, the domain is bounded so that the notion of a principal eigenvalue can be understood in the classical sense, i.e., it is the smallest eigenvalue and the associated eigenfunction pair is positive. We refer to~\cite{GMZ} for a rigorous construction of the above eigenvalues as well as a spreading result in the spatially periodic framework, which includes the homogeneous system as a particular case (once, again, the restriction on the dimension in~\cite{GMZ} could be lifted).

Using \eqref{eq:c^*_BRR_min}, in order to prove the theorem it is enough to show that $\Lambda_R(\alpha)$ converges to $\Lambda_{\BRR}(\alpha)$ as $R\to+\infty$, locally uniformly in $\{\a>0\}$. To do so, we consider, in a similar fashion as above (see also~\cite{BDR19} for the general heterogeneous case but with~$\alpha = 0$), the following principal eigenvalue problem with an alternative truncation: 
\begin{equation}\label{eq:eigen_BRR_trunc2}
\left\{
\begin{array}{rcll}
-D \Delta U  - \left(D\alpha^2 +  g'(0)-  \mu \right) U - \nu V_{|{y'=0}} & = & \Lambda_{\BRR,L,L'} (\alpha) U ,  & \quad \mbox{ in } B_{L'},\vspace{3pt}\\
- d \Delta  V   + \left(d\alpha^2 +  f'(0) \right) V & = &  \Lambda_{\BRR,L,L'} (\alpha) V,& \quad \mbox{ in }  B_{L'} \times (0,L),\vspace{3pt}\\
-d\, \partial_{y'} V_{|{y'=0}}-\left(\mu U - \nu V_{|{y'=0}}\right) & =&   0  , & \quad \mbox{ in } B_{L'} ,\vspace{3pt}\\
U_{| \partial B_{L'}}  \equiv V_{| \partial B_{L'} \times [0,L] } \equiv V (\cdot, L) & \equiv & 0,
\end{array}
\right.
\end{equation}
where $L,L'>0$ and $B_{L'} \subset \R^{N-1}$ denotes the ball of radius~$L'$ and centered at 0. Notice that the first variable $x'$ now belongs to $\R^{N-1}$, hence the spatial dimension has been reduced compared with the original evolution problem~\eqref{eq:evol_BRR}. This comes from the fact the first component of $x'$ in~\eqref{eq:evol_BRR} stands for the direction of the propagation and disappears in the eigenvalue problem on a cross section.

Then, one can find that
\begin{equation*}
\lim_{L' \to +\infty} \Lambda_{\BRR,L, L'} (\alpha) = \Lambda_{\BRR ,L} (\alpha).
\end{equation*}
In order to prove this relation, one must first use a strong maximum principle to infer that $\Lambda_{\BRR,L,L'} (\alpha) \geq \Lambda_{\BRR,L} (\alpha)$ and also that $\Lambda_{\BRR,L,L'}$ is decreasing with respect to $L'$. In particular, $\lim_{L ' \to +\infty } \Lambda_{\BRR,L,L'} (\alpha)$ exists and it is larger than or equal to $\Lambda_{\BRR,L} (\alpha)$. The opposite inequality can be found by taking a limit of the positive eigenfunction of \eqref{eq:eigen_BRR_trunc2}, up to some shift around the maximum point and a suitable normalization, to find a positive and bounded eigenfunction in the whole strip. By another strong maximum principle, one can deduce that $\lim_{L' \to +\infty} \Lambda_{\BRR,L,L'} (\alpha) \leq \Lambda_{\BRR,L} (\alpha)$. We omit the details of this relatively standard argument.

To sum up this part, we have that
\begin{equation}\label{eq:approximating_eig}
	\Lambda_{\BRR} (\alpha) = \lim_{L \to +\infty} \Lambda_{\BRR,L} (\alpha) = \lim_{L \to +\infty} \lim_{ L' \to +\infty} \Lambda_{\BRR ,L,L'} (\alpha).
\end{equation}

\emph{Step 2. Construction of the limit of $\Lambda_R(\alpha)$ as $R\to+\infty$.}
Now, we fix $\alpha \geq 0$ and take an eigenfunction pair $(U_R,V_R)$ associated with~$\Lambda_R (\alpha)$. Here, we normalize it so that 
$$\max \left\{\max_{\partial \omega_R} U_R , \max_{\partial \omega_R} V_R \right\} = 1.$$
Recall that $\Lambda_R (\alpha)$ can be characterized as in~\eqref{eq:R_eigen1}. In particular, it can be bounded from above uniformly with respect to $R$, by taking the ansatz $(U,V)=\gamma\left(1,\sqrt{\frac{R \mu}{\nu}}\right)$ with $\gamma = \left(|\partial \omega| + R \frac{\mu}{\nu} |\omega | \right)^{-1/2}$.
Noticing that the terms depending on $R$ in the definition of $\widetilde{\mathcal{F}}_R (\alpha)$ are nonnegative, we then conclude that $\Lambda_R (\alpha)$ can be bounded uniformly with respect to $R$ also from below.

Next, since $\partial \omega_R = R \partial \omega$ with $\partial \omega$ a bounded hypersurface, we can choose a subsequence $R_n \to +\infty$ such that $\Lambda_{R_n} (\alpha)$ converges to some value, denoted by $\Lambda_\infty(\alpha)$,
as $n \to +\infty$, as well as
$$\max \left\{\max_{\partial \omega_{R_n} } U_{R_n} , \max_{\partial \omega_{R_n}} V_{R_n} \right\} = \max \left\{ U_{R_n} (R_n y_n), V_{R_n} (R_n y_n) \right\},$$
with $y_n \in \partial \omega$ converging to some $y_\infty \in \partial \omega$ as $n \to +\infty$.

Recall also that $\partial \omega$ is a smooth hypersurface. Then, by standard elliptic estimates and up to extraction of another subsequence, we can find that $V_{R_n} (y+ R_n y_n)$ and $U_{R_n} (y + R_n y_n)$ converge to some functions $V_\infty (y)$ and $U_\infty (y)$ respectively defined in the half-space $\{y \cdot n_\infty \leq 0\}$ and the hyperplane~$\{y \cdot n_\infty  = 0\}$ of~$\mathbb{R}^N$, where $n_\infty$ is the outer unit normal vector of $\omega$ at $y_\infty$. Then, making the change of variables
$$y' = - y \cdot n_\infty \mbox{ and } x' = y +  y' n_\infty,$$
we see that the pair $\left(U_\infty,V_\infty\right)$ satisfies the system
$$
\left\{
\begin{array}{rcll}
-D \Delta U_\infty - \left(D\alpha^2 +  g'(0)-  \mu \right) U_\infty - \nu V_{\infty}|_{y'=0} &=&\Lambda_\infty(\alpha) U_\infty,  & \quad \mbox{ in } \mathbb{R}^{N-1},\vspace{3pt}\\
- d \Delta V_\infty  - \left(d\alpha^2 +  f'(0) \right) V_\infty  &= &\Lambda_\infty(\alpha) V_\infty,& \quad \mbox{ in }  \mathbb{R}^{N-1} \times \R^*_+,\vspace{3pt}\\
-d\, \partial_{y'} V_{\infty}|_{y'=0} -\left(\mu U_\infty - \nu V_{\infty}|_{y'=0}\right) &= &  0  , & \quad \mbox{ in } \R^{N-1} .\vspace{3pt}\\
\end{array}
\right.
$$
Moreover, we have by construction that, in these new variables,
$$\max \{ U_\infty (0) , V_\infty (0,0) \}= 1.$$
We also have that $U_\infty,V_\infty \geq 0$ and, thanks to a strong maximum principle and Hopf lemma, the strict inequalities $U_\infty>0 $ in $\R^{N-1}$ and $V_\infty >0$ in~$\R^{N-1} \times [0,+\infty)$ hold true.

\emph{Step 3. Proof of $\Lambda_\infty(\alpha) \leq \Lambda_{\BRR} (\alpha).$}
Since~$U_{\BRR,L,L'}$ and $V_{\BRR,L,L'}$, the eigenfunctions of~\eqref{eq:eigen_BRR_trunc2}, have compact support for any positive $L,L'$, this, together with the properties of~$\left(U_\infty,V_\infty\right)$ established at the end of the previous step,  allows us to define
$$\theta^* := \sup \{ \theta \geq 0  \ | \ \theta \left(U_{\BRR,L,L'} , V_{\BRR,L,L'}\right) < (U_\infty, V_\infty) \} \in (0,+\infty).$$
Moreover, necessarily $\left(U_\infty - \theta^* U_{\BRR,L,L'}, V_\infty - V_{\BRR,L,L'}\right)$ is nonnegative and there exists a point where at least one of the components vanished. Assume for instance that there exists some point $x_0 ' \in \mathbb{R}^{N-1}$ such that $U_\infty (x_0 ') =  \theta^* U_{\BRR,L,L'} (x_0 ')$. Since $U_\infty >0$, we have that $x_0 '$ belongs to the interior of the support of $U_{\BRR,L,L'}$, and, evaluating both equations at this point, we find that
\begin{eqnarray*}
0 & \geq & -D \Delta \left(U_\infty - \theta^* U_{\BRR,L,L'}\right) (x_0 ') - \left(D\alpha^2 +  g'(0)-  \mu \right) \left(U_\infty - \theta^*  U_{\BRR,L,L'}\right) (x_0 ') \\
& = &  \nu\left( V_\infty (\cdot,0) - \theta^* V_{\BRR,L,L'}\right) (x_0 ') + \Lambda_\infty(\alpha) U_\infty (x_0 ')-  \theta^* \Lambda_{\BRR,L,L'} (\alpha) U_{\BRR,L,L'} (x_0 ') \\
& \geq & \left(\Lambda_\infty(\alpha) - \Lambda_{\BRR,L,L'} (\alpha) \right) U_\infty (x_0 ').
\end{eqnarray*}
Therefore
$$\Lambda_\infty(\alpha) \leq \Lambda_{\BRR,L,L'} (\alpha).$$
Using a similar argument, one reaches the same conclusion if $U_\infty - \theta^* U_{\BRR,L,L'}$ is positive everywhere and $V_\infty - \theta^* V_{\BRR,L,L'}$ vanishes at some point.

Since $L$ and $L'$ were arbitrary positive constants, we conclude, thanks to \eqref{eq:approximating_eig}, that
$$\Lambda_\infty(\alpha) \leq \Lambda_{\BRR} (\alpha).$$

\emph{Step 4. Proof of $\Lambda_\infty(\alpha) \geq \Lambda_{\BRR} (\alpha).$} We proceed by contradiction and assume that
\begin{equation}\label{eq:Rinfty_contrad}
\Lambda_\infty(\alpha) < \Lambda_{\BRR} (\alpha).
\end{equation}
Let us first show that
\begin{equation}\label{eq:Rinfty_step0}
\| V_\infty \|_{L^\infty \left(\R^{N-1} \times \R_+ \right)} < +\infty.
\end{equation}
We already know by construction that $V_\infty \leq 1$ on $\R^{N-1} \times \{ 0\}$. It follows by standard elliptic estimates that $V_\infty$ is also bounded on $\R^{N-1} \times [0,K]$, for any $K>0$. Therefore, if~\eqref{eq:Rinfty_step0} is not true, then going back to the sequence of eigenfunctions $V_{R_n}$ this means that $\max_{\omega_{R_n}} V_{R_n}$ goes to~$+\infty$ as $n \to +\infty$, and that this maximum is reached at some point $\tilde{y}_n$ such that $d (\tilde{y}_n , \omega_{R_n})$, the distance between $\tilde{y}_n$ and $\omega_{R_n}$, goes to $+\infty$ as $n \to +\infty$. Then, denoting
$$\tilde{V}_n (y ) = \frac{V_{R_n} (\tilde{y}_n +y)}{V_{R_n} (\tilde{y}_n)},$$
and passing to the limit as $n \to +\infty$ thanks to standard elliptic estimates, one finds a positive solution~$\tilde{V}_\infty$ of
$$-d \Delta \tilde{V}_\infty - \left(d \alpha^2 + f'(0)\right) \tilde{V}_\infty = \Lambda_\infty(\alpha) \tilde{V}_\infty, \quad \mbox{ in } \ \R^{N}.$$
By construction, $\tilde{V}_\infty$ also satisfies
$$\max \tilde{V}_\infty = 1 = \tilde{V}_\infty (0).$$
Evaluating the above elliptic equation at 0, we find that
$$-d \alpha^2 - f'(0) \leq \Lambda_\infty(\alpha).$$
However, we recall from Proposition~5.1 in~\cite{GMZ} that 
\begin{equation}\label{eq:Pr5.1GMZ}
	\Lambda_{\BRR} (\alpha) \leq -d \alpha^2 - f'(0).
	\end{equation}
Therefore, we have reached a contradiction with~\eqref{eq:Rinfty_contrad}, hence~\eqref{eq:Rinfty_step0} is proved.

Using again~\eqref{eq:Pr5.1GMZ} and~\eqref{eq:Rinfty_contrad}, we now let $\beta >0$ such that
$$\Lambda_\infty(\alpha) + d\alpha^2 +  f'(0) = - d \beta^2.$$
Then, for any $A >0$ and $B >0$, we have that
$$\overline{V} (t,x',y') := A e^{- d\beta^2 t} + B e^{-\beta y'}$$
satisfies
$$\partial_t \overline{V} - d \Delta \overline{V} - \left(d\alpha^2 + f'(0)\right) \overline{V} = \Lambda_\infty(\alpha) \overline{V} \quad \mbox{ for } (x',y') \in \mathbb{R}^{N-1} \times \mathbb{R}_+^*, \ t >0 .$$
Thanks to \eqref{eq:Rinfty_step0}, we can choose $A$ and $B$ large enough so that $\overline{V}{|_{t=0}} > V_\infty $ for all $(x',y') \in \R^{N-1} \times [0,+\infty)$ and $\overline{V}{|_{y'=0}} > V_\infty{|_{y'=0}}$ for all $t>0$ and $x' \in \R^{N-1}$. Applying a comparison principle and passing to the limit as $t \to +\infty$,
we find that
\begin{equation}\label{eq:Vinfty_exp}
V_\infty (x' ,y') \leq B e^{-\beta y'}.
\end{equation}
Then we recall that the function $V_{\BRR}$ associated with the principal eigenvalue problem~\eqref{eq:eigen_BRR_trunc1} satisfies the linear ODE
$$- d V ''  - \left( d\alpha^2 + f'(0)\right) V = \Lambda_{\BRR} (\alpha) V \qquad \text{for }  y' \in \mathbb{R}_+^*.$$
Since $V_{\BRR}$ is positive, we obtain again condition \eqref{eq:Pr5.1GMZ}; moreover, $V_{\BRR}$ must satisfy 
\begin{equation}\label{eq:Vbrr_exp}
V_{\BRR} (y) \geq V_{\BRR} (0) e^{-\beta_0 y'},
\end{equation}
where $0 \leq \beta_0 < \beta$ is such that $d\beta_0^2 + d \alpha^2 + f'(0) + \Lambda_{\BRR} (\alpha)= 0$. Together with~\eqref{eq:Vinfty_exp}, this allows us to define
$$\eta^* := \sup \{ \eta \geq 0 \ | \ \eta (U_\infty,V_\infty) < (U_{\BRR},V_{\BRR}) \} \in (0,+\infty).$$
Then $\left(U_{\BRR} - \eta^* U_\infty, V_{\BRR} - \eta^* V_\infty\right)$ is nonnegative and there exists a sequence $(x_n ', y_n ') \in \R^{N-1} \times [0,+\infty)$ such that, as $n \to +\infty$,
$$\mbox{either} \ \ \left(U_{\BRR} - \eta^* U_\infty\right) (x_n ') \to 0 \ \mbox{ or } \ \left(V_{\BRR} - \eta^* V_\infty\right)  (x_n ',y_n ') \to 0.$$
In the latter case, using again~\eqref{eq:Vinfty_exp} and~\eqref{eq:Vbrr_exp}, we find that the sequence $y_n '$ must be bounded. If the sequence $(x_n ')_n$ is also bounded, then we find some point $(x_\infty ',y_\infty ') \in \mathbb{R}^{N-1} \times [0,+\infty)$ such that 
$$\mbox{either} \ \ \left(U_{\BRR} - \eta^* U_\infty\right) (x_\infty ') = 0 \ \mbox{ or } \ \left(V_{\BRR} - \eta^* V_\infty\right)  (x_\infty ',y_\infty ')  =0.$$
In both cases, by a strong maximum principle one reaches a contradiction with~\eqref{eq:Rinfty_contrad}.

The case when $(x_n ')_n $ is not bounded can be treated similarly, up to taking yet another shifting sequences $U_n (x) =\left( U_{\BRR} - \eta^* U_\infty\right) (x + x_n ')$ and $V_n (x,y) = \left( V_{\BRR} - \eta^* V_\infty\right) (x+x_n ',y)$ beforehand.

In any case, we conclude that $\Lambda_\infty(\alpha) = \Lambda_{\BRR} (\alpha)$. Since this limit does not depend on the choice of the sequence $R_n $, we find that
$$\Lambda_R (\alpha) \to \Lambda_\infty(\alpha) \ \mbox{ as } \ R \to +\infty.$$
Thanks to Dini's theorem, we also know that this convergence is locally uniform with respect to $\alpha$. Proceeding as in the end of the proof of Theorem~\ref{th:R0}, it is now straightforward that $c^* (R) \to c^*_{\BRR}$ as $R \to +\infty$. 
\end{proof}

\section{Numerical approximation of the eigenvalue problem and the spreading speed}\label{sec:numeric}

The objective of this section and the next one is to investigate numerically the behaviour of the spreading speed, in particular with respect to changes made to the geometry of the domain $\omega$. In order to achieve this, the following issues first need to be addressed:
\begin{enumerate}[label=(\roman*),itemsep=0pt,topsep=0pt]
	\item solving  the eigenvalue problem \eqref{coupled-Eig-Pb} numerically for fixed $\alpha$ and $\omega$;
	\item finding the $\alpha$ which minimizes $-\Lambda(\alpha)/\alpha$ for a fixed domain $\omega$, which, thanks to~\eqref{spreading-speed}, allows us to compute the spreading speed.
\end{enumerate}

Once this is done, we will validate our numerical method by comparing our results with those obtained analytically in~\cite{RTV}. There, the case of a cylinder with spherical section and homogeneous coefficients has been considered, and the spreading speed has been characterized not by means of the eigenvalue approach followed in this work, but as a solution of a finite dimensional system. We will restrict ourselves to the case of a two dimensional section~$\omega$, though our approach also works in higher dimensions.

Finally, we will conclude this section by giving a numerical counterexample to the monotonicity of the spreading speed with respect to diffusion, which completes our discussion in Section~\ref{sec:D_counter}.

\subsection{Numerical resolution of the coupled eigenvalue problem}

Here we explain how we find numerically the solutions of the eigenvalue problem~\eqref{coupled-Eig-Pb} when all parameters are fixed. 
First, let us note that the unknown functions $U$ and $V$ are in different spaces: $U$ is defined on $\partial \omega$, while $V$ is defined in the whole~$\omega$. In order to couple these quantities numerically, it is more convenient to define $U$ also in the bulk, which we will do by considering its harmonic extension in~$\omega$.

Therefore, we introduce the functional space of $H^1 (\omega)$ functions whose trace on $\partial \omega$ is also $H^1$:
$$E (\omega) := \{ \varphi \in H^1 (\omega) \,  | \, \varphi_{| \partial \omega} \in H^1 (\partial \omega) \}.$$
Then, our test space will be
$$ E (\omega) \times H^1 (\omega).$$
More precisely, for any $\alpha \geq 0$ we look for $W =(U,V) \in E (\omega) \times H^1 (\omega)$ such that
\begin{equation}
 A_\varepsilon(\alpha)( W , \Psi ) = \lambda_\varepsilon (\alpha) B(W , \Psi) 
 \label{variational-form}
 \end{equation}
for any $\Psi  = (\varphi, \psi)$ in the test space $E (\omega) \times H^1 (\omega)$. Here the bilinear forms $A_\varepsilon(\alpha)$ and $B$ are given by
\begin{align*}
A_\varepsilon(\alpha)(W , \Psi ) & = d\int_\omega \nabla V \cdot  \nabla \psi - \int_{\omega}\left(d\alpha^2 + \f\right)  V \psi  - \int_{\partial \omega}\ka \left( \sqrt{\mu \nu } U - \nu  V \right) \psi   \\  
& \quad - \int_{\partial \omega}\left(D\alpha^2 + \g -\ka\mu\right)U \varphi -\int_{\partial \omega}\ka\sqrt{\mu\nu}  V \varphi   \\
& \quad + D\int_{\partial \omega} \nabla_\tau U \cdot \nabla_\tau \varphi   + \varepsilon\int_\omega \nabla U \cdot \nabla \varphi ,
\end{align*}
where $\varepsilon >0$ will be chosen small, and 
\[ B(W , \Psi ) = \int_{\partial \omega} U \varphi + \int_\omega V \psi . \]
Notice that, since we have extended $U$ to the domain $\omega$, we now distinguish $\nabla U$, the gradient of $U$ as a function of $H^1 (\omega)$, and $\nabla_\tau U$, the gradient of $U$ as a function of $H^1 (\partial \omega)$. Note also that, when $\varepsilon=0$, the term containing the extended part of $U$ in $\omega$ vanishes. In this case, supposing $(U,V) \in H^1(\partial \omega)\times H^1(\omega)$, we obtain the variational formulation associated to~\eqref{coupled-Eig-Pb}. We use the abuse of notation $A_0(\alpha)$ when we refer to this case.  	

It is not difficult to see that this new variational formulation leads to a system similar to~\eqref{coupled-Eig-Pb} with two differences: an additional equation for $U$ in $\omega$, and an additional term in the boundary eigenvalue equation for $U$:
\begin{equation}
\left\{
\begin{array}{rcll}
-D \Delta_\tau  U - \left(D\alpha^2 +  \partial_u g (y,0) - \kappa (y) \mu\right) U - \kappa (y) \sqrt{\mu \nu} V  + \varepsilon \partial_n U & = & \lambda_\varepsilon (\alpha) U, & \quad \mbox{ on } \partial \omega,\vspace{3pt}\\
-d \Delta V - \left(d\alpha^2 + \partial_v f (y,0) \right) V  & = & \lambda_\varepsilon (\alpha) V, & \quad \mbox{ in }  \omega,\vspace{3pt}\\
d \partial_n V-\kappa(y) \left( \sqrt{\mu \nu } U - \nu V \right) & =& 0  , & \quad \mbox{ on } \partial \omega ,\vspace{3pt}\\
-\varepsilon\Delta U & = & 0, & \quad \mbox{ in } \omega .\vspace{3pt}\\
 U>0 \text{ and } V>0,  &  & & \quad \text{ in } \omega .
\end{array}
\right.
\label{coupled-Eig-Pb-numeric}
\end{equation}
By reasoning as in Section~\ref{section2}, it is possible to prove that there exists a unique $\lambda_\varepsilon (\alpha)$ such that the above eigenvalue problem admits a positive eigenfunction, which is also unique up to multiplication by a positive factor. Furthermore, as $\varepsilon \to 0$, the eigenvalue $\lambda_\varepsilon (\alpha)$ converges to~$\Lambda (\alpha)$, and the associated eigenfunction pair also converge (weakly in $H^1 (\partial \omega) \times H^1 (\omega)$ and strongly in $L^2 (\partial \omega) \times L^2 (\omega)$) to that of problem~\eqref{coupled-Eig-Pb}.

With this variational formulation it is possible to use classical finite element methods in order to find numerical solutions of \eqref{variational-form}. In practice, the two dimensional domain $\omega$ is meshed and $\mathbb{P}_1$ or $\mathbb{P}_2$ elements are used in order to build the numerical approximation. Our technical numerical work (meshing, assembly, resolution of the resulting system of equations) has been done with the software FreeFEM~\cite{freefem}. 

\subsection{Efficient computation of the spreading speed}

The spreading speed \eqref{spreading-speed} is given as the minimum of $s(\alpha) := -\Lambda(\alpha)/\alpha$ for $\alpha>0$. By Proposition~\ref{prop:concavity}, we know that this minimum exists and is uniquely reached for some value~$\alpha^*$, provided that $\Lambda (0) < 0$. Due to the concavity of $\Lambda (\alpha)$, any bracketing algorithm will converge linearly to $\alpha^*$. Yet, for each tested $\alpha$, the eigenvalue problem needs to be solved, which is numerically costly. Instead, thanks to Proposition~\ref{prop:concavity}, one can differentiate~$\Lambda$, hence~$s$. This allows us to use a hybrid approach combining a gradient descent with the secant method to achieve a super-linear rate of convergence. More precisely, after one step given by the gradient descent, the next iterate is either computed by the secant method or by another gradient descent step method. The criterion for accepting the iterate given by the secant method is the decrease of the objective function.

For numerical purposes, we actually consider the quantity $s_\varepsilon(\alpha) = -\lambda_\varepsilon(\alpha)/\alpha$, with $\varepsilon = 10^{-6}$ and $\lambda_\varepsilon(\alpha)$ the approximate eigenvalue defined in the previous section. However, the concavity, differentiability and existence of a minimizer for $\alpha \mapsto s_\varepsilon(\alpha)$ follow with the same arguments used for $s(\alpha)$ in Proposition~\ref{prop:concavity}. Therefore, all the above considerations still apply.

\subsection{Some initial numerical experiments}
\label{sec:initial-experiments}

In order to validate the algorithm proposed in this work, we compare the results it provides with those obtained in \cite{RTV}, where the case of the disk, homogeneous $f$ and $g\equiv 0$ was studied analytically. More precisely, in~\cite{RTV}, it was shown that $c^*$ is the smallest value of a parameter for which a certain algebraic system admits a real solution. Using this system, one can compute the spreading speed to machine precision in the aforementioned setting.

To test our algorithm, we have taken the following values of the parameters: $\mu=\nu =1$, $\f \equiv 0.5$ and $\g\equiv0$. We have considered two cases, corresponding to different values of $D$ with respect to $2d$, which is the threshold determined in~\cite{RTV} that separates different behaviours of $c^*$ as a function of the radius of the disk. In the first case, for $d=1$ and $D=1.5$ so that $D<2d$, computations are made for disks of radius $R \in [0.13,50]$ with increment of $0.1$. The computations are made using meshes with approximately $1000$ nodes and $\mathbb P_1$ finite elements. For the case $D>2d$, instead, we have used $d=1$ and $D=3$ and the same radii. This comparison is summarized in Table \ref{tab:CompVariableR}. The precision obtained on $\alpha^*$, the point where $s(\alpha)$ is minimized, is of the same order as the precision obtained for the speed.

\begin{table}
	{ 
		\centering
		
		\begin{tabular}{|c|c|c|c|}
			
			\hline 
			
			Nodes & $h$ & Value of $c^*$  & Rel. error \\ \hline 
			
			928 & 0.11 & 0.9923449724  & 1.5e-4 \\ \hline
			
			3561 & 0.06 & 0.9923066335 & 3.8e-5 \\ \hline 
			
			7974 & 0.04 & 0.9923279467 & 1.7e-5 \\  \hline
			
			14075 & 0.028 & 0.9923353817 & 9.7e-6  \\ \hline
			
			21870 & 0.022 & 0.9923388328 & 6.2e-6 \\
			
			\hline 
			
		\end{tabular} \quad 
		\begin{tabular}{|c|c|c|c|}
			
			\hline 
			
			Nodes & $h$ & Value of $c^*$  & Rel. error \\ \hline 
			
			923 & 0.1 & 1.287934003  & 1.1e-4 \\ \hline
			
			3548 & 0.05 & 1.288045457 & 2.9e-5 \\ \hline 
			
			7994 & 0.036 & 1.288066083 & 1.3e-5 \\  \hline
			
			14068 & 0.032 & 1.288073276 & 7.2e-6  \\ \hline
			
			21828 & 0.024 & 1.288076614 & 4.6e-6 \\
			
			\hline 
			
		\end{tabular}
		
		\caption{Comparison of spreading speeds with known results in \cite{RTV} in the case of the unit disk. The left table corresponds to $d=1,D=1.5$ and the right one to $d=1,D=3$. The parameter~$h$ denotes the mesh size.}
		
		\label{tab:CompMeshSizes}
	}
\end{table}

\begin{table}

	\centering

	\begin{tabular}{|c|c|}

		\hline

		& Max. relative error for $c^*$ \\

		\hline 

		$d=1, D=1.5$ & 2e-4     \\ \hline 

		$d=1, D=3$ & 2e-4 \\ 

		\hline 

	\end{tabular}

	\caption{Comparison of the numerically computed spreading speeds with known results in~\cite{RTV} in the case of disks of variable radius $R \in [0.13,50]$ for meshes with approximately $1000$ nodes.}

	\label{tab:CompVariableR}

\end{table}

We have also performed a comparison for $R=1$ and different mesh sizes. The results are summarized in Table \ref{tab:CompMeshSizes}, where $h$ denotes the size of the mesh elements. For $d=1$ and $D=1.5$, the numerical values obtained are compared to the value $0.9923449724$ from \cite{RTV}. For $d=1$ and $D=3$, the numerical values obtained are compared to the value $1.288082554$ from~\cite{RTV}. It can be observed that, as the size of the meshes decreases, the values of the speed converge to the values obtained with the methods from \cite{RTV}. This analysis confirms that the method proposed in this work is capable of approximating the spreading speed accurately. The main advantage of the new method proposed here is the fact that it is not limited to disks, but can be applied to arbitrarily general shapes.\\

Next, we give a numerical counterexample where the speed is not increasing with respect to~$D$. As we have seen in Section~\ref{section:c^*_D}, such a situation can only occur in the non-radial case. In particular, we have proved that increasing $D$ may lead to extinction. Now, we numerically show  that the monotonicity of the $c^*$ with respect to the diffusion $D$ may not hold true even when the speed is always positive, i.e., when condition~\eqref{instability} holds. 

In this simulation, we first choose $\omega \subset \R^2$ the disk of radius 1. In particular, here we denote the spatial variable $y= (y_1,y_2)$. We also fix the following parameters
$$\kappa \equiv 1 = \mu = \nu = d ,$$
and choose
$$f(y,v) \equiv f(v) := 0.8 \, v \left(1-v\right).$$
The choice of the value 0.8 comes from the fact that the speed associated with the equation in the bulk with the Robin boundary condition
$ \partial_n v = - v$
is negative. In some sense, this means that the solution cannot persist if the component~$u$ on the surface is too small. Finally, we assume the function $g$ to be non-radial:
$$g(y,u) := \left(y_1 - 0.8\right)  u.$$
Due to this heterogeneity, the mass of the surface eigenfunction should be mostly located on some part of the unit circle. Putting these facts together, one may expect that  the term $D \int_{\partial \omega } |\nabla U |^2$ should heavily weigh in the formulas~\eqref{eq:Rayleigh} and~\eqref{eq:F} which are used to compute~$\Lambda$, and this is confirmed numerically. 

First of all, we observe numerically that, as $D \to +\infty$,
$$\Lambda_D (0) \to 1.66 \cdot 10^{-3} >0.$$
Recalling that $\Lambda_D (0)$ is decreasing with respect to $D$ (see Remark~\ref{rmk:several}), this suggests that~\eqref{instability} holds. 

Then, we perform numerical computations of the spreading speed for parameters~$D$ corresponding to a discretization of $[0.1, 100]$ with a step equal to~$0.1$. We find that the speed is always positive and we obtain the results shown in Figure~\ref{fig:counter-example}. It can be clearly noticed that the spreading speed is not monotone with respect to $D$ in this case.

\begin{figure}
	\centering
	\begin{overpic}[width=0.5\textwidth]{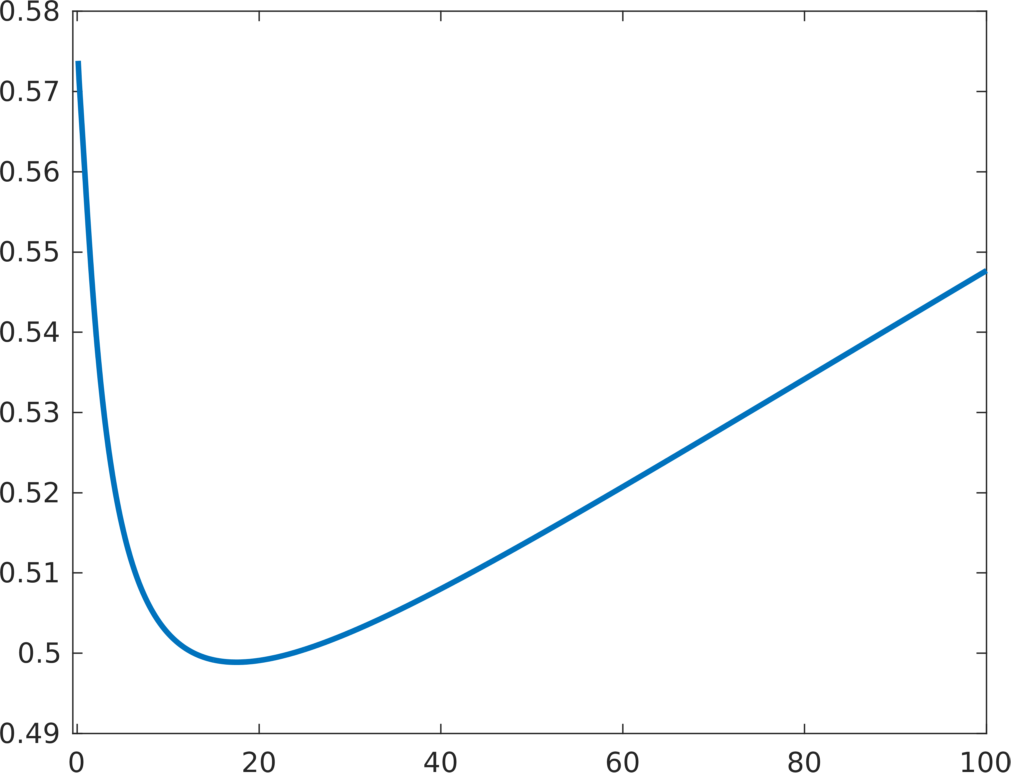}
		\put (101.5,1.5) {$D$}
		\put (-15,73.5) {$c^*(D)$}
	\end{overpic}
	\caption{Numerical counterexample on the monotonicity of the spreading speed with respect to $D$, corresponding to the choice of parameters given in the last part of Section~\ref{sec:initial-experiments}.} 
	\label{fig:counter-example}
\end{figure}

\section{Shape optimization of the spreading speed}\label{sec:shape}

Having established a method which is capable of computing the spreading speed of solutions of~\eqref{eq:evol} for general domains, we can perform a numerical study of the influence of the shape of the domain on this spreading speed. In particular, we will search numerically for sets which optimize the spreading speed under certain given constraints and choices of parameters.

\subsection{The notion of a shape derivative}

Our fundamental tool for searching an optimal shape is the notion of a shape derivative, which we briefly recall. Given a (typically bounded and smooth) shape $\omega\subset\mathbb{R}^N$ and a vector field $\theta \in W^{1,\infty}(\mathbb{R}^N,\mathbb{R}^N)$, one may consider the perturbed domain $(I+\theta)(\omega)$ where $I$ denotes the identity. It can be seen that, for $\|\theta\|_{W^{1,\infty}}$ small enough, this transformation is indeed a diffeomorphism. The shape derivative $J'(\omega)$ of a functional $\omega \mapsto J(\omega)$ is then defined as the Fr\'echet derivative of the mapping $\theta \mapsto J((I+\theta)(\omega))$ at~$\theta=0$.  In other words, it is a bounded linear operator such that
	\[ J((I+\theta)(\omega)) = J(\omega)+J'(\omega)(\theta) + o\, (\|\theta\|_{W^{1,\infty}}).\] 
We point out that throughout this section, the prime notation will always refer to the shape derivative. We refer to \cite{henrot-pierre-english} and \cite{allaireCO} for more details regarding the concept of shape derivative. 

Classical examples of shape derivatives are bulk and surface integrals of functions which do not depend on the shape. Suppose that $\omega$ is at least of class $C^2$, and for some given functions $v\in H^1(\Bbb{R}^d)$ and $w\in H^2 (\Bbb{R}^d)$, denote
\begin{equation}
 J_1(\omega) = \int_\omega v , \qquad  J_2(\omega) = \int_{\partial \omega} w .
 \label{eq:integrals}
\end{equation}
Then, for  a given vector field $\theta\in W^{1,\infty}$, the shape derivatives of $J_1$ and $J_2$ are given by
\begin{equation}
 J_1' (\omega)(\theta) = \int_{\partial \omega} v \theta\cdot n ,\quad \ J_2' (\omega)(\theta) = \int_{\partial \omega} \left( \pd{ w}{n}+\mathcal H w \right)\theta\cdot n,
 \label{eq:sh-deriv-integrals}
 \end{equation}
where $\mathcal H$ is the mean curvature of $\partial \omega$ (defined as the sum of the principal curvatures). We again refer to \cite{allaireCO} and \cite{henrot-pierre-english} for more details about these classical shape derivative formulas. Notice that the shape derivatives shown above only involve the normal component of the perturbation field $\theta$. This fact is true in general under mild regularity assumptions. A classical result states that when $\omega$ is at least $C^1$ and the shape derivative exists, then for every $\theta \in W^{1,\infty}$ there is a linear form $l$ on $C^1(\partial \omega)$ such that
\[ J'(\omega)(\theta) = l(\theta\cdot n).\]
This form will be called \emph{standard form} in the sequel, and our objective is to obtain the shape derivative of the spreading speed and express it in  standard form. For a proof of this structure theorem, one may consult \cite[Proposition 5.9.1, Theorem 5.9.2]{henrot-pierre-english} and \cite[Theorem~3.5]{delfour-zolesio}.

\subsection{Computing the shape derivative of the spreading speed}
\label{sec:computing_shape_deriv}

In this section, we compute the shape derivative of the spreading speed studied in this work. This will allow us to investigate its dependence on the choice of the domain, and in particular to search numerically for shapes which maximize or minimize it.

To simplify the presentation and the formulas, we assume that
$$\kappa(y)\equiv 1, \quad \g=g'(0) .$$ 
Actually, keeping the dependence on $y$ poses no supplementary technical difficulty in the computation of the shape derivative, the only difference being that, since $\ka$ and $\g$ appear in boundary integrals, when computing the shape derivative, in view of \eqref{eq:sh-deriv-integrals}, normal derivatives of these quantities need to be computed. 

In order to underline the dependence of the spreading speed on the shape, we write $\Lambda (\alpha, \omega)$ and $\lambda_\varepsilon (\alpha, \omega)$ instead of $\Lambda (\alpha)$ and $\lambda_\varepsilon (\alpha)$, as well as $A (\alpha, \omega)$, $A_\varepsilon(\alpha, \omega)$ and $B(\omega)$ instead of, respectively, $A(\a)$, $A_\varepsilon (\alpha)$ and $B$. Then we consider the functionals 
\begin{equation}
 J(\omega) = \min_{\alpha>0} \frac{-\Lambda(\alpha,\omega)}{\alpha},\qquad J_\varepsilon(\omega) = \min_{\alpha>0} \frac{-\lambda_\varepsilon(\alpha,\omega)}{\alpha}. 
 \label{eq:Jnotations}
\end{equation}
In the following we show in detail how to compute the shape derivative of $J_\varepsilon$ which we used in our numerical simulations. As we will see below, the shape derivative of $J$ is then obtained by letting $\varepsilon=0$. 

For any vector field $\theta \in W^{1,\infty} (\R^N , \R^N)$, problems~\eqref{coupled-Eig-Pb} and~\eqref{coupled-Eig-Pb-numeric} on the perturbed domain $\omega_t := (I + t \theta) (\omega)$ can be rewritten as some eigenvalue problems on the original domain~$\omega$ whose parameters depend analytically on $t$ small enough. Thus, it follows from Kato's perturbation theory, as explained in~\cite[Chapter~VII, Section~6.5]{kato-perturbation}, that $t \mapsto \Lambda (\alpha,\omega_t)$ and $t \mapsto \lambda_\varepsilon (\alpha, \omega_t)$ are analytic on a neighborhood of 0. In particular, both $\Lambda (\alpha , \omega)$ and $\lambda_\varepsilon (\alpha, \omega)$ admit shape derivatives evaluated at~$\theta$, which we denote respectively by $\Lambda ' (\alpha , \omega) (\theta )$ and $ \lambda_\varepsilon '( \alpha , \omega) (\theta)$.

Now denote by~$\alpha_{\omega_t}$ (respectively $\alpha_\omega$) the unique value where the minimum in $J_\varepsilon(\omega_t)$ (respectively $J_\varepsilon (\omega)$) is reached. By adapting the proof of Proposition~\ref{prop:concavity}, one may check that $\alpha_{\omega_t}$ is also the unique critical point of $\alpha \mapsto s_\varepsilon (\alpha , \omega_t) := \frac{-\leps (\alpha , \omega_t)}{\alpha}$, i.e., it satisfies
\begin{equation}\label{eq:crit_point}
	\partial_\alpha s_\varepsilon (\alpha_{\omega_t}, \omega_t)=0 \qquad \text{for all $t\sim 0$}.
	\end{equation}
	 Moreover, the derivative $\partial_\alpha s_\varepsilon (\alpha, \omega_t)$ can be computed by a formula analogous to~\eqref{eq:derivL} and, in particular, it is continuously differentiable as another consequence of Kato's theory. The same proof also shows that the second derivative $\partial_{\alpha}^2 s_\varepsilon  (\alpha ,\omega)$ is positive at $\alpha = \alpha_\omega$. Finally, applying the implicit function theorem to $(\alpha, t) \mapsto \partial_\alpha s_\varepsilon (\alpha, \omega_t)$, we obtain that $\alpha_{\omega_t}$ is differentiable with respect to~$t$ at $t=0$.

Then, by differentiating the function $t\mapsto s_{\varepsilon}(\alpha_{\omega_t}, \omega_t)$ at $t=0$ and using \eqref{eq:crit_point}, the shape derivative of $J_\varepsilon (\omega)$ evaluated for $\theta$, which we denote by $J_\varepsilon' (\omega)(\theta)$, is given by
	$$J_\varepsilon' (\omega)(\theta) = \frac{ -  \leps' (\alpha_\omega,\omega)(\theta)}{\alpha_\omega}.$$
	Therefore, it remains to compute the shape derivative of $\lambda_\varepsilon (\alpha, \omega)$ at fixed $\alpha$. This is the purpose of our next theorem. In the following, $\D$ denotes the differential or Jacobian matrix, and $\Hess$ denotes the Hessian of a function.

\begin{theo}
	\label{thm:sh-deriv-coupled-numeric}
The shape derivative of $\omega \mapsto \leps(\alpha,\omega)$ in the standard form, when $\alpha$ is fixed, is given by
    \begin{equation}
    \leps '(\alpha,\omega)(\theta) = \int_{\partial \omega} F_\varepsilon\,\theta\cdot n, \label{eq:shderiv-numeric}
    \end{equation}
    where, considering the solution $(U,V)$ of \eqref{coupled-Eig-Pb-numeric} which satisfies $ \int_{\partial \omega} U^2+\int_\omega V^2$,
    \begin{equation}
    	\begin{split}
    		\label{eq:func-sh-deriv-numeric}
    F_\varepsilon & :=  d|\nabla V|^2 -\left(d\alpha^2+\f\right) V^2 + \nu\left(  \partial_n (V^2)  + \mathcal H V^2\right)
     \\ 
    & \quad -2\sqrt{\mu\nu} 
    \left(U\partial_n V + \mathcal H UV\right) -\left(D\alpha^2+g'(0)-\mu\right)
 \mathcal H U^2  \\
    & \quad + D\left(
    \mathcal H |\nabla_\tau  U|^2 -  2\Hess b \nabla_\tau U\cdot \nabla_\tau U\right)  \\
    & \quad- \leps(\alpha,\omega) V^2-\leps(\alpha,\omega) \mathcal H U^2 +\varepsilon |\nabla U|^2 -2\varepsilon(\partial_n U)^2 ,  
   \end{split}
    \end{equation}
    where $\mathcal{H}$ is the mean curvature of $\partial \omega$ and $b$ the signed distance from $\partial \omega$. Finally, the shape derivative of $J_\varepsilon(\omega)$ is given by
    \[ J'_\varepsilon(\omega)=-\frac{\leps'(\alpha_\omega,\omega)}{\alpha_\omega},\]
    where $\alpha_\omega$ is the unique minimizer of $\alpha \mapsto -\leps(\alpha,\omega)/\alpha$ for $\alpha>0$.
\end{theo}
\begin{proof}
As in the proof of Proposition~\ref{prop:concavity}, following the ideas in \cite{Cea86},it is useful to consider a Lagrangian $\mathcal L : \Bbb R \times \mathcal U_{ad} \times \left(H^2_{loc} (\Bbb{R}^d)  \times H^1_{loc} (\Bbb{R}^d)\right)^2$, where the dependence on the shape~$\omega$ is emphasized by including the shape as a variable belonging to the set of admissible domains~$\mathcal U_{ad}$, consisting of bounded, simply connected domains with $C^3$ boundary. The Lagrangian $\mathcal L$ is defined as follows:
\begin{equation}
	\label{eq:lagrangian}
	 \mathcal L (\lambda,\omega,W,\Psi) = \lambda + A_\varepsilon (\alpha,\omega)(W,\Psi)-\lambda B(\omega)(W,\Psi).
	 \end{equation}
Notice that the first component of each pair of test functions is assumed to be $H^2$, so that its trace on any hypersurface is $H^1$. This can be done without loss of generality because the principal eigenfunction can be shown to possess such a regularity by a bootstrap argument.
We also point out that all variables of this Lagrangian are considered to be independent. 
Evaluating~$\mathcal L$ with $W = (U,V)$, the eigenfunction normalized by $B (\omega) (W,W)= 1$, and $\lambda = \leps(\alpha , \omega)$, we find
\[ \mathcal{L}(\leps(\alpha , \omega),\omega,W, \Psi) = \leps(\alpha,\omega).\]

Using the chain rule, the derivative of the above expression with respect to $\omega$ (we recall that, in this section, we use the prime notation to denote the shape derivative) in the direction given by $\theta$ is equal to
\begin{equation*}
	\leps' (\alpha, \omega)(\theta) = \pd{\mathcal L}{\lambda} \cdot \left( \leps'  (\alpha, \omega)(\theta)\right) +\pd{\mathcal L}{\omega}  (\theta)  +\pd{\mathcal L}{W} \cdot\left( W' (\theta)\right)+\pd{\mathcal L}{\Psi}\cdot \left(\Psi' (\theta)\right),
\end{equation*}
with all the partial derivatives of $\mathcal{L}$ evaluated at $(\leps(\alpha, \omega),\omega,W,\Psi)$. Recall that the shape derivatives $\lambda_\varepsilon '(\alpha,\omega)(\theta)$ and $W'(\theta)$ are indeed well defined.

As before, in order to compute the shape derivative, it is useful to see under what conditions the partial derivatives of this Lagrangian with respect to the other variables vanish. For the partial derivatives of $\mathcal L$ with respect to the variables $W$ or $\Psi$ to vanish, we find the variational formulation of \eqref{coupled-Eig-Pb-numeric}. In other words, the partial derivatives of $\mathcal{L}$ with respect to both $W$ or $\Psi$ vanish if $\lambda= \leps (\alpha, \omega)$ and $W = \Psi = (U,V)$ the eigenfunction pair of \eqref{coupled-Eig-Pb-numeric}. Furthermore, the normalization condition $B(W,W) = 1$ eliminates the derivative of~$\mathcal{L}$ with respect to $\lambda$ as in \eqref{eq:2.18bis}.

In the end, we are left with
\[ \lambda_\varepsilon ' (\alpha, \omega)(\theta) = \frac{\partial \mathcal L}{\partial \omega}(\lambda_\varepsilon (\alpha, \omega),\omega,W,W)(\theta). \]
Therefore, recalling the definition \eqref{eq:lagrangian}, since the operators $A_\varepsilon$ and $B$ are sums of integral terms,  we only need to compute the shape derivatives of integrals of constant (with respect to the geometry), functions for which we can use \eqref{eq:integrals} and \eqref{eq:sh-deriv-integrals}.

The most technical part is the differentiation of the term $T(U) = D\int_{\partial \omega} | \nabla_\tau U |^2$ which appears in $A_\varepsilon (\alpha,\omega) (W,\Psi)$. It is not straightforward, since the tangential gradient $\nabla_\tau$ depends on the normal, thus, implicitly, on the shape $\omega$. It is possible to use the shape derivative of the normal given by $${n}' (\omega)(\theta) = -\nabla_\tau (\theta \cdot n),$$ in order to differentiate this term (for more details see~\cite[Chapter 9]{delfour-zolesio}). If we explicitly write the tangential gradient of $U$, we have
\[ T(U ) =D \int_{\partial \omega}  | \nabla U| ^2  - (\nabla U\cdot n)^2.\]
By elliptic estimates and a bootstrap argument (recall that $U$ is the first component of the principal eigenfunction), it is possible to differentiate the previous expression with respect to $\omega$, which gives
\begin{align*}
	T' (U )(\theta) &= D\int_{\partial \omega} \left(\mathcal H|\nabla_\tau U|^2+\partial_n(|\nabla_\tau U|^2)\right) \theta \cdot n + 2D \int_{\partial \omega} \partial_n U \nabla U \cdot \nabla_\tau (\theta \cdot n) \\
	& = D\int_{\partial \omega} \left(\mathcal H|\nabla_\tau U|^2+\partial_n(|\nabla_\tau U|^2)\right) \theta \cdot n + 2D \int_{\partial \omega} \partial_n U \nabla_\tau U \cdot \nabla_\tau (\theta \cdot n)
	\end{align*}

 In order to further simplify this expression all derivatives of $\theta \cdot n$ should be eliminated and, preferably, all terms involving derivatives of order higher than one should be rewritten (for regularity and computational purposes). An integration by parts on $\partial \omega$ gives the following:
 \begin{equation*}
 \!	\int_{\partial \omega}\!\!\!\! \Delta_\tau U \partial_n U \theta \cdot n  = -\!\int_{\partial \omega}\!\!\!\! \nabla_\tau U \cdot \nabla_\tau (\partial_n U \theta \cdot n )
 	 = -\!\int_{\partial \omega}\!\!\!\! \partial_n U\nabla_\tau U \cdot \nabla_\tau (\theta \cdot n)-\!\int_{\partial \omega}\!\!\!\! \nabla_\tau U\cdot \nabla_\tau \partial_n U \theta \cdot n .
 \end{equation*}
In particular, 
\[\int_{\partial \omega} \partial_n U\nabla_\tau U \cdot \nabla_\tau (\theta \cdot n)= -\int_{\partial \omega} \Delta_\tau U \partial_n U \theta \cdot n-\int_{\partial \omega} \nabla_\tau U\cdot \nabla_\tau \partial_n U \theta \cdot n .\]
 Denoting by $b$ the signed distance function to $\partial \omega$ and noting that, since $\partial \omega$ is regular enough, then $n = \nabla b$ (see \cite[Chapter 9]{delfour-zolesio}), we obtain
\[ \nabla_\tau \partial_n U = \nabla_\tau (\nabla U \cdot n) = \nabla (\nabla U \cdot n)-(\nabla(\nabla U\cdot n)\cdot n)n = \Hess U n +\Hess b\nabla U-  \Gamma n, 
\]
where $\Gamma$ is some $L^1 (\partial \omega)$ function. We do not make the function $\Gamma$ explicit, since it will eventually vanish when multiplied with the tangential gradient.

Combining the relations obtained above gives
\begin{align*}
	\int_{\partial \omega} \partial_n U\nabla_\tau U \cdot \nabla_\tau (\theta \cdot n)  = & -\int_{\partial \omega} \Delta_\tau U \partial_n U\theta \cdot n - \int_{\partial \omega} \nabla_\tau U \cdot \left(\Hess U n +\Hess b \nabla U - \Gamma n\right) \theta \cdot n \\
	= & -\int_{\partial \omega} \Delta_\tau U \partial_n U \theta \cdot n  -\int_{\partial \omega}\left( \Hess U n \cdot \nabla_\tau U + \Hess b \nabla_\tau U\cdot \nabla_\tau U\right)\theta \cdot n .
\end{align*}
On the other hand, the term containing higher order derivatives on $U$ is
\begin{align*}
	\partial_n \left(|\nabla_\tau U|^2\right) & = \nabla \left( |\nabla U|^2-(\partial_n U)^2\right)\cdot n\\
& = \left(2\Hess  U \nabla U-2\partial_n U(\Hess  U n + \D n \nabla U)\right)\cdot n\\
& = 2\Hess U\nabla U\cdot n-2\partial_n U \Hess  U n\cdot n \\
& = 2\Hess U \nabla_\tau U\cdot n  ,
\end{align*}
where we used the symmetry of the matrices $\Hess U$ and $\D n$ (recall that $n=\nabla b$), and the fact that $(\D n)n = 0 $.
Fortunately, by putting the above two computations together, the terms containing the Hessian $\Hess U$ in the expression of the shape derivative of $T$ cancel. 

Finally, the shape derivative of $T$ is given by
\[
T'(U)(\theta) = D \int_{\partial \omega} \left(
\mathcal H |\nabla_\tau  U|^2 - 2\Delta_\tau U \partial_n U - 2\Hess b \nabla_\tau U\cdot \nabla_\tau U\right) \theta \cdot n 
.\]
Using the boundary equation $-D\Delta_\tau U = \leps (\alpha,\omega) U +\left(D\alpha^2+g'(0)-\mu\right)U +\sqrt{\mu\nu} V-\eps \partial_n U$, it is possible to further eliminate the Laplace-Beltrami operator in the expression above. 

One can now compute the shape derivative of $\leps (\alpha ,\omega)$, using again the formulas~\eqref{eq:integrals} and~\eqref{eq:sh-deriv-integrals} to differentiate the other integral terms. We omit the details and, after some straightforward computations, one finds the following result:
\begin{align*}
\leps'  (\alpha, \omega)(\theta) 
& = d\int_{\partial \omega} |\nabla V|^2 \theta \cdot n  -  \int_{\partial \omega} \left(d\alpha^2+\f\right) V^2 \theta \cdot n +\varepsilon \int_{\partial \omega} |\nabla U|^2 \theta \cdot n \\
& \quad + \nu \int_{\partial \omega} \left( 
\partial_n (V^2) + \mathcal H V^2\right) \theta \cdot n  
 -2\sqrt{\mu\nu} \int_{\partial \omega}
\left( \partial_n (UV) + \mathcal H UV\right) \theta \cdot n \\  
& \quad -\left(D\alpha^2+g'(0)-\mu\right) \int_{\partial \omega} 
 \left( \partial_n(U^2)+ \mathcal H U^2\right)  \theta \cdot n \\ 
& \quad +  \int_{\partial \omega} D\left(
	\mathcal H |\nabla_\tau  U|^2- 2\Hess b \nabla_\tau U\cdot \nabla_\tau U   \right) \theta \cdot n  \\
& \quad + \int_{\partial \omega} 2\left(\leps(\alpha,\omega) U+(D\alpha^2+g'(0)-\mu)U+\sqrt{\mu\nu}V-\varepsilon\partial_n U\right) \partial_n U\theta \cdot n \\
& \quad - \leps(\alpha,\omega) \int_{\partial \omega} V^2 \theta \cdot n 
- \leps(\alpha,\omega) \int_{\partial \omega} \left( \mathcal H U^2+\partial_n (U^2)\right) 
 \theta \cdot n .   
\end{align*}
Finally, by grouping similar terms and performing some simplifications, we obtain exactly~\eqref{eq:shderiv-numeric} with~\eqref{eq:func-sh-deriv-numeric}, and the theorem is proved.
\end{proof}
 
The spreading speed for problem \eqref{eq:evol} is obtained by letting $\varepsilon =0$ in the approximate problem \eqref{coupled-Eig-Pb-numeric}. It is indeed possible to redo the same computations in this case, even though there is no information on the behaviour of~$U$ outside $\partial \omega$. Indeed, one may note that in the expression of the shape derivative \eqref{eq:func-sh-deriv-numeric}, when setting $\varepsilon=0$, the terms containing the normal derivative of $U$ and the gradient of the extension of $U$ outside $\partial \omega$ vanish. Therefore, one may consider an arbitrary extension of $U$ when dealing with the tangential gradient, obtaining an analogous result in the end. This is in accordance with the fact that, from a theoretical point of view, the spreading speed and the solution of \eqref{coupled-Eig-Pb} does not depend on the extension of $U$ outside~$\partial \omega$.

\begin{theo}
	\label{thm:sh-deriv-coupled}
The shape derivative of $\omega \mapsto \Lambda(\alpha,\omega)$ in the standard form, when $\alpha$ is fixed, is given by
	\begin{equation}
	\Lambda'(\alpha,\omega) (\theta) = \int_{\partial \omega} F_0\,\theta \cdot n, \label{eq:shderiv-spreading}
	\end{equation}
	where $F_0$ is given by \eqref{eq:func-sh-deriv-numeric} with $\varepsilon=0$, $(U,V)$ being the solution of \eqref{coupled-Eig-Pb} which satisfies $ \int_{\partial \omega} U^2+\int_\omega V^2$.
	
	Therefore, the shape derivative of $J(\omega)$ is given by
	\[ J'(\omega)=-\frac{\Lambda'(\alpha_\omega,\omega) }{\alpha_\omega},\]
	where $\alpha_\omega$ is the minimizer of $\alpha \mapsto -\Lambda(\alpha,\omega)/\alpha$ for $\alpha>0$, and $\Lambda'(\alpha_\omega,\omega)$ is computed using~\eqref{eq:shderiv-spreading} with~$\alpha=\alpha_\omega$.
\end{theo}

\subsection{Numerical optimal shapes}

In this section we apply the results obtained in the previous sections with the aim of numerically finding the shapes which optimize (either maximize or minimize) the spreading speed among simply connected sets in the plane (in particular, we do not consider the addition or removal of holes in the domain $\omega$). 
We also point out that we do not theoretically address  the question of existence of optimal shapes. This appears to be a challenging issue, and our numerical simulations actually suggest that such an optimal shape may or may not exist, depending on the parameters. 

In the numerical computations the approximation $J_\varepsilon$ given by \eqref{eq:Jnotations} is used.
As usual, when performing numerical shape optimization, it is not possible to guarantee that the results given by the numerical algorithms are more than local optimizers of the spreading speed. In the numerical simulations, the approximate eigenvalue problem \eqref{coupled-Eig-Pb-numeric} and the corresponding result of Theorem \ref{thm:sh-deriv-coupled-numeric} are used.

As stated in the previous paragraphs, we choose to work in the class of simply connected domains in the plane. In order to further simplify the numerical treatment, we also suppose that the domains are star-shaped. These domains can be uniquely determined using a function $\rho : [0,2\pi] \to \mathbb{R}_+$ of class $C^3$ (for $\partial\omega$ to be sufficiently regular) such that $\rho(0)=\rho(2\pi)$. Note that the shape derivative \eqref{eq:shderiv-numeric}-\eqref{eq:func-sh-deriv-numeric} contains terms related to the tangential gradient, curvature and Hessian of the signed distance function to the boundary. The explicit parametrization of the boundary by means of $\rho$ allows us to be able to explicitly compute all these quantities needed in the numerical algorithms.  Such a function is discretized by using a truncation of its Fourier series

\[ \rho(\vartheta) \approx a_0+\sum_{k=1}^M a_k\cos(k \vartheta) + b_k \sin(k\vartheta). \]

The shape $\omega$ is thus parametrized using the variables $\bo v :=(a_0,a_1,...,a_M,b_1,...,b_M) \in \Bbb{R}^{2M+1}$. In order to use a generic gradient algorithm, the gradient with respect to each of the variables should be computed. If $\hat r(\vartheta) := (\cos \vartheta,\sin \vartheta)$, then a parametrization of the boundary is given by $(x(\vartheta),y(\vartheta)) = \rho(\vartheta)\hat r(\vartheta)$. It is not difficult to see that the tangent vector corresponding to this parametrization is $\bo T(\vartheta) = (x'(\vartheta),y'(\vartheta)) = \rho'(\vartheta)\hat r (\vartheta)+\rho(\vartheta)\hat r'(\vartheta)$. Since~$\hat r(\vartheta)$ and~$\hat r'(\vartheta)$ are orthogonal, it follows that $\|\bo T(\vartheta)\| = \sqrt{\rho(\vartheta)^2+\rho'(\vartheta)^2}$. Finally, the outer unit normal is given by $\bo N(\vartheta) = \left(\rho(\vartheta)\hat r-\rho'(\vartheta) \hat r'\right) )/\sqrt{\rho(\vartheta)^2+\rho'(\vartheta)^2}$. In order to compute the partial derivatives with respect to the Fourier coefficients, it is enough to evaluate \eqref{eq:shderiv-numeric} for $\theta = \hat r \cos(k\vartheta)$ and $\theta = \hat r\sin(k\vartheta)$, $k \in \{0,1,...,M\}$. When the partial derivatives with respect to all the Fourier coefficients are known, a classical gradient descent algorithm is used in order to search for numerical optimizers. 

With the considerations above, the spreading speed $c^*(\omega)$ is approximated by a function~$\mathcal C^*(\bo v)$ for which $\nabla \mathcal C^*$ can be computed. The classical gradient descent algorithm used is the following:
\begin{itemize}[topsep=0pt,noitemsep]
	\item[$\star$] {\bf Initialization:} Choose initial Fourier coefficients $\bo v_0$
	\item[$\star$] choose the initial step $\delta t>0$ and a stopping criterion $\text{tol}>0$;
	\item[$\star$] {\bf Optimization loop:}  while $\delta t>\text{tol}$
	\begin{itemize}[noitemsep,topsep=0pt]
		\item compute $\mathcal C^*(\bo v_i)$ and $\nabla \mathcal C^*(\bo v_i)$;
		\item define $\bo v_{i+1} = \bo v_i - \delta t \nabla \mathcal C^*(\bo v_i)$;
		\item if a {\bf  constraint}  is imposed, then {\bf project} onto the constraint: $\bo v_{i+1} = \mathcal P(\bo v_{i+1})$.
		\item if $\mathcal C^*(\bo v_{i+1}) < \mathcal C^*(\bo v_i)$, then accept the iterate and eventually increase $\delta t$;
		\item else reject the iterate and decrease $\delta t$.
	\end{itemize}
\end{itemize}

As underlined before, it is not possible to guarantee that the results given by the numerical optima are global optima. In order to avoid as much as possible local optima, the algorithm is executed multiple times with randomly generated initializations.

\subsubsection{The homogeneous case with no reaction on the surface}

We restrict ourselves to the homogeneous case and start with a situation where $f' (0) >0$ and $g' (0) = 0$. More precisely, here we will fix
\begin{equation}
	\label{eq:param1}
f ' (0) = 0.5 >  0 = g' (0), \qquad \text{and} \qquad \kappa (\cdot) \equiv 1 = \mu = \nu .
\end{equation}
The diffusion parameters $d$ and~$D$, and, of course, the shape of the domain $\omega$ may vary.

First of all, we start by briefly discussing the problem of minimizing the speed. With the chosen parameters, according to Proposition~\ref{prop:sufficient} the speed is always positive; yet, as established in Theorem \ref{th:R0}, the speed goes to $0$ when the size of the domain goes to $0$. This means that the infimum of the speed over all domains $\omega$ is not reached.

On the other hand, when dealing with variational eigenvalue problems on disconnected domains, it is classical that the spectrum may be recovered by considering the union of the spectra of all connected components. In particular, if~$\omega_n$ is the union of $n$ disks $B(x_k,R_n)$, $k=1,...,n$ of the same radius $R_n$, and the total area or perimeter is fixed and independent on $n$, then, thanks to Theorem \ref{th:R0}, $c^*(\omega_n)=c^*(B(0,R_n))\to 0$ as $n\to+\infty$, since $R_n\to 0$. This means that, even under an area or perimeter constraint, the infimum of the speed is not reached. In practice, our numerical computations fail because the domain is pushed towards a non star-shaped one. 

After this analysis, at least for this subsection, we will focus on the problem of maximizing the speed with respect to the domain~$\omega$.\\

{\bf Maximizing the speed for disks.} A first case to test our method, in which the optimal shapes are known, is maximizing the speed in the class of disks when $D>2d$. In such a case, it was proven in \cite{RTV} that the maximal spreading speed $c_M$ is obtained for a disk of radius $R^*$, with
\begin{equation}
	\label{eq:R^*_RTV}
	c_M:=\frac{D f'(0)}{\sqrt{(D-d)f'(0)}} \quad \text{and} \quad R^* := \frac{2D\nu}{(D-2d)\mu}. 
\end{equation}
 We performed the optimization process without any additional constraints. The algorithm was tested for $d=1,D=3$, and we obtained a maximal spreading speed of $1.5$ for a radius of $6$, in accordance with \eqref{eq:R^*_RTV}.

When $D \leq 2d$, it was shown in \cite{RTV} that the speed is increasing with respect to $R$ and converges to $2\sqrt{df'(0)}$ when $R\to+\infty$, without reaching its supremum. Again, our numerical observations are consistent with the theory developped in \cite{RTV}, as we see the radius increases and the speed converges to the expected quantity.\\

Let us now concentrate on the cases where area or perimeter constraints are imposed. This will provide some insight to understand our numerical observations in the unconstrained case.\\

{\bf Maximizing the speed under area constraint.} A first natural constraint to be considered is the area constraint. At each iteration, when computing the next iterate, a projection step is performed. The projection chosen here is to simply perform a scaling so that the projected shape has the desired area. 

The first computation deals with the parameters indicated in \eqref{eq:param1}, $d=1$ and $D=1.5$, under the constraint $\text{Area}(\omega) = \pi$. The discretization of the function $\rho$ uses $33$ Fourier coefficients. The tolerance for the descent step was set to 1e-6. Some instances of the shape optimization process are shown in Figure \ref{fig:max-area}. The evolution of the speed is plotted in Figure~\ref{fig:max-area-cost}. The optimization process has $119$ iterations, but multiple finite element computations are realized per iteration in order to find the parameter $\alpha$ which gives the spreading speed. A total of $1549$ finite element computations were necessary for this computation.

The optimal domain obtained in this first case is an approximation of the unit disk. For some other choices of parameters $d$,~$D$, under area constraint, the optimal shape remains the disk of the corresponding area. Some of the computations performed are given below:
\begin{itemize}[noitemsep,topsep=0pt]
	\item $d=1, D=1.5$, $\text{Area}(\omega) \in \{1,10,100,200,500\}$: the numerical solution is close to a disk;
	\item $d=1, D=3$: in this case, the disk of radius $R^*=6$ maximizes the speed among disks. Denoting $A^* := \pi(R^*)^2 = 36\pi$, we have the following: the obtained numerical solution  is the disk for $\text{Area}(\omega) \in \{1,10,A^*-1,A^*-0.1,A^*\}$. However, for $\text{Area}(\omega)>A^*$, the algorithm breaks in a similar way to what happens in the minimization problem, it does not give a definite optimal shape, and the disk of the desired area does not seem to be optimal. A theoretical justification for this behavior is sketched hereafter. 
\end{itemize}
Indeed, in the latter case when $D > 2d$ and $\text{Area}(\omega) >A^*$, then it is clear that a disconnected set, consisting of two disks of areas $A^*$ and $\text{Area} (\omega) - A^*$, respectively, achieves a larger speed than the single disk of area $\text{Area}( \omega)$. Yet such a domain is not achievable by our numerical algorithm. Together with our numerical computations for unconstrained disks described above, this suggests that the maximum of the speeds among domains $\omega$ with $\text{Area}(\omega)>A^*$ is always the speed in the optimal disk. In particular, above this threshold, the maximal speed no longer depends on the constraint.\\

\begin{figure}
	\centering
	\begin{tabular}{ccccc}
	\includegraphics[width=0.18\textwidth]{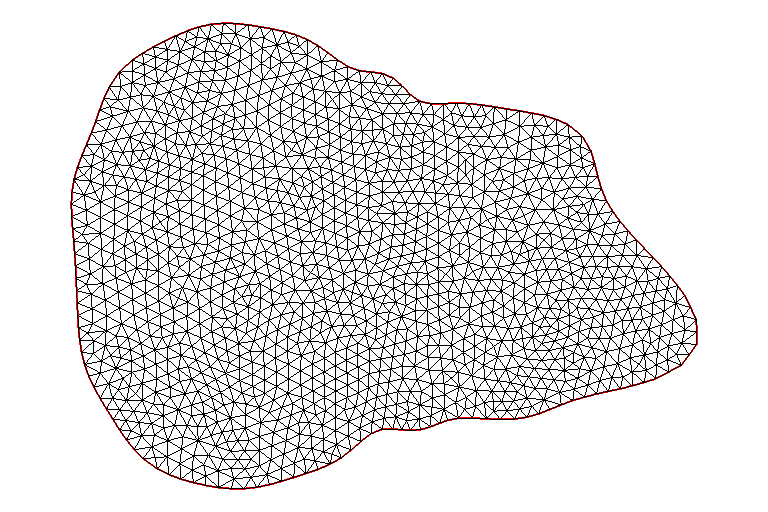} &
	\includegraphics[width=0.18\textwidth]{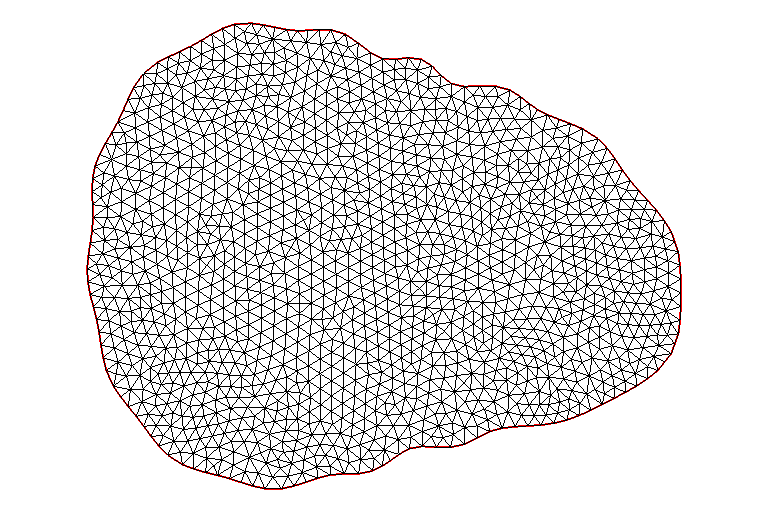}&
	\includegraphics[width=0.18\textwidth]{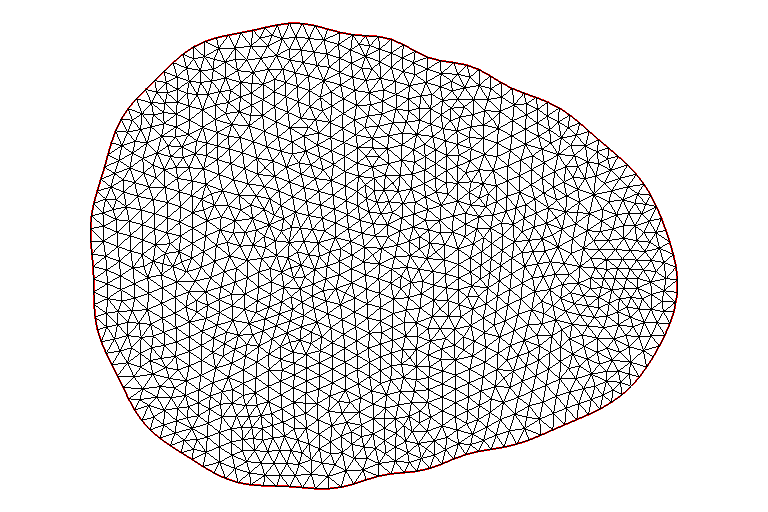}&
	\includegraphics[width=0.18\textwidth]{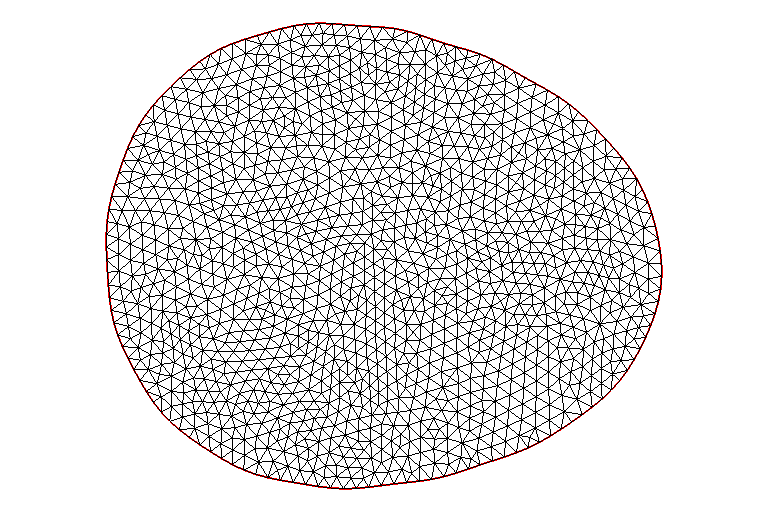}&
	\includegraphics[width=0.18\textwidth]{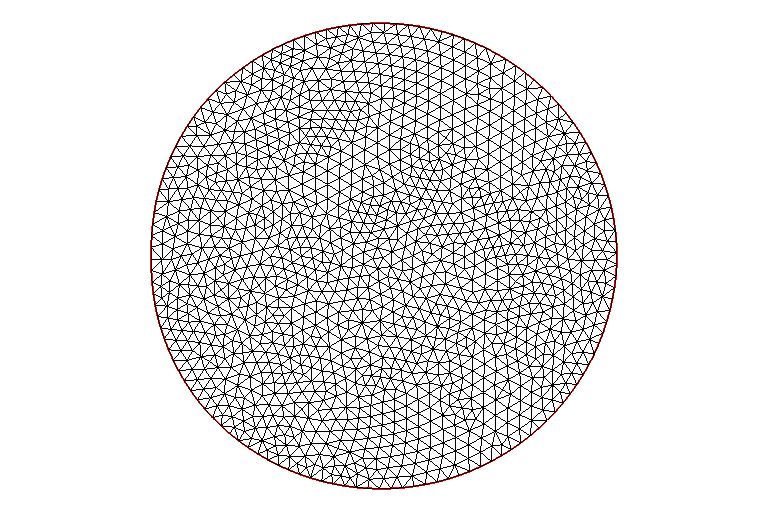}\\
	Iter 1 & Iter 5 & Iter 10 & Iter 20 & Iter 119
	\end{tabular}
\caption{Evolution of the domain during the maximization process of the speed under the constraint $\text{Area}(\omega)=\pi$, for the choice of the parameters \eqref{eq:param1}, $d=1$ and $D=1.5$.}
\label{fig:max-area}
\end{figure}

\begin{figure}
	\centering
	\includegraphics[width=0.5\textwidth]{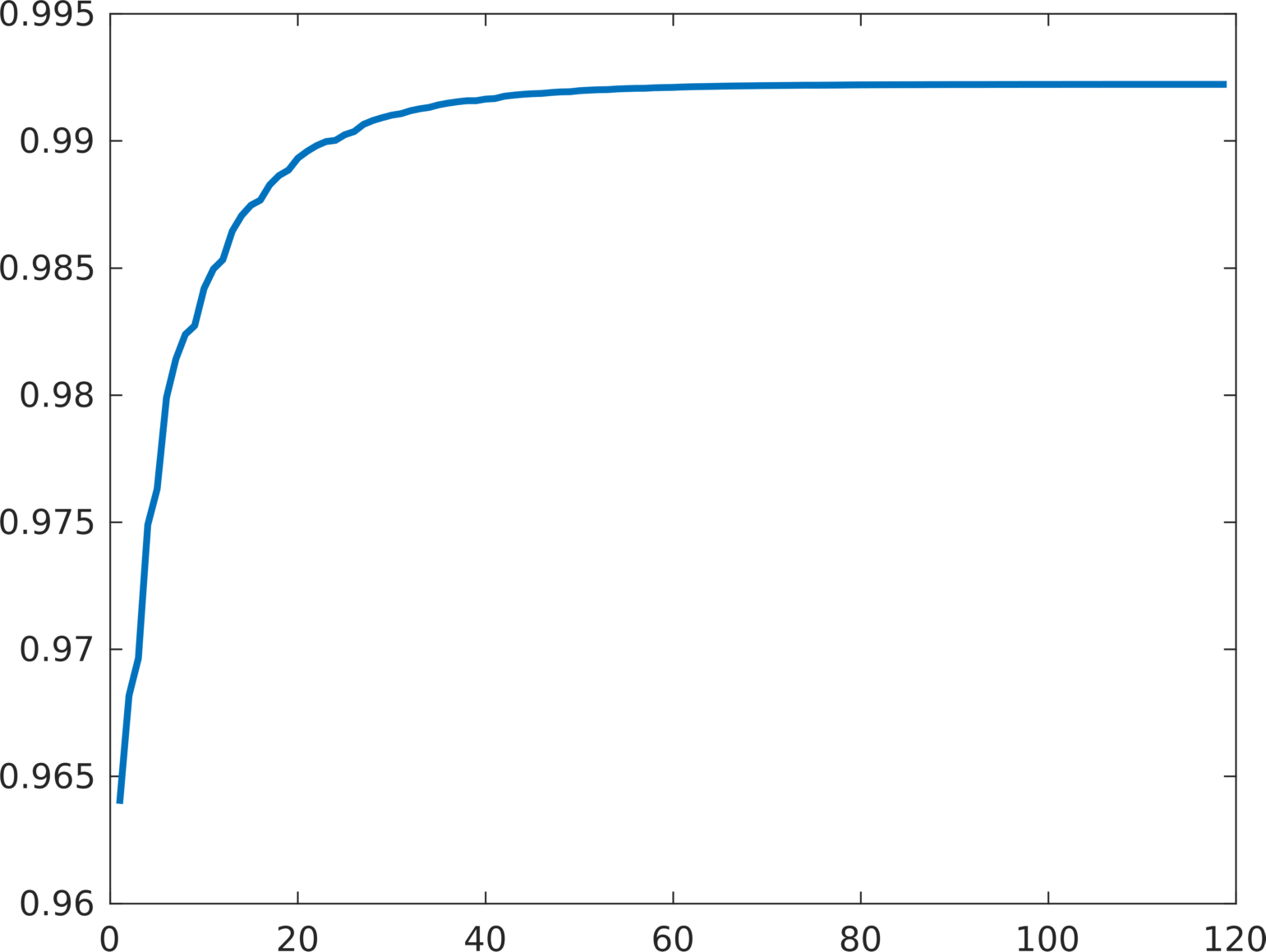}
	\caption{Evolution of the spreading speed during the optimization process: area constraint. The parameters are the same of Figure \ref{fig:max-area}. We plot the number of iteration on the horizontal axis and the corresponding value of the approximated speed on the vertical one. }
	\label{fig:max-area-cost}
\end{figure}

{\bf Maximizing the speed under perimeter constraint.} A second natural constraint is to fix the perimeter of the shape $\omega$. In the same way as for the area constraint, the perimeter constraint is enforced by rescaling the shape at each iteration in order to have the desired perimeter.

As before, we first consider the parameters indicated in \eqref{eq:param1}, $d=1$ and $D=1.5$, and now take the constraint $\text{Perim}(\omega) = 2\pi$. The discretization of the function $\rho$ uses $25$ Fourier coefficients. The tolerance for the descent step was set to 1e-6. Some instances of the shape optimization process are shown in Figure \ref{fig:max-perim}. The evolution of the speed is plotted in Figure~\ref{fig:max-perim-cost}. The optimization process has $86$ iterations, and needs a total of~$2037$ finite element computations.

Similarly to the case of the area constraint, there is a change in behaviour when comparing the constraint with the perimeter of the optimal disk, when it exists. Some of the computations performed are given below:
\begin{itemize}[noitemsep,topsep=0pt]
	\item $d=1, D=1.5$, $\text{Perim}(\omega) \in \{2\pi,4\cdot (2\pi),10\cdot (2\pi)\}$: the numerical solution is close to a disk;
	\item $d=1, D=3$: in this case the disk of radius $R^*=6$ maximizes the speed among disks. Denoting $P^* := 2\pi R^* = 12\pi$, we have the following: the obtained numerical solution  is the disk for $\text{Perim}(\omega) \in \{2\pi,10\pi,P^*-1,P^*-0.1,P^*\}$. However, for $\text{Perim}(\omega)>P^*$, the algorithm does not give a definite optimal shape, and the disk does not seem to be optimal. The same type of theoretical justification as in the case of the area constraint applies.\\
\end{itemize}

\begin{figure}
	\centering
	\begin{tabular}{ccccc}
		\includegraphics[width=0.18\textwidth]{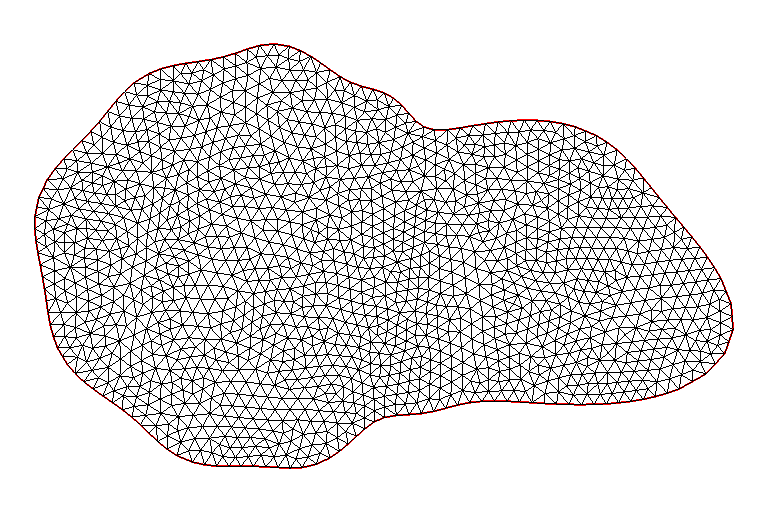} &
		\includegraphics[width=0.18\textwidth]{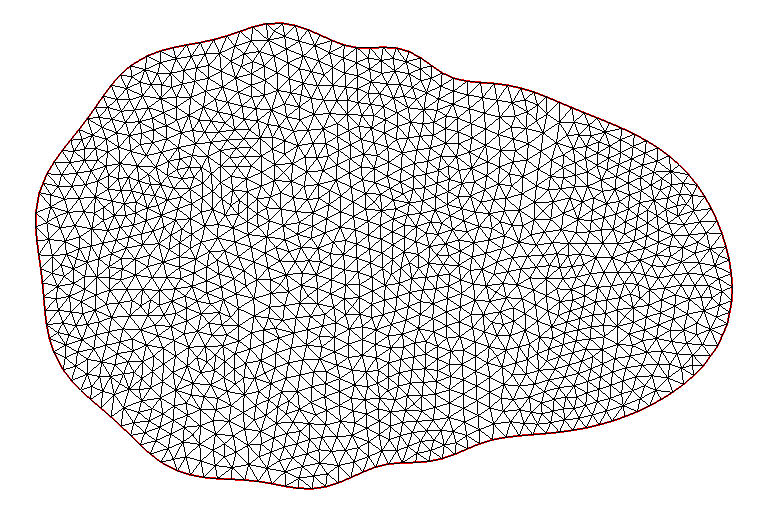}&
		\includegraphics[width=0.18\textwidth]{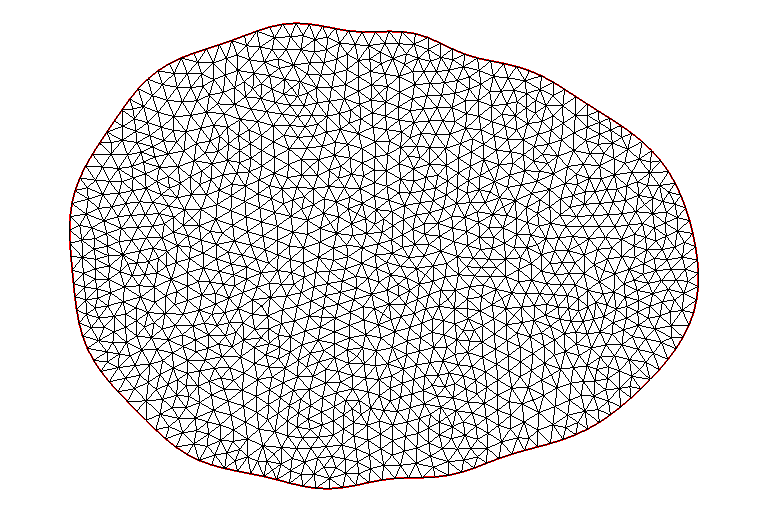}&
		\includegraphics[width=0.18\textwidth]{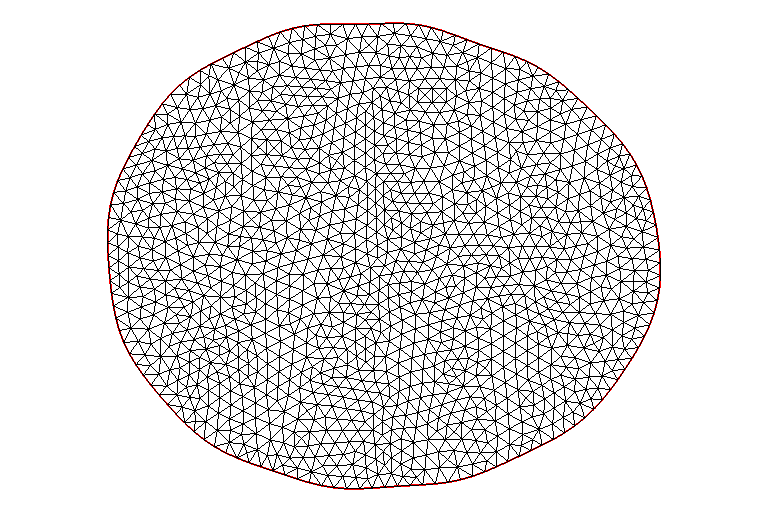}&
		\includegraphics[width=0.18\textwidth]{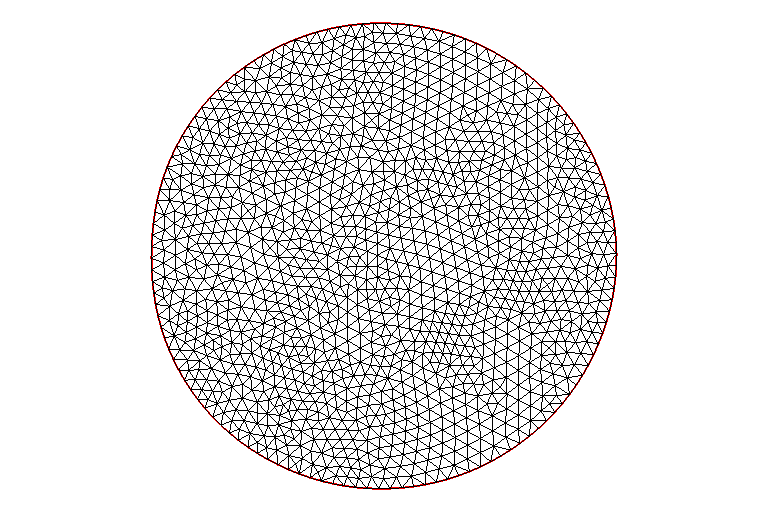}\\
		Iter 1 & Iter 5 & Iter 10 & Iter 20 & Iter 86
	\end{tabular}
	\caption{Evolution of the domain during  the maximization process of the speed under the constraint $\text{Perim}(\omega)=2\pi$, for the choice of the parameters \eqref{eq:param1}, $d=1$ and $D=1.5$.}
	\label{fig:max-perim}
\end{figure}

\begin{figure}
	\centering
	\includegraphics[width=0.5\textwidth]{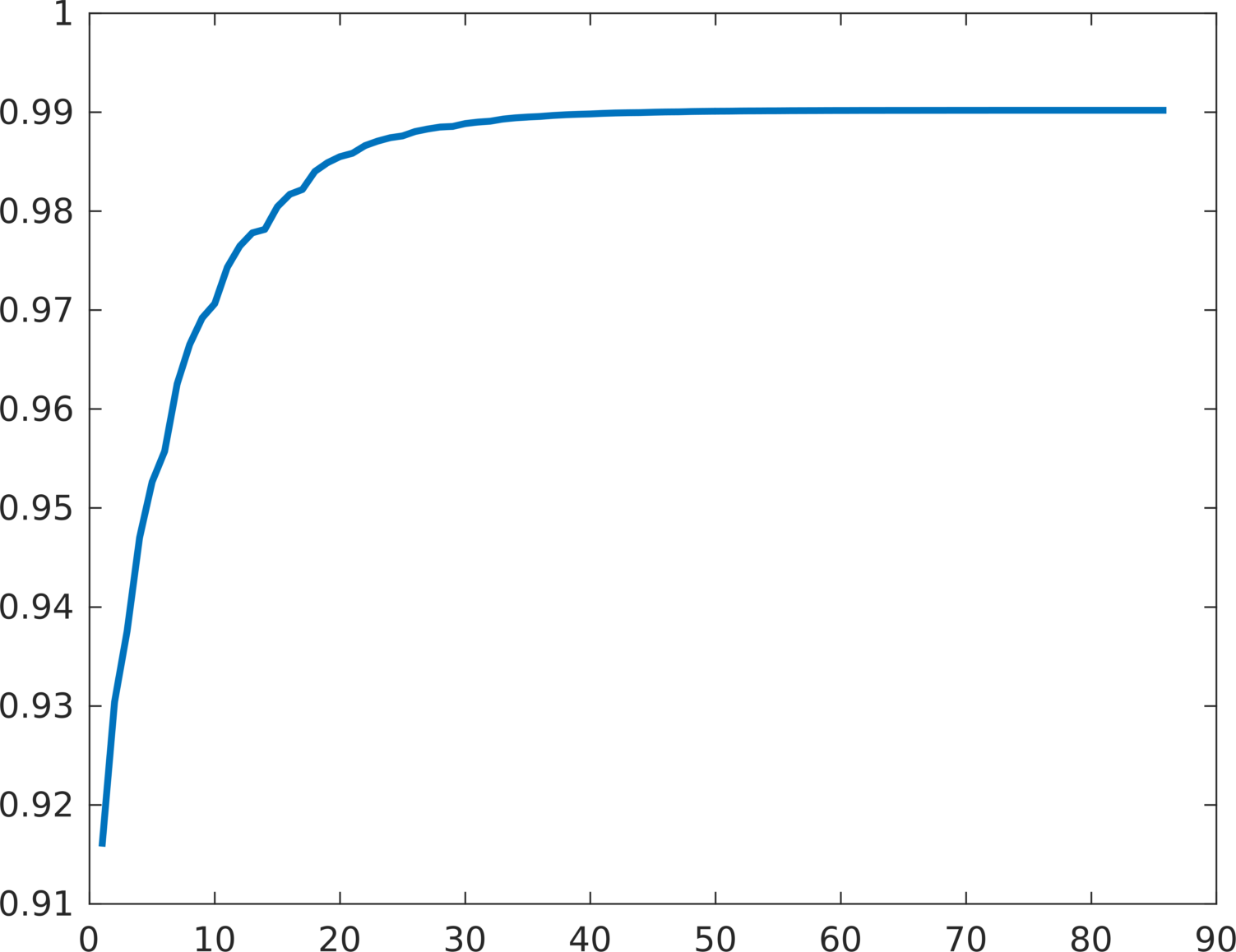}
	\caption{Evolution of the spreading speed during the optimization process: perimeter constraint. The parameters are the same of Figure \ref{fig:max-perim}. We plot the number of iteration on the horizontal axis and the corresponding value of the approximated speed on the vertical one. }
	\label{fig:max-perim-cost}
\end{figure}

Summing up the above findings, we conjecture that the disk is, in some sense, always a maximizer for the speed. However, if the constraint (either on the area or the perimeter) is too large, an optimal (but not admissible) shape would actually be the disconnected union of two disks.

We observe, in practice, that variation of the spreading speed with respect to some shapes appears to be very small, so that it is not clear from numerics whether the maximal speed could also be reached by a connected domain.

Let us now turn to the case of maximization without constraints.\\

{\bf Maximizing the speed in the unconstrained case.} 
The next question concerns the case where the shape is free but unconstrained. In the case when $D \leq 2d$ and the domain~$\omega$ is a disk, as we have recalled above, the spreading speed is increasing with respect to the radius of $\omega$. Moreover, when, e.g., the area is fixed, we have observed that the disk is the optimal shape. One should therefore expect that there is no optimal domain. This is consistent with out numerical simulations, as we observe that the domain grows throughout the optimization process. 

In the case when $D > 2d$, together with our ongoing assumption $f' (0) > 0 = g'(0)$, following the above discussion, it seems reasonable to conjecture that
$$\max_{\omega} c^* (\omega) =c_M= c^* ( R^*),$$
where the maximum is taken over the set of all smooth and connected domains, and the right-hand side refers to the speed when $\omega$ is a disk with optimal radius $R^*$ (see \eqref{eq:R^*_RTV}).

Our numerical observations are consistent with this conjecture. First, we find that, when the optimization process starts from a domain which is roughly smaller than the optimal disk (e.g., which is included in a disk of radius less than $R^*=6$), then the optimization algorithm converges to the optimal disk. On the other hand, if we start from a domain which is too large, then the algorithm often fails. We believe that the domain still approaches the optimal disk but does so by losing its smoothness and connectedness. In all cases, we find that the computed spreading speed never exceeds that of the optimal disk, which further supports our conjecture.

\subsubsection{The homogeneous case with reaction on both the bulk and surface}

We again restrict ourselves to the case of homogeneous coefficients and assume that $\kappa (\cdot) \equiv 1=\mu=\nu$. Now we allow not only the diffusion rates $d$ and $D$ to vary, but also the reaction terms derivatives $f'(0)$ and $g'(0)$. From the previous section, one may expect that the disks play a fundamental role in the issue of shape optimization, both with or without constraints.

We therefore start with an analytical study of the monotonicity of the speed $c^*(R)$, when the domain $\omega$ is a disk (or, in higher dimension, a ball), with respect to the radius $R$. By generalizing the methods of~\cite{RTV}, is is possible to show that there are four different regimes for the behaviour of $c^*(R)$, as illustrated by Figure~\ref{fig:regimes_cyl}. Indeed, denoting $x = \frac{D}{d}>0$ and $y= \frac{g'(0)}{f'(0)}\geq 0$, we have:
\begin{enumerate}
	\item in the (red) region $y\leq 2-x$, $(x,y)\neq(1,1)$, $c^*(R)$ is increasing;
	\item in the (blue) region $x>1/2$, $y\geq x/(2x-1)$, $(x,y)\neq(1,1)$,  $c^*(R)$ is decreasing;
	\item when $(x,y)=(1,1)$, corresponding to the black dot, the speed is equal to $2\sqrt{df'(0)}=2\sqrt{Dg'(0)}$ for all $R >0$;
	\item finally, in the remaining (orange) region, if we define
	\begin{equation*}
	c_M :=\frac{|D f'(0)-d g'(0)|}{\sqrt{(D-d)(f'(0)-g'(0))}}, \qquad
	R^*:= N \frac{\nu}{\mu} \frac{\sqrt{c_M^2-4 D g'(0)}}{\sqrt{c_M^2-4 d f'(0)}},
	\end{equation*}
	(recall that in our numerical simulations $N=2$), then $c^*(R)$ is increasing for $R<R^*$, decreasing for $R>R^*$ and attains its maximum value $c_M$ at $R=R^*$. 
\end{enumerate}
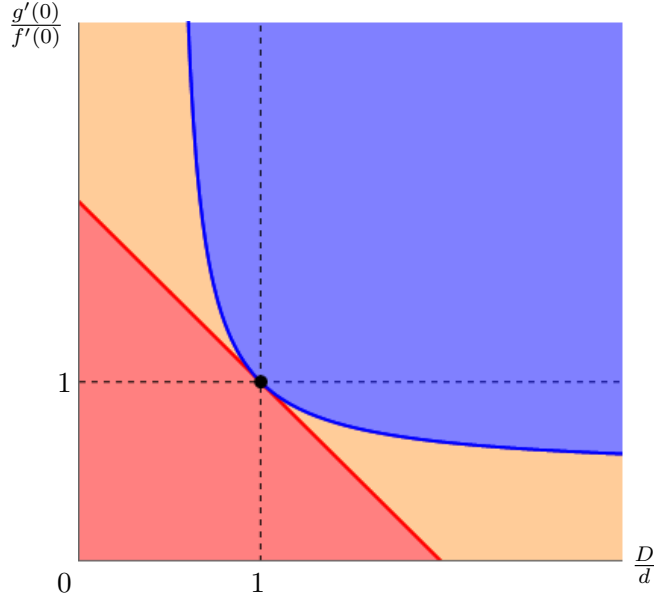
\begin{figure}[ht]
	\centering
	\begin{overpic}[width=0.45\textwidth]{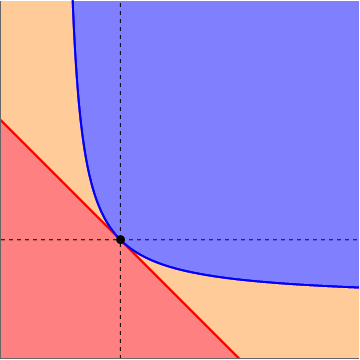}
		\put (101.5,-2) {$\frac{D}{d}$}
		\put (-13,98) {$\frac{g'(0)}{f'(0)}$}
		\put (-4,-5.5) {$0$}
		\put (31.5,-5.5) {$1$}
		\put (-4,31.5) {$1$}
	\end{overpic}
	\vspace{0.2cm}
	\caption{Regimes for the monotonicity of $c^*(R)$ in a cylinder with spherical section.} 
	\label{fig:regimes_cyl}
\end{figure}
Observe that, for $g'(0) = 0$, i.e., $y=0$, one recovers the quantities introduced in \eqref{eq:R^*_RTV}, as well as the results of~\cite{RTV} with the two regimes depending on whether $D \leq 2d$ or $D >2d$. The numerical observations made in the previous section actually extend to the whole red and orange regions, regardless of whether $g'(0)$ is zero or positive. More precisely, for any set of parameters in the red region, the disk maximizes the speed under either perimeter or area constraints. However, there seems to be no optimal shape in the unconstrained maximizing problem, due to the fact that the supremum of the speed is only approached when the size of the domain goes to infinity. On the other hand, in the orange region, the disk maximizes the speed under either perimeter or area constraints only if the constraint is `below' the one of the optimal radius $R^*$ (i.e., $\text{Perim}(\omega) \leq 2 \pi R^*$ or $\text{Area}(\omega)\leq \pi (R^*)^2$); moreover, it appears that the disk with radius $R^*$ maximizes the speed among the whole class of smooth and bounded domains.

However, if $g' (0) > f'(0)/2$, a third regime appears in the case when $y \geq x / (2x-1)$, which is equivalent to
$$d \leq d^* := D\left(  2 - \frac{f'(0)}{g' (0)} \right).$$
Notice the striking similarity between the definition of $d^*$ and the diffusion threshold~\eqref{BRR_threshold} from~\cite{BRR_influence_of_a_line}; the difference is that the roles of the surface parameters $\left(D,g'(0)\right)$ and bulk parameters $(d,f'(0))$ are reversed. It is then not so surprising that, when $d \leq d^*$, the bulk has a purely negative effect on the propagation: in this regime the speed is always less than $\lim_{R \to 0} c^* (R)$ which, as shown in Theorem~\ref{th:R0}, is $2 \sqrt{D g'(0)}$ the speed associated with the surface equation when exchange terms are removed.
\begin{figure}
	\centering
	\begin{tabular}{ccccc}
		\includegraphics[width=0.18\textwidth]{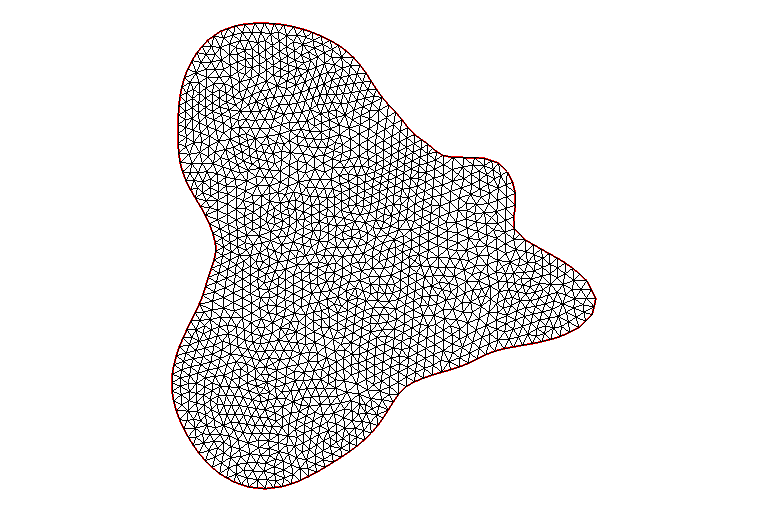} &
		\includegraphics[width=0.18\textwidth]{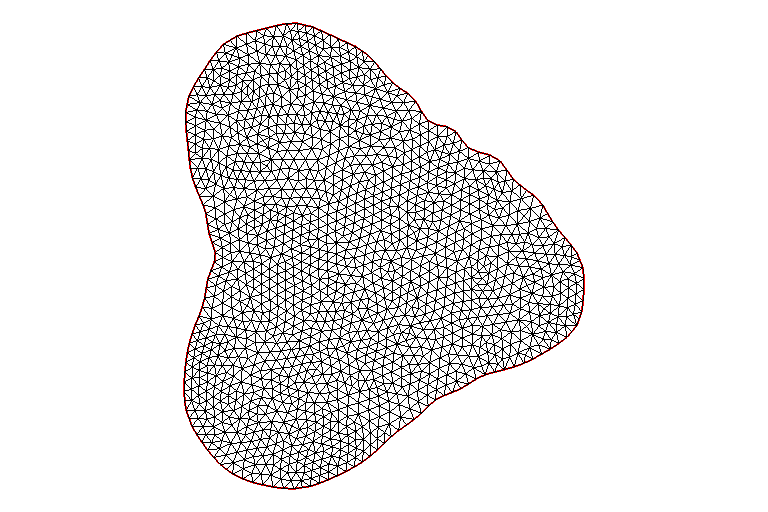}&
		\includegraphics[width=0.18\textwidth]{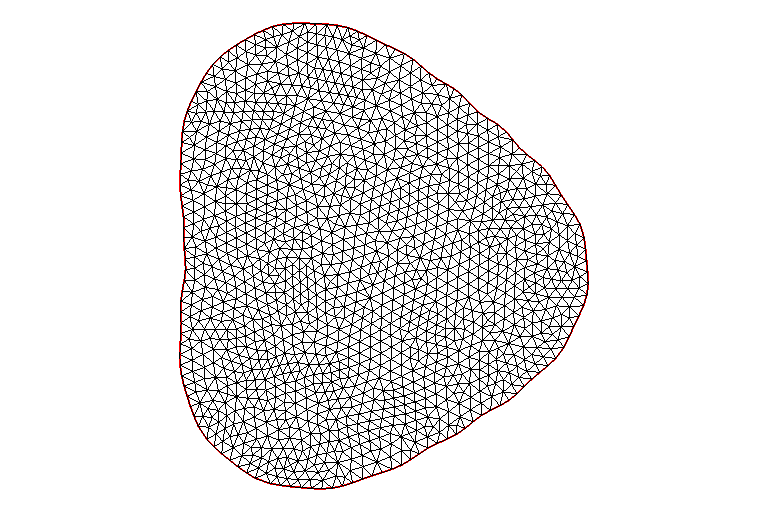}&
		\includegraphics[width=0.18\textwidth]{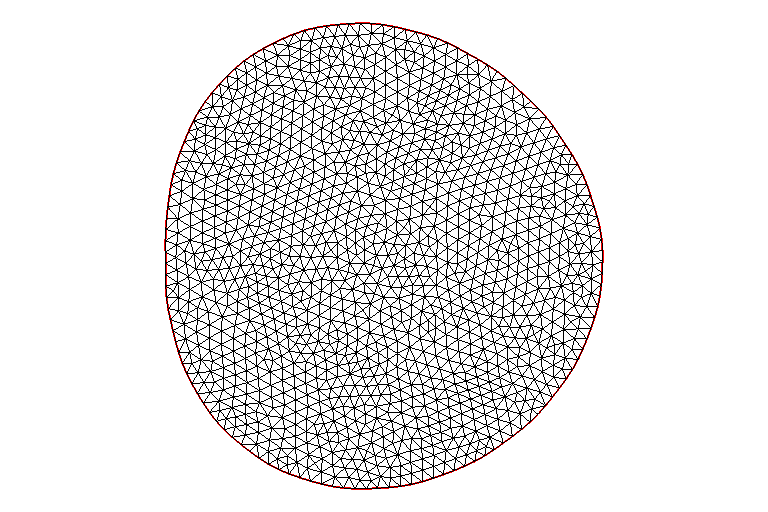}&
		\includegraphics[width=0.18\textwidth]{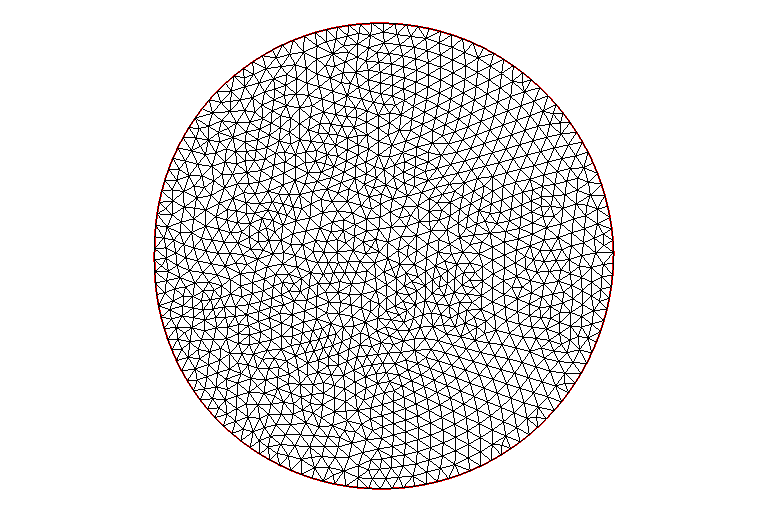}\\
		Iter 1 & Iter 5 & Iter 20 & Iter 60 & Iter 183
	\end{tabular}
	\caption{Evolution of the domain during  the minimization of the speed under the constraint $\text{Area}(\omega)=\pi$ for parameters $f'(0)=1$, $g'(0)=1.5$, $D=1.5$, $d=1$, $\kappa(\cdot)\equiv 1=\mu=\nu$.}
	\label{fig:min-area}
\end{figure}
\begin{figure}
	\centering
	\includegraphics[width=0.5\textwidth]{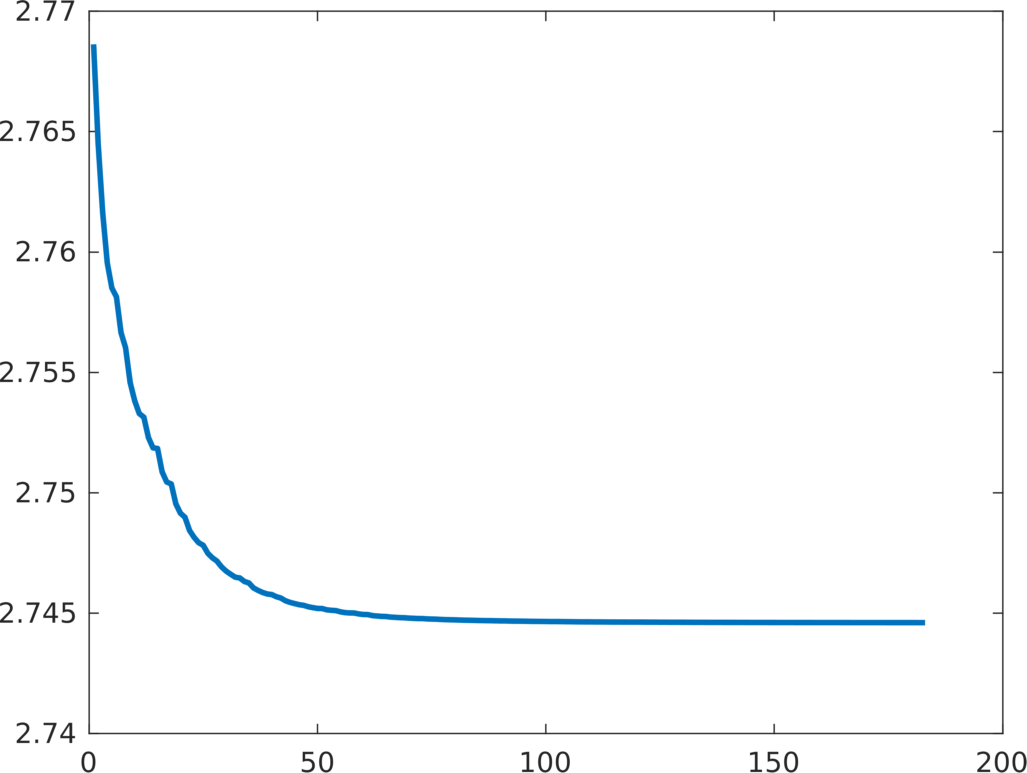}
	\caption{Evolution of the spreading speed during the minimization process illustrated in Figure~\ref{fig:min-area}. We plot the number of iteration on the horizontal axis and the corresponding value of the approximated speed on the vertical one.}
	\label{fig:min-area-cost}
\end{figure}

This new regime naturally leads to different observations on shape optimization issues, which are symmetrical to our observations in the case when $f'(0) > 0 = g'(0)$ and $D \leq 2d$. First of all, in the unconstrained case, we expect that there is no maximizing nor minimizing shape for the spreading speed. Indeed, again by generalizing the arguments of \cite{RTV}, it is possible to show that, in this regime,
$$\lim_{R \to +\infty} c^* (R) < \lim_{R \to 0} c^* (R),$$
and we conjecture that
$$\lim_{R \to +\infty} c^* (R) = \inf_\omega c^* (\omega) < \sup_\omega c^* (\omega) = \lim_{R \to 0} c^* (R),$$
where both optima are taken over the set of all smooth and connected domains, but none of them is reached. Though our numerical algorithm does not converge, this conjecture is supported by numerical observations since, in the maximizing case the domain tends to break up, while in the minimizing case it unlimitedly grows before our algorithm deteriorates.

For a similar reason, there is no maximizing shape in the constrained (area or perimeter) case either. The disk actually becomes, in this new regime, the minimizing shape for the spreading speed. In Figures \ref{fig:min-area} and \ref{fig:min-area-cost} the minimization of the spreading speed is illustrated for parameters $f'(0)=1,g'(0)=1.5,D=1.5,d=1$ and constraint $\text{Area}(\omega) = \pi$. It can be seen that the obtained numerical optimal shape resembles a disk. The same behaviour is observed for the perimeter constraint.

\subsubsection{The heterogeneous case}

We conclude this section by presenting the results of some numerics in the case of heterogeneous coefficients. Indeed, we have seen that the shape derivative formula established in Section \ref{sec:computing_shape_deriv} is also valid when keeping the dependence of $\f$ in $y$. We give below some simple optimization examples in this case. For simplicity we choose $g'(0)=0$ and $\kappa(\cdot)\equiv 1$. Note that, when $\f$ depends on $y$, the spreading speed is no longer invariant under rigid motions. It is also observed numerically that, during the maximization of the spreading speed, the domain $\omega$ tends to move towards regions where $\f$ attains its maximum. In order to be able to use the parametrization with respect to the origin, we have chosen $\f$ so that the origin belongs to the region that contains the maximal values of~$\f$.

\begin{figure}[h]
	\centering
	\begin{tabular}{cccc}
      \includegraphics[height=0.15\textwidth]{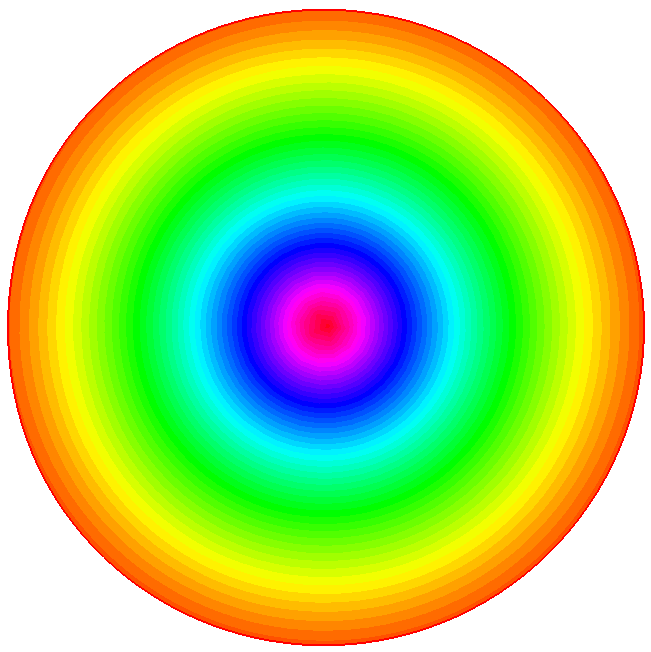}&
      \includegraphics[height=0.15\textwidth]{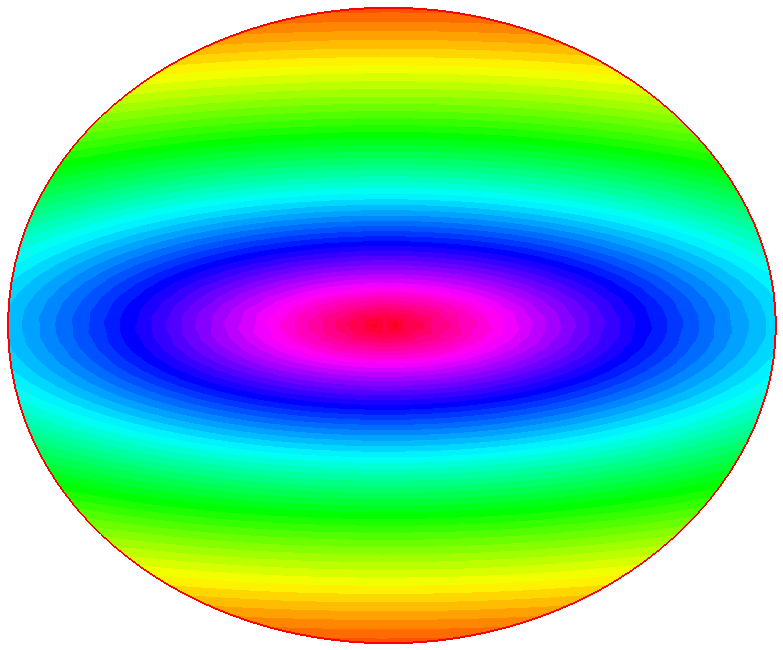}&
      \includegraphics[height=0.15\textwidth]{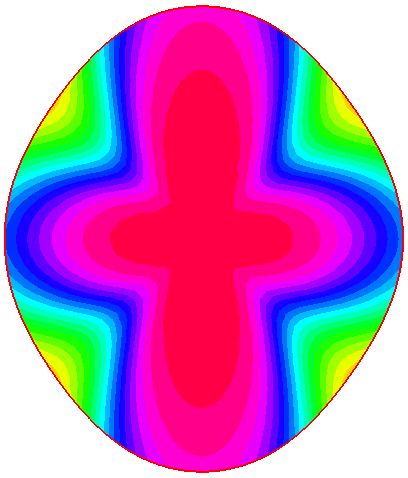}&
      \includegraphics[height=0.15\textwidth]{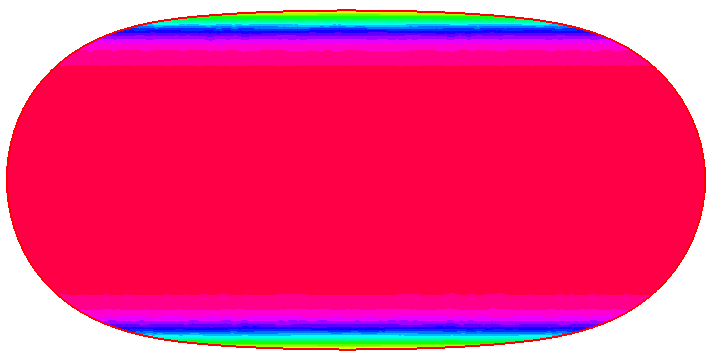}\\
      Case 1 & Case 2& Case 3 & Case 4
	\end{tabular}
	\caption{Shapes that maximize the spreading speed in the heterogeneous case for various choices of $\f$, under constraint $\text{Area}(\omega)=\pi$. The value of $\f$ is also represented, the color orange representing the lowest values and the color red the highest ones. The other parameters are as in Figure \ref{fig:max-area-cost}.}
	\label{fig:heterogeneous}
\end{figure}
As one may expect, in this heterogeneous case, the disk may no longer be optimal. Some examples, together with the associated choice of $\f$, are shown in Figure \ref{fig:heterogeneous}. The choices of $\f$ are the following: 
\begin{enumerate}
	\item $\f = 1+\exp\left(-|y|^2\right)$;
	\item $\f = 1+\exp\left(-\sqrt{0.1y_1^2+y_2^2}\right)$;
	\item $\f = \displaystyle 1+0.5\tanh\left(\frac{1+0.3\cos(4\vartheta)}{0.1+0.7\sqrt{y_1^2+0.3y_2^2}}-1\right)$, where $\vartheta$ is the angle made with the $Oy_1$ axis in polar coordinates, taken from $-\pi$ to $\pi$;
	\item $\f = 1+0.5\tanh\left(-10(y_2^2-0.4)\right)$.
\end{enumerate}
The maximization is always made using an area constraint $\text{Area}(\omega)=\pi$, with the same algorithm as in the homogeneous case. It can be seen that the optimal domain found numerically is located around the regions where $\f$ is highest. As expected, we observe that, when $\f$ is not radial, the optimal domain is no longer the disk.

\begin{center}
	\textsc{Acknowledgments}
\end{center}

B.B. was partially supported by the project ANR-18-CE40-0013 SHAPO financed by the French Agence Nationale de la
Recherche (ANR). T.G. was partially supported by the project ANR-14-CE25-0013 NONLOCAL financed by the French Agence National de la Recherche (ANR). A.T. was partially supported by: the European Research
Council, under the European Union's Seventh Framework Programme (FP/2007-2013) /
ERC Grant Agreement number 321186 - ReaDi \lq\lq Reaction-Diffusion E\-qua\-tions, Propagation and Modelling" led by Henri Berestycki; the Institute of Complex Systems of Paris \^Ile-de-France, under DIM 2014 \lq\lq Diffusion heterogeneities in different spatial dimensions and applications to ecology and medicine";
the Ministry of Science, Innovation and Universities of the Government of Spain through projects PGC2018-097104-B-100 and contract Juan de la Cierva Incorporaci\'on IJCI-2015-25084.

\bibliographystyle{siam}
\bibliography{biblio}

\begin{thebibliography}{10}

\bibitem{allaireCO}
{\sc G.~Allaire}, {\em Conception optimale de structures}, vol.~58 of
  Math\'ematiques \& Applications (Berlin) [Mathematics \& Applications],
  Springer-Verlag, Berlin, 2007.

\bibitem{AW78}
{\sc D.~Aronson and H.~Weinberger}, {\em Multidimensional nonlinear diffusion
  arising in population genetics}, Adv. in Math., 30 (1978), pp.~33--76.

\bibitem{BCRR}
{\sc H.~Berestycki, A.-C. Coulon, J.-M. Roquejoffre, and L.~Rossi}, {\em The
  effect of a line with nonlocal diffusion on {F}isher-{KPP} propagation},
  Math. Models Methods Appl. Sci., 25 (2015), pp.~2519--2562.

\bibitem{BDR19}
{\sc H.~Berestycki, R.~Ducasse, and L.~Rossi}, {\em Generalized principal
  eigenvalues for heterogeneous road–field systems}, Commun. Contemp. Math.,
  (2019), pp.~1950013, 35.

\bibitem{BDRpreprint}
\leavevmode\vrule height 2pt depth -1.6pt width 23pt, {\em Influence of a road
  on a population in an ecological niche facing climate change}, J. Math.
  Biol., 81 (2020), pp.~1059--1097.

\bibitem{BerestyckiNirenberg-TWcylinders}
{\sc H.~Berestycki and L.~Nirenberg}, {\em Travelling fronts in cylinders},
  Ann. Inst. H. Poincar{\'e} Anal. Non Lin{\'e}aire, 9 (1992), pp.~497--572.

\bibitem{BRR_plus}
{\sc H.~{Berestycki}, J.-M. {Roquejoffre}, and L.~{Rossi}}, {\em {Fisher-{KPP}
  propagation in the presence of a line: further effects}}, Nonlinearity, 26
  (2013), pp.~2623--2640.

\bibitem{BRR_influence_of_a_line}
\leavevmode\vrule height 2pt depth -1.6pt width 23pt, {\em {The influence of a
  line with fast diffusion on Fisher-KPP propagation}}, Journal of Mathematical
  Biology, 66 (2013), pp.~743--766.

\bibitem{BRR_shape_of_expansion}
{\sc H.~Berestycki, J.-M. Roquejoffre, and L.~Rossi}, {\em The shape of
  expansion induced by a line with fast diffusion in {F}isher-{KPP} equations},
  Comm. Math. Phys., 343 (2016), pp.~207--232.

\bibitem{BRR_road_fields_travelling_waves}
\leavevmode\vrule height 2pt depth -1.6pt width 23pt, {\em Travelling waves,
  spreading and extinction for {F}isher-{KPP} propagation driven by a line with
  fast diffusion}, Nonlinear Anal., 137 (2016), pp.~171--189.

\bibitem{BRR_SIR_preprint}
{\sc H.~{Berestycki}, J.-M. {Roquejoffre}, and L.~{Rossi}}, {\em Propagation of
  epidemics along lines with fast diffusion}, Bull. Math. Biol., 83 (2021),
  pp.~Paper No. 2, 34.

\bibitem{BRTpreprint}
{\sc H.~Berestycki, L.~Rossi, and A.~Tellini}, {\em Coupled reaction-diffusion
  equations on adjacent domains}.
\newblock arXiv:1903.11717.

\bibitem{Cea86}
{\sc J.~C\'ea}, {\em Conception optimale ou identification de formes: calcul
  rapide de la d\'eriv\'ee directionnelle de la fonction co\^ut}, RAIRO
  Mod\'el. Math. Anal. Num\'er., 20 (1986), pp.~371--402.

\bibitem{delfour-zolesio}
{\sc M.~C. Delfour and J.-P. Zol\'{e}sio}, {\em Shapes and geometries}, vol.~22
  of Advances in Design and Control, Society for Industrial and Applied
  Mathematics (SIAM), Philadelphia, PA, second~ed., 2011.
\newblock Metrics, analysis, differential calculus, and optimization.

\bibitem{Dietrich1}
{\sc L.~Dietrich}, {\em Existence of travelling waves for a reaction--diffusion
  system with a line of fast diffusion}, Appl. Math. Res. Express. AMRX,
  (2015), pp.~204--252.

\bibitem{Dietrich2}
\leavevmode\vrule height 2pt depth -1.6pt width 23pt, {\em Velocity enhancement
  of reaction-diffusion fronts by a line of fast diffusion}, Trans. Amer. Math.
  Soc., 369 (2017), pp.~3221--3252.

\bibitem{DiRo}
{\sc L.~Dietrich and J.-M. Roquejoffre}, {\em Front propagation directed by a
  line of fast diffusion: large diffusion and large time asymptotics}, J.
  \'{E}c. polytech. Math., 4 (2017), pp.~141--176.

\bibitem{Evans-pde}
{\sc L.~C. Evans}, {\em Partial differential equations}, vol.~19 of Graduate
  Studies in Mathematics, American Mathematical Society, Providence, RI,
  second~ed., 2010.

\bibitem{FLT}
{\sc K.~Fellner, E.~Latos, and B.~Q. Tang}, {\em Well-posedness and exponential
  equilibration of a volume-surface reaction-diffusion system with nonlinear
  boundary coupling}, Ann. Inst. H. Poincar\'{e} Anal. Non Lin\'{e}aire, 35
  (2018), pp.~643--673.

\bibitem{GMZ}
{\sc T.~Giletti, L.~Monsaingeon, and M.~Zhou}, {\em A {KPP} road-field system
  with spatially periodic exchange terms}, Nonlinear Anal., 128 (2015),
  pp.~273--302.

\bibitem{freefem}
{\sc F.~Hecht}, {\em {N}ew development in {F}ree{F}em++}, J. Numer. Math., 20
  (2012), pp.~251--265.

\bibitem{henrot-pierre-english}
{\sc A.~Henrot and M.~Pierre}, {\em Shape variation and optimization}, vol.~28
  of EMS Tracts in Mathematics, European Mathematical Society (EMS),
  Z\"{u}rich, 2018.
\newblock A geometrical analysis, English version of the French publication [
  MR2512810] with additions and updates.

\bibitem{kato-perturbation}
{\sc T.~Kato}, {\em Perturbation theory for linear operators}, Springer-Verlag,
  Berlin, 1995.
\newblock Reprint of the 1980 edition.

\bibitem{MS2018}
{\sc Y.~Morita and K.~Sakamoto}, {\em A diffusion model for cell polarization
  with interactions on the membrane}, Jpn. J. Ind. Appl. Math., 35 (2018),
  pp.~261--276.

\bibitem{Pauthier}
{\sc A.~Pauthier}, {\em The influence of nonlocal exchange terms on
  {F}isher-{KPP} propagation driven by a line of fast diffusion}, Commun. Math.
  Sci., 14 (2016), pp.~535--570.

\bibitem{RR2014}
{\sc A.~R\"{a}tz and M.~R\"{o}ger}, {\em Symmetry breaking in a bulk-surface
  reaction-diffusion model for signalling networks}, Nonlinearity, 27 (2014),
  pp.~1805--1827.

\bibitem{T16}
{\sc A.~Tellini}, {\em Propagation speed in a strip bounded by a line with
  different diffusion}, J. Differential Equations, 260 (2016), pp.~5956--5986.

\bibitem{T19}
\leavevmode\vrule height 2pt depth -1.6pt width 23pt, {\em Comparison among
  several planar {F}isher-{KPP} road-field systems}, in Contemporary research
  in elliptic {PDE}s and related topics, vol.~33 of Springer INdAM Ser.,
  Springer, Cham, 2019, pp.~481--500.

\bibitem{RTV}
{\sc A.~Tellini, L.~Rossi, and E.~Valdinoci}, {\em The effect on {F}isher-{KPP}
  propagation in a cylinder with fast diffusion on the boundary}, SIAM Journal
  on Mathematical Analysis, 49 (2017), pp.~4595--4624.

\end{thebibliography}

\end{document}